\documentclass[11pt,american]{amsart}
\usepackage{amsmath,amsfonts,amssymb,amscd,amsthm,amsbsy,bbm,epsf,calc,graphicx,esint}
\usepackage{mathtools}
\usepackage{color}
\usepackage{datetime}
\usepackage{comment}
\usepackage[colorlinks,linkcolor = blue,citecolor = blue]{hyperref}
\allowdisplaybreaks

\usepackage[normalem]{ulem}

\numberwithin{equation}{section}

\newtheorem{theorem}{Theorem}[section]
\newtheorem{proposition}[theorem]{Proposition}
\newtheorem{lemma}[theorem]{Lemma}

\newtheorem{remark}{Remark}
\newtheorem{corollary}[theorem]{Corollary}

\newcommand{\N}{\mathbb N}
\newcommand{\R}{\mathbb R}

\newcommand{\eps}{\varepsilon}

\newcommand{\vv}{\langle v\rangle}

\DeclareMathOperator{\weak}{\rightharpoonup}
\DeclareMathOperator{\embeds}{\hookrightarrow}

\newcommand{\diff}{\tilde\partial}

\newcommand{\red}[1]{\textcolor{red}{#1}}

\newcommand{\set}[1]{\left\{#1\right\}}
\newcommand{\brak}[1]{\langle #1 \rangle} 

\let\pa\partial
\newcommand{\cM}{\mathcal{M}}
\newcommand{\cK}{\mathcal{K}}
\newcommand{\uml}{\"}

\numberwithin{equation}{section}

\title[Existence of smooth solutions]{Existence of smooth solutions to the Landau-Fermi-Dirac equation with Coulomb potential}

\author{William Golding}
\address{Department of Mathematics, The University of Texas at Austin, 2515 Speedway, Austin TX, 78712}
\email{wgolding@utexas.edu} 
\thanks{WG is partially supported by the NSF-DMS grant 1840314}
\author{Maria Pia Gualdani}
\address{Department of Mathematics, The University of Texas at Austin, 2515 Speedway, Austin TX, 78712}
\email{gualdani@math.utexas.edu}
\thanks{MG is partially supported by the DMS-NSF 2019335 and would like to thank NCTS Mathematical Division of Taipei for their kind hospitality.}
\author{Nicola Zamponi}
\address{Vienna University of Technology,
Wiedner Hauptstra\ss e 8-10, 1040 Vienna (Austria)}
\email{nicola.zamponi@asc.tuwien.ac.at}
\thanks{NZ acknowledges support from the
Alexander von Humboldt Foundation (AvH)
and from the Austrian Science Foundation (FWF), grants P30000, P33010.}
\thanks{Key words: Landau-Fermi-Dirac equation, existence and uniqueness, regularity, coercivity, dissipation, long-time behavior, algebraic decay, H-theorem. 2020 Mathematics Subject Classification: 35BXX, 35K59, 35K55, 35P15, 35Q84, 82C40, 82D10}
\begin{document}
\begin{abstract}
In this paper, we prove global-in-time existence and uniqueness of smooth solutions to the homogeneous Landau-Fermi-Dirac equation with Coulomb potential. The initial conditions are nonnegative, bounded and integrable. We also show that any weak solution converges towards the steady state given by the Fermi-Dirac statistics. Furthermore, the convergence is algebraic, provided that the initial datum is close to the steady state in a suitable weighted Lebesgue norm. 
\end{abstract}
\maketitle

\section{Introduction}
We consider the homogeneous Landau-Fermi-Dirac equation {with Coulomb potential}
\begin{align}\label{LFD}
\partial_t f = \frac{1}{8\pi} \textrm{div}_v \int_{\R^3}\frac{\Pi(v-v_*)}{|v-v_*|}\left[f(v_*)(1-\varepsilon f(v_*))\nabla_v f - f(v)(1-\varepsilon f(v))\nabla_{v_*} f(v_*)\right] \;dv_*,
\end{align}
where $\Pi(z)$ is the standard projection matrix 
$$
\Pi(z) = Id - \frac{z\otimes z}{|z|^2}. 
$$
The function $f(v,t)$ models the distribution of velocities within a single species quantum gas. The particles considered here are fermions (e.g. electrons) {interacting} in a grazing collision regime \cite{alexandre2004landau}. The parameter $\varepsilon$ quantifies the strength of the quantum effects of the system for the particular species considered and depends on Planck's constant, the mass of the species, and the number of independent quantum weights of the species. In particular, {we notice that in the case $\varepsilon = 0$ eq.~\eqref{LFD} reduces to} the classical Landau equation. The Pauli exclusion principle implies  that $f$ satisfies the a priori bound
$$
0\le f \le \frac{1}{\varepsilon}.
$$  This bound is the key ingredient in our proof. See also \cite{Dolb94} for a discussion on the Boltzmann equation with Fermi-Dirac statistic.

Equation (\ref{LFD}) is well-understood for the cases of moderately soft and hard potentials, namely when the kernel $\frac{1}{|v-v^*|}$ is replaced by $\frac{1}{|v-v^*|^{-\gamma-2}}$ for $\gamma \ge-2$. In \cite{ABDL_21_III}, the authors consider the moderately soft potentials case ($-2\le \gamma \le 0$) and show algebraic convergence of non degenerate solutions 
towards equilibrium for initial data 
satisfying a suitable non saturation condition. Existence and uniqueness of weak solution for hard potentials ($\gamma \ge 0$) are shown in \cite{B04}, regularity and smoothing effects are studied in \cite{Chen1, Chen2}, and exponential convergence towards equilibrium in \cite{ABDL_21_II}. In \cite{alonso2021use}, the authors present fundamental properties of the entropy and entropy production functional for hard and moderately soft potentials. The existence of nondegenerate steady for any potential is shown in \cite{BL04}.

The Landau-Fermi-Dirac equation shares several properties with the classical Landau equation. 
Multiplying \eqref{LFD} {by} a test function $\phi$, integrating by parts, {and applying a straightforward symmetry argument}, one obtains
\begin{align}\label{weak.LFD}
\int_{\R^3} &\partial_t f  \phi \;dv 
= - \frac{1}{16\pi} \int_{\R^3}\int_{\R^3}    \frac{\Pi(v-v_*)}{|v-v_*|}[  f(1-\varepsilon f)\nabla_{v_*} f-f_*(1-\varepsilon f_*)\nabla_v f][\nabla\phi_*-\nabla\phi]\;dv_*dv.
\end{align}
Conservation of mass, momentum and energy conservation follows from \eqref{weak.LFD} by choosing $\phi(v) \in \{ 1, v, |v|^2 \}$.
%
A version of the H-theorem for \eqref{LFD} is also available: with
$$
\phi = \ln \left( \frac{\varepsilon f}{1-\varepsilon f}\right)
$$
in \eqref{weak.LFD}, one obtains  that
\begin{align}
\frac{d}{dt}H_\varepsilon[f](t) = - \frac{1}{16\pi} \int_{\R^3}\int_{\R^3}  & f(1-\varepsilon f)f_*(1-\varepsilon f_*)  \cdot  \nonumber  \\
 & \cdot \frac{\Pi(v-v_*)}{|v-v_*|}\left[ \frac{ \nabla_{v_*} f_*}{f_*(1-\varepsilon f_*)} -\frac{\nabla_v f}{f(1-\varepsilon f)}\right]^2\;dv_*dv \le 0,\label{H-thm}
\end{align}
where 
$$
H_\varepsilon[f] := \frac{1}{\varepsilon}  \int_{\R^3} \varepsilon f \ln (\varepsilon f) + ( 1-\varepsilon f)\ln(1-\varepsilon f) \;dv.
$$
{Eq.~\eqref{H-thm} is the {\em entropy balance equation} associated to \eqref{LFD}, with $-H_\varepsilon$ being the (physical) Fermi-Dirac entropy functional. 
The only smooth function that nullifies the entropy production, $\frac{d}{dt}H_\varepsilon[f] =0$, is the {\em Fermi-Dirac equilibrium distribution} 
\begin{align}\label{Fermi-Dirac}
\cM_\eps(v) := \frac{ae^{-b|v-u|^2}}{1+\varepsilon ae^{-b|v-u|^2}},
\end{align}
which is also the {\em{only smooth minimizer}} of $H_\varepsilon$ under the constraints of given mass, momentum, and energy \cite{BL04}. 
The constants  $a>0$, $b>0$, and $u\in\R^3$ are determined by the mass, first and second moment of the initial data
\begin{align*}
\int_{\R^3}\begin{pmatrix}
1 \\ v \\ |v|^2
\end{pmatrix}\cM_\eps(v)dv =
\int_{\R^3}\begin{pmatrix}
1 \\ v \\ |v|^2
\end{pmatrix}f(v,t)dv =
\int_{\R^3}\begin{pmatrix}
1 \\ v \\ |v|^2
\end{pmatrix}f(0,v)dv.
\end{align*}
There are other nonsmooth distributions of the form 
$$
\mathcal{F}_\varepsilon(v) := {\varepsilon^{-1}} \chi_{\Omega},
$$
with $\Omega$ of $\R^3$ a measurable subset,
that satisfies (formally) $H_\varepsilon[\mathcal{F}_\varepsilon]  = \frac{d}{dt}H_\varepsilon[\mathcal{F}_\varepsilon] =0$ and solves (\ref{LFD}). These particular stationary solutions are called {\em saturated Fermi-Dirac states}. As such, any solution to (\ref{LFD}) with general initial data could approach, as time grows, such saturated states. However, given an initial data with mass $\rho$, momentum $u$ and energy $E$, there exists only one value of $\varepsilon $, uniquely determined by $\rho$, $u$ and $E$, for which $\mathcal{F}_\varepsilon(v) $ is an admissible stationary solution. For $\varepsilon$ below such value, the only steady-state is $ \cM_\eps$.  

{Taking the formal limit $\varepsilon\to 0$ in eq.~\eqref{LFD}, one obtains the classical Landau equation. Furthermore, $H_\varepsilon[f]\to H[f] = \int_{\R^3} f\ln f \; dv$ as $\varepsilon\to 0$ modulus a multiple of the mass $\int_{\R^3}f\; dv$:
\begin{align*}
H_\varepsilon[f] - (\ln\varepsilon - 1) \int_{\R^3}f \; dv \to H[f]\quad
\mbox{as }\varepsilon\to 0.
\end{align*}
 The addition  of a multiple of the mass to $H_\varepsilon[f]$ does not change the entropy balance equation \eqref{H-thm} nor the form of the equilibrium distribution \eqref{Fermi-Dirac}, thanks to the conservation of mass property. 
The equilibrium distribution $\cM_\eps(v)$ also converges towards the classical Maxwellian distribution $M(v) = a e^{-b|v-u|^2}$ as $\eps\to 0$. Finally, strictly related to the limits $H_\varepsilon[f]\to H[f]$ and $\cM_\eps\to M$ is the fact that} the relative entropy 
\begin{align*}
H_\varepsilon[f|\cM_\eps] := & \int_{\R^3} \cM_\eps\left[ \frac{f}{\cM_\eps} \ln \left(\frac{f}{\cM_\eps} \right)-\frac{f}{\cM_\eps} +1\right] \\
&  +   \frac{1}{\varepsilon}\int_{\R^3}  (1-\varepsilon \cM_\eps)\left[  \frac{1-\varepsilon f}{1-\varepsilon \cM_\eps} \ln \left(\frac{1-\varepsilon f}{1-\varepsilon \cM_\eps} \right)-\frac{1-\varepsilon f}{1-\varepsilon \cM_\eps} +1\right]\;dv
\end{align*}
converges to the relative entropy of the classical Landau equation 
\begin{align*}
H[f|M] := & \int_{\R^3} M\left[ \frac{f}{M} \ln \left(\frac{f}{M} \right)-\frac{f}{M} +1\right] \;dv.
\end{align*}
The next observation concerns the structure of the collision operator. For a smooth $f$, the interaction term can be expressed as a second order elliptic nonlinear operator with non-local coefficients: 
$$
\textrm{div}_v \left( A[f(1-\varepsilon f)]\nabla f - f(1-\varepsilon f) \nabla a[f]\right),
$$
where the matrix $A[f(1-\varepsilon f)] $ is defined through the map 
$$
A: g \mapsto A[g],$$
with  
\begin{equation}\label{defn_coef}
A[g] := \frac{1}{8\pi}\int_{\R^3}\frac{\Pi(v-v_*)}{|v-v_*|}g(v_*) \;dv_*, \quad a[f] := \frac{1}{4\pi}\int_{\R^3}\frac{f(v_*)}{|v-v_*|} \;dv_*.
\end{equation}

\subsection{Main Results} Our first result concerns existence of smooth solutions to (\ref{LFD}). Unlike in the case of the classical Landau equation, we are able to show global-in-time existence of smooth solutions for a general class of initial datum. Our regularity estimates depend on the quantum parameter. At the present moment, it seems out of reach to obtain similar results uniformly with respect to $\varepsilon$. Therefore, in the rest of the manuscript we set $\eps = 1$.

\begin{theorem}\label{thm_existence}
Suppose $f_{in}:\R^3 \rightarrow \R$ satisfies $0 \le f_{in} \le 1$, $(1 + |v|^3)f_{in}\in L^1(\R^3)$, and $H_1(f_{in}) < 0$. Then, there is a solution $f:[0,\infty)\times\R^3 \rightarrow \R$ with $f\in C([0,\infty);L^2(\R^3))$ such that $f(0) = f_{in}$, $0 \le f \le 1$, $f\in L^\infty([0,\infty);L^p(\R^3)) \cap L^2([0,T];H^1(\R^3))$ for each $1\le p \le \infty$, and for each $T>0$,
and $\varphi \in L^2([0,T];H^1(\R^3))$,
\begin{equation}\label{weak_FDL}
\int_{0}^T  \langle \varphi ,\partial_t f\rangle_{H^1,H^{-1}} \;dt = - \int_0^T\int_{\R^3} \left(A[f(1-f)]\nabla f - \nabla a[f]f(1-f)\right) \cdot \nabla \varphi \;dvdt.
\end{equation}

Moreover, $f$ has decreasing (Fermi-Dirac) entropy and satisfies conservation of mass, energy, and momentum. 

If the initial data has moments $(1 + |v|^m)f_{in}\in L^1(\R^3)$ with $m>9$, the solution is unique. 
\end{theorem}

By a simple time rescaling, we obtain global-in-time existence and uniqueness for any quantum parameter:
\begin{corollary}
Fix $\eps > 0$ and let $f_{in}:\R^3 \rightarrow \R$ satisfies $0\le f_{in} \le  \eps^{-1}$, $(1+|v|^3)f_{in} \in L^1(\R^3)$, and $H_{\eps}(f_{in}) < 0$. Then, there is a unique $f:[0,\infty)\times\R^3 \rightarrow \R$ with $f\in C([0,\infty);L^2(\R^3))$ such that $f(0) = f_{in}$, $0 \le f \le \eps^{-1}$, $f\in L^\infty([0,\infty);L^p(\R^3)\cap  L^2([0,T];H^1(\R^3))$ for each $1\le p \le \infty$, and for each $T>0$,
and $\varphi \in L^2([0,T];H^1(\R^3))$,
\begin{equation*}
\int_{0}^T  \langle \varphi ,\partial_t f\rangle_{H^1,H^{-1}} \;dt= - \int_0^T\int_{\R^3} \left(A[f(1-\eps f)]\nabla f - \nabla a[f]f(1-\eps f)\right) \cdot \nabla \varphi \;dvdt.
\end{equation*}
Moreover, $f$ has decreasing (Fermi-Dirac) entropy and satisfies conservation of mass, energy, and momentum.
\end{corollary}

{Theorem \ref{thm_existence} is proved in several steps. First, we approximate the problem by discretizing the time variable and adding suitable regularizing terms. The approximating problem is well-posed thanks to suitable fixed point arguments. After, we use uniform $L^2$ and  entropy inequalities to take limits as our regularizing terms vanish. A crucial ingredient is the uniform positive lower bound for the diffusion matrix $A[f(1-f)]$, which follows from the boundedness of the second moment of $f$ and a uniform negative upper bound for the Fermi-Dirac entropy. This guarantees that equation \eqref{LFD} remains uniformly parabolic during the evolution of the system.}

The weak solutions from Theorem \ref{thm_existence} are, in fact, smooth solutions, provided the initial data {has} high enough moments: 
\begin{theorem}\label{thm_regularity}
Let $f$ be a weak solutions as in Theorem \ref{thm_existence}. If the initial data $f_{in}$ is, in addition, such that {$(1+|v|^{12}) \in L^1(\R^3)$ then $f\in C^\infty((0,T];C^\infty(\R^3))$.}
\end{theorem}

{The higher regularity of the solution is obtained thanks to parabolic regularity arguments, Morrey's inequality and Schauder estimates. The parabolic regularity argument yields estimates for $f$ in $W^{1,\infty}(0,T; W^{-1,p})\cap L^\infty(0,T; W^{1,p}))$  for any $p\in [2,6]$.  Via interpolation between Sobolev spaces, we obtain a bound for $f$ in a fractional Sobolev space. From here, we deduce, the H\uml older continuity of $f$ via Morrey's inequality. A standard parabolic bootstrap argument yields $f\in C^\infty((0,T];C^\infty(\R^3))$.}

Our regularity results do not hold in the limit $\varepsilon \to 0$, since they heavily rely on the bound $f\le \frac{1}{\varepsilon}$. For the classical Landau equation,  the Cauchy problem has been understood only for weak solutions \cite{V98} \cite{Desvillettes14} \cite{AP77}  \cite{alexandre2004landau} \cite{GZ18}. Recently,  in  \cite{GuGu16} and  \cite{Silvestre2015}, the authors showed that, for a short time, weak solutions become instantaneously regular and smooth. The long time asymptotic for  weak solutions has been studied in \cite{CM2017verysoft} and \cite{CDH15}.  However, the question of whether solutions stay smooth for all time or become unbounded after a finite time is still open. Recent research has produced several conditional results regarding this inquiry. These results show regularity properties of solutions that already possess some basic properties (yet to be verified). They include (i) conditional uniqueness \cite{Fournier2010} \cite{ChGu20}, and (ii) conditional smoothness for solutions in  $L^\infty(0,T,L^p(\mathbb{R}^d))$ with $p>\frac{d}{2}$ \cite{Silvestre2015} \cite{GuGu16}.  In a very recent manuscript \cite{Desvillettes-He-Jiang-2021}, the authors studied behavior of solutions in the space $L^\infty(0,T,\dot{H}^1(\mathbb{R}{^3}))$. They show that for general initial data there exists a time $T^*$ after which the weak solutions belong to $L^\infty((T^*, +\infty), {H^1}(\mathbb{R}{^3}))$. This result agrees with the one  in \cite{GGIV2019partial}, in which the authors  showed that the set of singular times for weak solutions has Hausdorff dimension at most $\frac{1}{2}$. In \cite{BGS22}, the authors show that self-similar blow-up of type I cannot occur for solutions to the Landau equation.

The second  result of this paper concerns the convergence towards the steady state as the time approaches infinity.  We show that the convergence is algebraic, provided that the initial datum $f_{in}$ is close to the steady state $\cM_\eps$ in a suitable weighted Lebesgue norm. Hereafter, we denote with $\cM$ the function defined in (\ref{Fermi-Dirac}) with $\varepsilon=1$.
\begin{theorem}\label{thm_longtime}
Given any initial datum $f_{in} : \R^3\to [0,1]$, $f_{in}\in L^1_2$, such that $H_1[f_{in}] < 0$, the solution $f$ to the initial value problem associated to \eqref{LFD} converges strongly in $L^1$ as $t\to\infty$ to the Fermi-Dirac distribution $\cM$ with same mass, momentum and energy as $f_{in}$.

Furthermore, there exists a constant $\ell>0$
such that, if $$\int_{\R^3}(f_{in} - \cM)^2 \cM^{-1}(1-\cM)^{-1}dv < \ell,$$
\begin{align*}
\int_{\R^3}(f_{in} - \cM)^2 \cM^{-1}(1-\cM)^{-1} (1+|v|^2)^{N/2} dv < \infty, \quad \textrm{for some} \; N\geq 1,
\end{align*}
 then 
\begin{align*}
\int_{\R^3}\frac{(f(t) - \cM)^2}{\cM(1-\cM)} dv \lesssim (1+t)^{-N},\qquad t>0.
\end{align*}
\end{theorem}

{
The unconditional convergence (without rate) towards the steady state is obtained from the entropy balance equation in the following way. We integrate the balance equation in time and use the ellipticity properties of the entropy dissipation to deduce that $f(t_n)\to\cM$ along a suitable sequence of time instants $t_n\to\infty$. The monotonicity in time of the relative entropy yields that $f(t)\to\cM$ strongly in $L^1$ as $t\to\infty$. 

The algebraic convergence for initial data close to the steady state is achieved by linearizing \eqref{LFD} around $\cM$. First, we show existence of a spectral gap for the linearized Landau-Fermi-Dirac operator between two different weighted Lebesgue space. Precisely, such relation has the structure
\begin{align*}
-(Lh,h)_{E_1} \geq \lambda \|h\|_{E_2},\qquad h\in D(L)\cap N(L)^\perp,
\end{align*}
with $E_2$ not included in $E_1$. This latter fact is the reason why we are not able to obtain exponential convergence towards equilibrium, but only algebraic.
After, we derive a uniform bound for some moment of the solution to the linearized equation in a weighted Lebesgue space.
In the last step, we bound the contributions of the nonlinear corrections, and derive a differential inequality for the weighted $L^2$-norm of the perturbation 
$$h := \frac{f-\cM}{\cM(1-\cM)}.$$ 
An elementary argument of ordinary differential equations' theory yields algebraic convergence to zero with rate $N$ for $\|h\|_{L^2(m)}$, provided that, at initial time, the latter is small enough and $\|h\|_{L^2(m (1+|v|^2)^{N/2})}<\infty$.
}

\subsection{Notations}

Here we list some of the notation conventions adopted throughout the manuscript:
\begin{itemize}
    \item Universal constants that may change from line to line are denoted $C$ or $C(A,B)$ if the constant is allowed to depend on the quantities $A$ and $B$.
    
    \item We write $A \lesssim B$ to mean there is a universal constant $C$ such that $A \le CB$. Similarly, we write $A \sim B$ to mean $A \lesssim B$ and $B \lesssim A$. If we write $A \lesssim_\Lambda B$, the implicit constant $C$ is allowed to depend on $\Lambda$.
    
    \item We write $L^p([0,T];X)$ for $T > 0$ and $X$ a Banach space to denote the space of strongly measurable $X$-valued functions satisfying $$\int_0^T \|f(t)\|_X^p \ dt < \infty.$$
    When we write $L^p$ without specifying the measure space, we mean $L^p(\R^3)$.
    
    \item We use the japanese bracket notation $\brak{v} := (1 + |v|^2)^{1/2}$. Given $p\in [1,\infty]$, we denote with $L^p_m$ the space of functions that have the following norm
    \begin{equation*}
        \|f\|_{L^p_m}^p := \int_{\R^3} |f|^p\brak{v}^{m} \ dv,
    \end{equation*}
    finite. We denote with $\|f\|_{L^p}$ the $L^p(\R^3)$ norm of $f$.

	\item Given $p\in [1,\infty]$ we denote with $p^*\in [1,\infty]$ the conjugate exponent of $p$,  $p^* := \frac{p}{p-1}$. 

\end{itemize}

{In Section \ref{sec_coeff}, we recall some useful estimates for the coefficients $A[f]$, $a[f]$ appearing in \eqref{LFD}. In Section \ref{sec_ex}, we prove Theorem ~\ref{thm_existence}. In Section \ref{sec_reg}, we prove Theorem ~\ref{thm_regularity}. In Section \ref{sec_ltb}, we prove Theorem ~\ref{thm_longtime}.}

\section{Coefficient Bounds}\label{sec_coeff}
The following standard bounds will be used throughout our proofs. 
\begin{lemma}\label{lemA_lower}
Any $f(v)$ such that $0\le f(v) \le \frac{1}{\varepsilon}$, $\int_{\R^3} f (1+|v|^2) = E_0$, and $H_\varepsilon [f]\le H_0 <0$ satisfies
$$
\langle A[f(1-\varepsilon f)] \xi, \xi \rangle \ge \frac{C}{1+ |v|^3} |\xi|^2,  \quad \forall \xi \in \R^3,
$$
where $C$ depends on $E_0$ and $H_0$. 
\end{lemma}
\begin{proof}
We begin by quoting a known result (see \cite{DesVil2000a} Lemma 6 and Proposition 4, or  \cite{Silvestre2015} Lemma 3.2 and 3.3) that says that for any nonnegative function $\varphi$ with mass, second momentum and entropy bounded, for all $v\in \R^3$:
$$
\int_{\R^3}\frac{\Pi(v-v_*)}{|v-v_*|}\varphi(v_*) \;dv_* \ge \frac{C}{1+ |v|^3} \mathbb{I},
$$
where the constant $C$ depends only on the quantities 
$$
\int \varphi(v)\;dv, \quad \int \varphi(v)|v|^2\;dv,\quad \int \varphi(v) |\ln\varphi(v)| \;dv.
$$
In light of this inequality, we need only show that 
\begin{align}\label{LlogL}
\int f(1-\varepsilon f) |\ln f(1-\varepsilon f)| \;dv < +\infty,
\end{align}
and that there exists a strictly positive constant $m_0$ such that 
\begin{align}\label{M_bound_below}
\int f(1-\varepsilon f) \;dv \ge m_0.
\end{align}
The proof of (\ref{LlogL}) and (\ref{M_bound_below}) can be found in \cite{B04} in Lemma 3.1. 
\end{proof}

The previous lemma together with (\ref{H-thm}) show that, as long as the initial data has strictly negative entropy, our equation is {\em{uniformly}} {parabolic}, and saturated-Fermi-Dirac-distributions are not admissible solutions.

\begin{lemma}[Upper Bound on \texorpdfstring{$A[f]$}{}]\label{lemA_upper}
For $A[f]$ defined in \eqref{defn_coef}, and for any $f\in L^p \cap L^q$ with $1\le p < 3/2 < q \le \infty$, we have
\begin{equation}
\|A[f]\|_{L^\infty} \le C(p,q)\|f\|_{L^q}^{1-\alpha}\|f\|_{L^p}^\alpha,
\end{equation}
where $\alpha = \frac{\frac{1}{3}-\frac{1}{p^*}}{\frac{1}{q^*} - \frac{1}{p^*}} \in (0,1)$. Furthermore, $\nabla \cdot A[f] = \nabla a[f]$.
\end{lemma}

\begin{proof}
For $R > 0$ arbitrary,
\begin{equation*}
\begin{aligned}
|A[f]| &\le \int_{|x - y| \le R} \frac{|f(y)|}{|x - y|} \ dy + \int_{|x-y| \ge R} \frac{|f(y)|}{|x-y|} \ dy\\
	&\le \|f\|_{L^q}\left(\int_{|x - y| \le R} \frac{1}{|x - y|^{q^*}} \ dy\right)^{1/q^*} + \|f\|_{L^p}\left(\int_{|x - y| \ge R} \frac{1}{|x - y|^{p^*}} \ dy\right)^{1/p^*}\\
	&\lesssim_{p,q} \|f\|_{L^q}R^{{3/q^*} - 1} + \|f\|_{L^p}R^{3/p^* - 1},
\end{aligned}
\end{equation*}
provided $q^*<3$ and $p^*>3$. Optimizing in $R$ yields $R \approx \|f\|_{L^q}^{-\beta}\|f\|_{L^p}^{\beta}$ and the bound,
\begin{equation*}
\|f\|_{L^q}R^{\frac{3(q-1)}{q} - 1} + \|f\|_{L^p}R^{\frac{3(p-1)}{p} - 1} \lesssim \|f\|_{L^q}^\alpha\|f\|_{L^p}^{1-\alpha},
\end{equation*}
for $\beta = \frac{1}{3\left(\frac{1}{q^*} - \frac{1}{p^*}\right)}$,  and $\alpha = \beta\left(1-\frac{3}{p^*}\right) > 0$. Note that $0 < \alpha < 1$. 

Finally, notice that 
\begin{align*}
\textrm{div} A[f] = -\frac{1}{4\pi} \int_{\mathbb{R}^3} \frac{x-y}{|x-y|^3} f(y)\;dy = \frac{1}{4\pi} \nabla  \int_{\mathbb{R}^3}  \frac{f(y)}{|x-y|} f(y)\;dy = \nabla a[f].
\end{align*}
\end{proof}

\begin{lemma}[Upper Bound on \texorpdfstring{$\nabla a[f]$}{}]\label{lema_upper}
For $a[f]$ defined in \eqref{defn_coef}, we have 
\begin{equation*}
\|\nabla a[f]\|_{L^2} \le C \|f\|_{L^{6/5}},
\end{equation*}
and
\begin{equation*}
\|\nabla a[f]\|_{L^\infty} \le C(p,q)\|f\|_{L^p}^\alpha\|f\|_{L^q}^{1-\alpha} ,
\end{equation*}
for any $1 \le p < 3 < q \le \infty$, and some $\alpha  \in (0,1)$.
\end{lemma}

\begin{proof}
The Hardy-Littlewood-Sobolev inequality (in $\R^3$) states that
$$\left\|\frac{1}{|x|^\lambda}\ast f\right\|_{L^q} \lesssim_{\alpha,p,q} \|f\|_{L^p}$$
provided $1 < p,q,\frac{3}{\lambda} < \infty$ and $\frac{1}{p} + \frac{\lambda}{3} = 1 + \frac{1}{q}$ (see \cite{L83}). The kernel $K(x)$ for $\nabla a$ satisfies $K(x) \sim |x|^{-2}$ and the $L^2$ estimate follows immediately. The $L^\infty$-bound follows the same steps as in Lemma \ref{lemA_upper}. 
\end{proof}

\section{Existence of bounded weak solutions}\label{sec_ex}

In order to find weak solutions to \eqref{LFD}, we first introduce an extra dissipative term $\delta_1 \Delta f$ 
to counter the degenerate ellipticity of $A[f(1-f)]$ (see Lemma \ref{lemA_lower}) and study the approximating problems
\begin{equation}\label{FDL_approx}
\partial_t f = \nabla \cdot \left(A[f(1-f)]\nabla f - \nabla a[f]f(1-f)\right) + \delta_1\Delta f. 
\end{equation}
We will first prove there exist solutions to \eqref{FDL_approx}, then taking $\delta_1$,
we recover global-in-time weak solutions to \eqref{LFD}. To this end, we introduce an auxiliary equation,
\begin{equation}\label{linear}
\begin{aligned}
\frac{\left(f_k - f_{k-1}\right)}{\tau} = \nabla \cdot \left(A_{k-1}\nabla f_k - \nabla a_{k-1}z_+(1-z)_+ + \delta_1\nabla f_k\right) - \delta_2 |v|^m f_k,\\
A_{k-1} := A[f_{k-1}(1-f_{k-1})] \qquad \text{and} \qquad a_{k-1} := a[f_{k-1}],
\end{aligned}
\end{equation}
obtained by dividing the time interval $[0,T]$ into $N$ subintervals, each of length $\tau$, linearizing (\ref{FDL_approx}) around a measurable function $z$, and adding an additional localizing term, $\delta_2|v|^m f$. In the first step of our construction, we use the Lax-Milgram Theorem to find unique weak solutions to \eqref{linear} and prove the following proposition: 
\begin{proposition}\label{prop_step1}
Let $f_{k-1} \in L^1$ with $0 \le f_{k-1} \le 1$, $z$ be a measurable function, and $m \ge 0$. Then, there is a unique $f_k \in H^1 \cap L^2_m$ that satisfies
\begin{equation}\label{weak_linear}
\begin{aligned}
\int\frac{\left(f_k - f_{k-1}\right)}{\tau} \varphi - \nabla a_{k-1}z_+(1-z)_+ \cdot\nabla\varphi \;dv= &- \int\nabla\varphi \cdot A_{k-1}\nabla f_k\;dv\\
&-\delta_1\int\nabla \varphi \cdot \nabla f_k-\delta_2\int |v|^m\varphi f_k\;dv,
\end{aligned}
\end{equation}
 for any $\varphi \in H^1 \cap L^2_m$.
\end{proposition}
For a fixed $k$, Proposition \ref{prop_step1} defines a solution operator $\Phi$ to \eqref{linear} via $\Phi(z) = f_k$. In the second step of our construction, we seek solutions $f_k$ to the nonlinear system:
\begin{equation}\label{discrete}
\frac{\left(f_k - f_{k-1}\right)}{\tau} =  \nabla \cdot \left(A_{k-1}\nabla f_k - \nabla a_{k-1}f_k(1-f_k) + \delta_1\nabla f_k\right) - \delta_2 |v|^m f_k,
\end{equation}
for any fixed $\delta_1, \delta_2, \tau, m > 0$. Equivalently, we seek a fixed point $f_k = \Phi(f_k)$ satisfying the bound $0\le f_k \le 1$. To this end, we show $\Phi:L^2 \rightarrow L^2$ is continuous and compact, and the set $\{z \ | \ z = t\Phi(z), \text{for some }t\in[0,1]\}$ is bounded in $L^2$. Therefore, we apply the Schaeffer Fixed Point Theorem to conclude the following proposition:
\begin{proposition}\label{prop_step2}
Suppose $f_0 \in L^1$ with $0 \le f_0 \le 1$. Then, there is a family of functions $\{f_k\}_{k=0}^{N}$ such that $f_k \in L^2_m \cap H^1$ and $\{f_k\}$ solve \eqref{discrete}. That is, for $k \ge 1$,
$f_k$ satisfies that for any $\varphi \in H^1 \cap L^2_m$,
\begin{align}
\int\frac{\left(f_k - f_{k-1}\right)}{\tau} \varphi - \nabla a_{k-1}f_k(1-f_k) \cdot\nabla\varphi  \;dv = &- \int A_{k-1}\nabla f_k \cdot \nabla \varphi \;dv \nonumber  \\
&- \delta_1\int \nabla f_k \cdot \nabla\varphi \;dv- \delta_2\int |v|^mf_k\varphi \;dv.\label{weak_nonlinear}
\end{align}
Furthermore, for each $k\ge 1$, $f_k \in L^1$ and $0 \le f_k \le 1$.
\end{proposition}
In the third step of our construction, we seek a weak solution to the auxiliary equation \eqref{FDL_approx}
on a time interval $[0,T]$. To this end, we divide $[0,T]$ into $N$ pieces of size $\tau_N$ and from Proposition \ref{prop_step2}, for an initial datum $f_{in}$, we may define
$$
f^{(N)}(v,t)=f_{in}(v)\chi_{0}(t) + \sum_{k=1}^N f_k(v)\chi_{(t_{k-1},t_k]}(t),
$$
where $\{f_k\}_{0\le k\le N}$ solves \eqref{discrete} with parameters $\tau = \delta_2 = \tau_n$. We show propagation of $L^1$ moments and use a variant of the Aubin-Lions Lemma to conclude the following proposition:
\begin{proposition}\label{prop_step3}
Suppose $f_{in} \in L^1$, $|v|^2f_{in} \in L^1$, and $0 \le f_{in} \le 1$ and $\delta_1 > 0$. Then, for any $T > 0$, there is an $f:[0,T] \times \R^3 \rightarrow \R$ with $0\le f \le 1$ such that for each $1\le p < \infty$, $f \in L^\infty([0,T];L^p)$, $f\in C([0,T];L^2)$,  $f\in L^2([0,T];H^1)$, $f \in L^\infty([0,T];L^1_2)$ and $f$ satisfies \eqref{FDL_approx} in the form,
\begin{align}
\int_{\R^3} f_{in}\varphi(0) \;dv- \int_0^T\int_{\R^3} f\partial_t \varphi \;dvdt=& - \int_0^T\int_{\R^3} \left(A[f(1-f)]\nabla f - \nabla a[f]f(1-f)\right) \cdot \nabla\varphi \;dvdt \nonumber  \\
    &- \delta_1\int_0^T\int_{\R^3} \nabla f \cdot \nabla \varphi \;dvdt, \label{approx_distribution}
\end{align}
for each $\varphi \in C^\infty_c([0,T)\times \R^3)$ and
\begin{align}\label{approx_weak}
\int_0^T \langle \partial_tf, \Phi \rangle_{H^1,H^{-1}}\;dt =&  - \int_0^T\int_{\R^3} \left(A[f(1-f)]\nabla f - \nabla a[f]f(1-f)\right) \cdot \nabla\Phi  \;dvdt \nonumber  \\
&- \delta_1\int_0^T\int_{\R^3} \nabla f \cdot \nabla \Phi\;dvdt,
\end{align}
for each $\Phi \in L^2([0,T];H^1)$. Furthermore, $f$ conserves mass and satisfies the bound
\begin{equation*}
\|f\|_{L^\infty([0,T];L^1_2)} + \delta_1\|\nabla f\|_{L^2([0,T];L^2)} + \delta_1\|\partial_t f\|_{L^2([0,T];H^{-1})} \le C(\|f_{in}\|_{L^1_2},T). 
\end{equation*}
\end{proposition}
Finally, in the fourth step of our construction, we conclude the proof of Theorem \ref{thm_existence}. From Proposition \ref{prop_step3}, for an initial datum $f_{in}$ and a sequence $\delta_n \rightarrow 0^+$, we obtain a family of solutions $\{f_n\}$ to the equation \eqref{FDL_approx} with parameters $\delta_1 = \delta_n$ on the interval $[0,T]$. We show propagation of higher $L^1$ moments and an H-Theorem for the equation \eqref{FDL_approx}. Combined with Lemma \ref{lemA_lower}, this implies a uniform lower bound on the coefficients $A[f_n(1-f_n)]$, which is sufficient to gain compactness as $n \rightarrow \infty$.

\subsection{Step 1: Existence and Uniqueness of Solutions to \eqref{linear}}

In this step, we use the Lax-Milgram Theorem to prove Proposition \ref{prop_step1}. We recall that in this step, we construct weak solutions $f_k$ to
\begin{equation}
\begin{aligned}
\frac{\left(f_k - f_{k-1}\right)}{\tau} = \nabla \cdot \left(A_{k-1}\nabla f_k - \nabla a_{k-1}z_+(1-z)_+ + \delta_1\nabla f_k\right) - \delta_2 |v|^m f_k,\\
A_{k-1} := A[f_{k-1}(1-f_{k-1})] \qquad \text{and} \qquad a_{k-1} := a[f_{k-1}],
\end{aligned}
\end{equation}
where $f_{k-1}$, $z$, $\tau$, $\delta_1$, $\delta_2$, and $m$ are fixed.

\begin{flushleft}
\underline{Proof of Proposition \ref{prop_step1}.}\\
\end{flushleft}
We define
\begin{align*}
B[u,\varphi] &= \int A_{k-1}\nabla u \cdot \nabla \varphi + \delta_1\nabla u \cdot \nabla \varphi + \tau^{-1} u \varphi + \delta_2|v|^m u\varphi \;dv,\\
L[\varphi] &= \int \tau^{-1}f_{k-1} \varphi + \nabla a_{k-1}z_+(1-z)_+ \cdot\nabla \varphi \;dv.
\end{align*}
Since $0 \le f_{k-1} \le 1$, $A_{k-1} \ge 0$ and we have 
\begin{equation*}
B[u,u] \ge \int\delta_1 |\nabla u|^2 + \tau^{-1}u^2 + \delta_2|v|^m u^2 \;dv\gtrsim_{\delta_1,\delta_2,\tau} \|u\|_{H^1}^2 + \|u\|_{L^2_m}^2.
\end{equation*}
Therefore, $B[u,\varphi]$ is coercive on $H^1 \cap L^2_m$. Moreover, $B$ is bounded on $H^1 \cap L^2_m$ thanks to Lemma \ref{lemA_upper} and $0\le f_{k-1} \le 1$ as
\begin{equation*}
\begin{aligned}
|B[u,\varphi]| &\le \left(\|A_{k-1}\|_{L^\infty} + \delta_1\right)\|\nabla u\|_{L^2}\|\nabla \varphi\|_{L^2} + \tau^{-1}\|u\|_{L^2}\|\varphi\|_{L^2} + \delta_2\|u\|_{L^2_m}\|\varphi\|_{L^2_m}\\
	&\lesssim_{\delta_1,\delta_2,\tau,\|f_{k-1}\|_{L^1}} \|u\|_{H^1\cap L^2_m}\|\varphi\|_{H^1\cap L^2_m}.
\end{aligned}
\end{equation*}
Also, $L$ is bounded on $H^1\cap L^2_m$ by the Cauchy-Schwarz Inequality and Lemma \ref{lema_upper},
\begin{equation*}
\begin{aligned}
|L(\varphi)| &\le \tau^{-1}\|f_{k-1}\|_{L^2}\|\varphi\|_{L^2} + \|\nabla a_{k-1}\|_{L^2}\|z_+(1-z)_+\|_{L^\infty}\|\nabla \varphi\|_{L^2}\\
&\lesssim_{\tau,\|f_{k-1}\|_{L^2\cap L^{6/5}}} \|\varphi\|_{H^1}.
\end{aligned}
\end{equation*}
We conclude, using the Lax-Milgram Theorem on $H^1 \cap L^2_m$, that there is a unique $f_k \in H^1 \cap L^2_m$ satisfying the weak formulation \eqref{weak_linear}.

\subsection{Step 2: Existence of Solutions to \eqref{discrete}}

In this step, we use a fixed point argument to prove Proposition \ref{prop_step2}. We show that the nonlinear, semi-discretized equation,
\begin{equation*}
\frac{\left(f_k - f_{k-1}\right)}{\tau} =  \nabla \cdot \left(A[(1-f_{k-1})f_{k-1}]\nabla f_k - \nabla a[f_{k-1}]f_k(1-f_k) + \delta_1\nabla f_k\right) - \delta_2 |v|^m f_k,
\end{equation*}
has a solution $f_k$ provided $f_{k-1}$ is known and satisfies $0 \le f_{k-1} \le 1$ and $f_{k-1}\in L^1$. Moreover, we show these assumptions are propagated, so that for a fixed $f_0 = f_{in}$, we have the existence of a family $\{f_k\}$ for $k = 0, 1, ...,N$ for any $N$.

We begin by showing the existence of solutions $f_k$ to the nonlinear weak formulation \eqref{weak_nonlinear} provided $f_{k-1}$ is known and satisfies $f_{k-1}\in L^1$ and $0 \le f_{k-1} \le 1$. To this end, we fix $k$ and define $ \Phi : X \to X$ with $\Phi(z) = f_k$, where $f_k$ is the unique solution to \eqref{weak_linear} given $z$ (and fixed $\delta_1, \delta_2, \tau, m, f_{k-1}$). We also fix $X$ to be $L^2$. We would like to apply the Schaeffer Fixed Point Theorem \cite[Theorem 11.3]{GT98} to $\Phi: X\rightarrow X$ to conclude that there exists a fixed point for $\Phi$ in $X$. To apply Schaeffer's Theorem we need to verify the following conditions:
\begin{itemize}
\item The map $\Phi$ maps $X$ into $X$, i.e. if $z\in L^2$, the weak solution $f_k$ to \eqref{weak_linear} also satisfies $f_k \in L^2$. This is done in Lemma \ref{step2_lem1} via an $L^2$ estimate.
\item The set of approximate fixed points,
$$\set{z \ \big| \ z = t\Phi(z) \ \text{for some }0\le t\le 1}$$
should be bounded in $X$. This is done in Lemma \ref{step2_lem2}.
\item The map $\Phi$ is compact. This is done in Lemma \ref{step2_lem3} via the compact embedding $\Phi(X) \subset H^1\cap L^2_m \embeds L^2$.
\item The map $\Phi: X\rightarrow X$ is continuous. This is done in Lemma \ref{step2_lem4} by showing that if $z_k \rightarrow z$, the corresponding weak solutions $\Phi(z_k)$ converge to the unique weak solution $\Phi(z)$ of \eqref{weak_linear}.
\end{itemize}  
To this end, we have our first a priori bound:
\begin{lemma}\label{step2_lem1}
For  $f_{k-1} \in L^1$ with $0 \le f_{k-1} \le 1$, let $f_k$ be the unique solution to \eqref{weak_linear}. Then, $f_k \in H^1 \cap L^2_m$ and satisfies the estimate
\begin{equation}\label{step2_lem1_eqn0}
\|f_k\|_{L^2}^2 + \tau\delta_1\|\nabla f_k\|_{L^2}^2 + 2\tau\delta_2\||v|^{m/2} f_k\|_{L^2}^2 \le \|f_{k-1}\|_{L^2}^2 + C\frac{\tau}{\delta_1}\|f_{k-1}\|_{L^{6/5}}^2.
\end{equation}
\end{lemma}

\begin{proof}
We test \eqref{weak_linear} with $\varphi = f_k$. Using $A_{k-1} \ge 0$, we obtain
\begin{equation*}
\tau^{-1}\|f_k\|_{L^2}^2 + \delta_1\|\nabla f_k\|_{L^2}^2 + \delta_2\||v|^{m/2}f_k\|_{L^2}^2 \le \tau^{-1}\int f_{k-1}f_k  \;dv + \int\nabla a_{k-1}(z)_+(1-z)_+ \cdot \nabla f_k \;dv .
\end{equation*}
We bound the first term on the right hand side with Young's inequality as
\begin{equation*}
\tau^{-1}\int f_{k-1}f_k  \;dv \le \frac{\tau^{-1}}{2}\|f_{k-1}\|_{L^2}^2 + \frac{\tau^{-1}}{2}\|f_k\|_{L^2}^2,
\end{equation*}
and the second term via Young's inequality and Lemma \ref{lema_upper} as 
\begin{equation*}
\int\nabla a_{k-1}(z)_+(1-z)_+ \cdot \nabla f_k \;dv \le C\delta_1^{-1}\|f_{k-1}\|_{L^{6/5}}^2 + \frac{\delta_1}{2}\|\nabla f_k\|_{L^2}^2.
\end{equation*}
Rearranging terms and combining bounds yield \eqref{step2_lem1_eqn0}.
\end{proof}

We note that the preceding lemma immediately implies the following result:
\begin{lemma}[A priori bounds on approximate fixed points]\label{step2_lem2}
Let $f_k$ be the unique solution to \eqref{weak_linear} with $f_{k-1} \in L^1$ with $0 \le f_{k-1} \le 1$ and $X := L^2$. The map $\Phi: X\to  X$ defined as $z \mapsto  f_k$ is such that $A := \{z \in X \ | \ t\Phi(z) = z \text{ for some }0\le t\le 1\}$ is a bounded subset of $X$.
\end{lemma}

\begin{proof}
Suppose $z \in A$. Then, we note by Lemma \ref{step2_lem1},
\begin{equation*}
\|z\|_{L^2}^2 \le \|\Phi(z)\|_{L^2}^2 \le \|f_{k-1}\|_{L^2}^2 + C\frac{\tau}{\delta_1}\|f_{k-1}\|_{L^{6/5}}^2,
\end{equation*}
which completes the proof.
\end{proof}

\begin{lemma}[Compactness]\label{step2_lem3}
For $\Phi$ and $X$ as in Lemma \ref{step2_lem2}, $\Phi(X)$ is pre-compact as a subset of $L^q(\R^3)$ for any $2 \le q < 6$.
\end{lemma}

\begin{proof}
Fix any such $q$. Then, we note that Lemma \ref{step2_lem1} guarantees that $\Phi(z)$ is bounded in $H^1 \cap L^2_m$, uniformly in $z$ measurable. We claim that $L^2_m\cap H^1$ embeds compactly in $L^q$ for $2 \le q < 6$ provided $m > 0$.

Indeed, fix $g_n$ a sequence uniformly bounded in $L^2_m$ and $H^1$, so that $\|g_n\|_{L^2_m \cap H^1} \le M$. Then, use Rellich-Kondrachev and a diagonalization argument to extract a subsequence $g_{n_k}$ for which $g_{n_k} \rightarrow g$ in $L^2(K) \cap L^q(K)$ for every $K \subset \R^3$ compact. We will show $g_{n_k} \rightarrow g$ in $L^q$. Fix $\eps > 0$. Then, decompose the norm into two parts via,
\begin{equation}\label{step2_lem3_eqn1}
\|g_{n_k} - g\|_{L^q}^q = \int_{B_R(0)} |g_{n_k} - g|^q \;dx + \int_{\R^3 \backslash B_R(0)} |g_{n_k} - g|^q \;dx .
\end{equation}
The first term converges to $0$ for any fixed $R$. For the second, we interpolate between $L^2$ and $L^6$ and use the Sobolev embedding $H^1\embeds L^6$ to guarantee the $L^6$ norm is uniformly bounded in $k$. Thus,
\begin{equation}
\begin{aligned}
\left(\int_{\R^3 \backslash B(0,R)} |g_{n_k} - g|^q \;dx \right)^{1/q} &\lesssim M^{1-\alpha}\left(\int_{\R^3 \backslash B(0,R)} |g_{n_k} - g|^2 \;dx \right)^{\alpha/2}\\
&\lesssim M^{1-\alpha}R^{-m\alpha/2}\|g_{n_k} - g\|_{L^2_m}^\alpha \lesssim MR^{-m\alpha{/2}},
\end{aligned}
\end{equation} 
where $\frac{1}{q} = \frac{\alpha}{2} + \frac{1 - \alpha}{6}$, i.e. $\alpha = \frac{6 - q}{2q}$. So for $m > 0$ and $2 \le q < 6$, this converges to $0$ as $R\rightarrow \infty$ uniformly in $k$. Thus, first pick $R$ sufficiently large that the second term of \eqref{step2_lem3_eqn1} is less than $\eps/2$ for all $k$. Then, pick $k$ sufficiently large such that first term of \eqref{step2_lem3_eqn1} is less than $\eps/2$.
\end{proof}

\begin{lemma}[Continuity]\label{step2_lem4}
Let $\Phi$ be defined as in Lemma \ref{step2_lem2}.
Suppose $z_n \rightarrow z$ strongly in $X$. Then, $\Phi(z_n) \rightarrow \Phi(z)$ strongly in $X$.
\end{lemma}

\begin{proof}
Suppose $z_n \rightarrow z$ in $X = L^2$. Combining the a priori bound from Lemma \ref{step2_lem1} and compactness from Lemma \ref{step2_lem2}, $\Phi(z_n)$ is uniformly bounded in $L^2_m \cap H^1$ and compact in $X = L^2$. Therefore, by extracting subsequences, it suffices to show that if $z_n \rightarrow z$ in $X$ and $\Phi(z_n) \rightarrow y$ in $X$ and $\Phi(z_n)\weak y$ in $H^1\cap L^2_m$, then $y = \Phi(z)$. Finally, since Proposition \ref{prop_step1} guarantees uniqueness of solutions to \eqref{weak_linear}, it suffices to show
\begin{equation}\label{step2_lem4_eqn1}
\begin{aligned}
\int\frac{\left(y - f_{k-1}\right)}{\tau} \varphi - \nabla a_{k-1}z_+(1-z)_+ \cdot\nabla\varphi \;dv =  &- \int\nabla\varphi \cdot A_{k-1}\nabla y \;dv \\
&-\delta_1\int\nabla \varphi \cdot \nabla y + \delta_2 |v|^m\varphi y \; dv.
\end{aligned}
\end{equation}
Since $\Phi(z_n)$ solves \eqref{weak_linear} with coefficients $z_n$, we know
\begin{align*}
\int\frac{\Phi(z_n) - f_{k-1}}{\tau} \varphi \;dv &= \int \nabla a_{k-1}(z_n)_+(1-z_n)_+ \cdot\nabla\varphi - \nabla\varphi \cdot A_{k-1}\nabla \Phi(z_n) \;dv \nonumber \\
&\qquad-\delta_1\int\nabla \varphi \cdot \nabla \Phi(z_n) + \delta_2 |v|^m\varphi \Phi(z_n) \;dv.
\end{align*}
The weak convergence $\Phi(z_n) \weak y$ in $L^2_m \cap H^1$ is sufficient to pass to the limit $n\rightarrow \infty$ in each term, except in the term containing $(z_n)_+(1-z_n)_+$. For this term, we first observe that
\begin{equation*}
\begin{aligned}
&\int \left|\left[(z_n)_+(1-z_n)_+ - z_+(1-z)_+\right]\nabla a_{k-1} \cdot \nabla \varphi \right|  \;dv\lesssim \|f_{k-1}\|_{L^1}\|\nabla \varphi\|_{L^2}\\
&\qquad\qquad\qquad\qquad\qquad\qquad\qquad\qquad\qquad\times\left(\int \left|(z_n)_+(1-z_n)_+ - z_+(1-z)_+\right|^2 \;dv \right)^{1/2}.
\end{aligned}
\end{equation*}
Since the function $\varphi(x) = x_+(1-x)_+$ is Lipschitz, we get 
\begin{equation*}
\begin{aligned}
\left(\int \left|(z_n)_+(1-z_n)_+ - z_+(1-z)_+\right|^2  \;dv \right)^{1/2} 	&\lesssim \|z_n - z\|_{L^2} \rightarrow 0,
\end{aligned}
\end{equation*}
since $z_n \rightarrow z$ in $X$. Therefore, we obtain \eqref{step2_lem4_eqn1} and the proof is complete.
\end{proof}

The following lemma states that any fixed point $f_k$ of $\Phi$ is also a solution to \eqref{weak_nonlinear}. Note, this is not immediate because \eqref{weak_nonlinear} does not contain any positive part operators, while \eqref{weak_linear} does.
\begin{lemma}\label{step2_lem5}
Suppose $f_k \in H^1\cap L^2_m$ satisfies $\Phi(f_k) = f_k$ with $f_{k-1} \in L^1$ and $0 \le f_{k-1} \le 1$. Then, $0 \le f_k \le 1$ and consequently $f_k$ solves \eqref{weak_nonlinear}.
\end{lemma}

\begin{proof}
The idea is to test the weak formulation \eqref{weak_linear} with $(f_k)_-$ and $(1-f_k)_-$ and show that both are identically $0$:
\begin{align*}
\tau^{-1}\int_{\{f_k \le 0\}} f_k^2 - f_{k-1}f_k  \;dv & + \int_{\{f_k \le 0\}} \nabla f_k \cdot \left(A_{k-1}\nabla f_k + \delta_1\nabla f_k\right)\;dv  \\
& + \delta_2\int_{\{f_k \le 0\}} |v|^m f_k^2  \;dv = 0.
\end{align*}
Since each term is positive, all are $0$ and we conclude $f_k = 0$ on $\{f_k \le 0\}$, i.e. $f_k \ge 0$.
Similarly, 
\begin{equation*}
\begin{aligned}
\tau^{-1}\int_{\{f_k \ge 1\}} (f_k - f_{k-1})(1 - f_k)\;dv&  - \int_{\{f_k \ge 1\}} \nabla f_k \cdot \left(A_{k-1}\nabla f_k + \delta\nabla f_k\right) \;dv \\
&+ \delta_2\int_{\{f_k \ge 1\}} |v|^mf_k(1-f_k)  \;dv= 0.
\end{aligned}
\end{equation*}
Now, because $f_{k-1} \le 1$, $(f_k - f_{k-1})\chi_{\{f_k \ge 1\}} \ge (f_k - 1)\chi_{\{f_k \ge 1\}} \ge 0$. Thus, each term is negative and we conclude $f_k \le 1$.
\end{proof}

Next, we show the assumption that $f_{k-1} \in L^1$ is propagated. That is, if $f_{k-1} \in L^1$, then $f_k \in L^1$ and therefore, we may iterate the fixed point argument to construct a family $\{f_k\}$ solving \eqref{weak_nonlinear}.
\begin{lemma}\label{step2_lem6}
Suppose $f_k \in H^1\cap L^2_m$ satisfies $f_k = \Phi(f_k)$ with $f_{k-1} \in L^1$ and $0 \le f_{k-1} \le 1$. Then, $f_k$ satisfies the estimate
\begin{equation}
\|f_k\|_{L^1} + \tau\delta_2\|f_k|v|^m\|_{L^1} = \|f_{k-1}\|_{L^1}.
\end{equation}
\end{lemma}

\begin{proof}
Let $\varphi_R(v)$ be a cutoff function in $C_c^\infty(\R^3)$ such that 
\begin{equation*}
\left\{  \begin{array}{c} 
                        0 \le \varphi_R \le 1,\\ 
                        \varphi_R(v) = 1 \quad \textrm{if}\quad  |v| \le R, \\
                        \varphi_R(v) = 0 \quad \textrm{if}\quad  |v| \ge 2R, \\
                       |\nabla\varphi_R| \le \frac{C}{R},\quad   |\nabla^2\varphi_R| \le \frac{C}{R^2}.
           \end{array}  \right.
\end{equation*}
Then, we test \eqref{weak_linear} with $\varphi_R$ to obtain
\begin{equation*}
\begin{aligned}
\int\left[\frac{\left(f_k - f_{k-1}\right)}{\tau} + \delta_2|v|^m f_k \right]\varphi_R\;dv = &- \int A_{k-1}\nabla f_k \cdot \nabla\varphi_R\;dv\\
& -\int \nabla a_{k-1}f_k(1-f_k) \cdot \nabla \varphi_R\;dv\\
&-\delta_1\int \nabla f_k \cdot \nabla\varphi_R\;dv\\
=: & \; I_1 + I_2 + I_3.
\end{aligned}
\end{equation*}
First, we claim that the right hand side converges to $0$ as $R\rightarrow \infty$. Indeed, we bound each term separately, beginning with $I_3$ as,
\begin{equation*}
\begin{aligned}
\delta_1 \int\nabla f_k \cdot \nabla\varphi_R \;dv &\le \frac{C\delta_1}{R^2}\int_{\{R \le |v| \le 2R\}} f_k\;dv\\
&\le \frac{C\delta_1}{R^2}\|f_k\|_{L^2}|\{R \le |v| \le 2R\}|^{1/2} \le \frac{C\delta_1}{R^{1/2}}.
\end{aligned}
\end{equation*}
Next, we bound $I_2$ using $0\le f_k \le 1$ and Lemma \ref{lema_upper} to obtain
\begin{equation*}
\int\left(\nabla a_{k-1}f_k(1-f_k)\right) \cdot \nabla \varphi_R\;dv \le \frac{C\|\nabla a_{k-1}\|_{L^2}\|f_k\|_{L^2}}{R} \le \frac{C\|f_{k-1}\|_{L^{6/5}}\|f_k\|_{L^2}}{R}.
\end{equation*}
For $I_1$, we integrate by parts to obtain
\begin{equation*}
\begin{aligned}
-\int A_{k-1}\nabla f_k \cdot \nabla\varphi_R\;dv &= \int f_k {(\nabla \cdot A_{k-1})}\cdot \nabla \varphi_R \;dv -\int \nabla \cdot \left(A_{k-1}f_k\right) \cdot \nabla \varphi_R\;dv\\
	&= \int f_k(\nabla \cdot A_{k-1})\cdot \nabla \varphi_R\;dv + \int \mathrm{tr}\left(f_kA_{k-1}\nabla^2 \varphi_R\right)\;dv\\
	&=:  I_1^1 + I_1^2.
\end{aligned}
\end{equation*}
Now, $I_1^1$ vanishes by a similar estimate, using Lemma \ref{lemA_upper}. Finally $I_1^2$ vanishes by the estimate
\begin{equation*}
\begin{aligned}
\int \mathrm{tr}\left(f_kA_{k-1}\nabla^2 \varphi_R\right) \;dv&\le \|A_{k-1}\|_{L^\infty}\|f_k\|_{L^2}\|\nabla^2 \varphi_R\|_{L^2}\\
    &\le \frac{C\|f_{k-1}\|_{L^1}\|f_k\|_{L^2}}{R^{1/2}}.
\end{aligned}
\end{equation*}
Thus, piecing together all the above estimates, we conclude that $ I_1 + I_2 + I_3$ vanishes as $R \rightarrow \infty$.
Second, taking $R_n \rightarrow \infty$ sufficiently fast so that $\varphi_{R_n}$ are increasing to $1$,  the monotone convergence theorem yields
\begin{equation*}
\int f_k\;dv + \tau\delta_2\int f_k|v|^m\;dv = \int f_{k-1}\;dv.
\end{equation*}
By Lemma \ref{step2_lem5}, $0 \le f_k \le 1$ and the proof is complete.
\end{proof}

\begin{flushleft}
\underline{Proof of Proposition \ref{prop_step2}}\\
\end{flushleft}
Fix $f_0 = f_{in}$ as in the statement of Proposition \ref{prop_step2}. Suppose moreover that $f_1,f_2, \dots, f_{k-1}$ have been constructed so that $0 \le f_i \le 1$ and $f_i \in L^1$ for $0 \le i \le k-1$ and $\{f_i\}_{i=0}^{k-1}$ satisfies \eqref{weak_nonlinear}. We will now construct $f_k$. Indeed, fix $X =L^2$ and $\Phi$ the solution map to \eqref{weak_linear} with $f_{k-1}$ fixed.

As stated at the beginning of this step, the role of Lemma \ref{step2_lem1}, \ref{step2_lem2},  \ref{step2_lem3} and \ref{step2_lem4} is to verify the hypotheses of the Schaeffer Fixed Point Theorem for $\Phi: X\rightarrow X$.
\begin{itemize}
\item Lemma \ref{step2_lem1} implies $\Phi$ maps $X$ to itself;
\item Lemma \ref{step2_lem2} implies that approximate fixed points of $\Phi$ are bounded in $X$;
\item Lemma \ref{step2_lem3} implies $\Phi$ is a compact map;
\item Lemma \ref{step2_lem4} implies $\Phi: X \rightarrow X$ is a continuous (nonlinear) map.
\end{itemize}
Therefore, the Schaeffer Fixed Point Theorem (see \cite[Theorem 11.3]{GT98} for a precise statement) yields a (not necessarily unique) fixed point $f_k$ of the map $z\mapsto \Phi(z)$. Because $\Phi(X) \subset L^2_m \cap H^1$, $f_k \in H^1 \cap L^2_m$. 
As $\Phi(f_k) = f_k$, $f_k$ solves
\begin{align}
\int\frac{\left(f_k - f_{k-1}\right)}{\tau} \varphi - \nabla a_{k-1}(f_k)_+(1-f_k)_+ \cdot\nabla\varphi  \;dv = &- \int A_{k-1}\nabla f_k \cdot \nabla \varphi \;dv \nonumber  \\
&- \delta_1\int \nabla f_k \cdot \nabla\varphi \;dv- \delta_2\int |v|^mf_k\varphi \;dv.\nonumber
\end{align}
However, since $0 \le f_{k-1} \le 1$ by Lemma \ref{step2_lem5}, $0\le f_k \le 1$, and we may remove the positive parts to conclude $f_k$ solves the desired weak formulation, namely \eqref{weak_nonlinear}. Finally, Lemma \ref{step2_lem6} implies $f_k\in L^1$. By induction, the proof is complete.

\subsection{Step 3: Existence of Solutions to \eqref{FDL_approx}}

In this step we construct weak solutions $f:[0,T] \times \R^3 \rightarrow \R$ to the nonlinear, continuous time equation,
\begin{equation}
\partial_t f = \nabla \cdot \left(A[f(1-f)]\nabla f - \nabla a[f]f(1-f)\right) + \delta_1\Delta f,
\end{equation}
on an arbitrary fixed time interval $[0,T]$ for any fixed $\delta_1 > 0$ and for fixed initial data $f_{in}$, where $f_{in} \in L^1$ and $0 \le f_{in} \le 1$. We first prove uniform in $\tau$ (the time mesh) and $\delta_2$ (the strength of the added localization) estimates on solutions to equation \eqref{discrete}. 
For all $T>0$, let $N=\frac{T}{\tau}$. Define the piecewise interpolant of $\{f_k\}$ as 
\begin{equation}\label{decomposition}
f^{(N)}(v,t)=f_{in}(v)\chi_{0}(t) + \sum_{k=1}^N f_k(v)\chi_{(t_{k-1},t_k]}(t),
\end{equation}
and the backward finite difference operator $D_\tau$ as 
$$
D_\tau f(t) :=\frac{f(t) - f(t-\tau)}{\tau}.
$$
We also  introduce the shift operator
$$
\sigma_N ( f^{(N)}) (\cdot,t) = f_{k-1} \quad \textrm{for} \quad t\in (t_{k-1},t_k].
$$
With this new notation, we can rewrite (\ref{weak_nonlinear}) as 
\begin{equation}\label{weak_nonlinear_II}
\begin{aligned}
\int_0^T \int  D_\tau f^{(N)} \varphi - &\nabla a_Nf^{(N)}(1-f^{(N)}) \cdot\nabla\varphi \;dvdt \\
 = &-\int_0^T  \int A_N\nabla f^{(N)} \cdot \nabla \varphi  - \delta_1 \nabla f^{(N)} \cdot \nabla\varphi  - \delta_2 |v|^m f^{(N)}\varphi \;dvdt,
\end{aligned}
\end{equation} 
where
$$A_N = A[\sigma_N(f^{(N)})(1-f^{(N)})] \qquad \text{and}\qquad a_N = a[\sigma_N(f^{(N)})].$$
For strong compactness, we need propagation of moments (shown in Lemma \ref{step3_lem2}) in the form of
\begin{equation*}
    \|f^{(N)}\|_{L^\infty([0,T];L^1_2)} \lesssim_T \|f_{in}\|_{L^1_2},
\end{equation*}
and a variation of the Aubin-Lions Lemma for piecewise constant functions, which requires an estimate (shown in Lemma \ref{step3_lem3}) of the form
\begin{equation*}
\left\|D_\tau f^{(N)} \right\|_{L^2([0,T];H^{-1})} + \|f^{(N)}\|_{L^2([0,T];H^1 \cap L^1_2)} \lesssim \|f_{in}\|_{L^1_2}.
\end{equation*}

We begin with $L^1$ and $L^2$ estimates, which are continuous-time analogous of Lemma \ref{step2_lem6} and Lemma \ref{step2_lem1}, respectively. 
\begin{lemma}[$L^1$ and $L^2$ Estimates]\label{step3_lem1}
Suppose $f_{in} \in L^1$ and $0\le f_{in} \le 1$. Then, the following estimates hold:
\begin{equation}\label{step3_lem1_eqn1}
\|f^{(N)}\|_{L^\infty([0,T];L^2)}^2 + 2\delta_1\|\nabla f^{(N)}\|_{L^2([0,T];L^2)}^2 \le \|f_{in}\|_{L^2}^2 + T\|f_{in}\|_{L^1}^{3/2},
\end{equation}
and
\begin{equation}\label{step3_lem1_eqn2}
\|f^{(N)}\|_{L^\infty([0,T];L^1)} \le \|f_{in}\|_{L^1}.
\end{equation}
\end{lemma}

\begin{proof}
Inequality \eqref{step3_lem1_eqn2} follows by iterating Lemma \ref{step2_lem6}. Next, we estimate the $L^2$ norm of $f^{(N)}$ by testing \eqref{weak_nonlinear} with $f_k$, using Young's inequality and $A_{k-1}\ge 0$ to obtain 
\begin{equation*}
\frac{1}{2}\|f_k\|_{L^2}^2 + \tau\delta_1\|\nabla f_k\|_{L^2}^2 \le \frac{1}{2}\|f_{k-1}\|_{L^2}^2+ \tau \int \nabla a_{k-1}f_k(1-f_k) \cdot \nabla f_k \;dv.
\end{equation*}
For the last integral, we integrate by parts, using $-\Delta a_{k-1} = f_{k-1}$ and get 
\begin{equation*}
\begin{aligned}
2\int \nabla a_{k-1} f_k(1-f_k) \cdot \nabla f_k \;dv &=2 \int \nabla a_{k-1} \cdot \nabla\left[ \frac{1}{2}f_k^2 - \frac{1}{3}f_k^3\right]\;dv\\
 	&= 2\int f_{k-1} \left[\frac{1}{2}(f_k)^2 - \frac{1}{3}(f_k)^3\right] \; dv.
\end{aligned}
\end{equation*}
Since $0\le f_k \le 1$, $\left[\frac{1}{2}(f_k)^2 - \frac{1}{3}(f_k)^3\right] \ge 0$. Therefore, using $0\le f_{k-1} \le 1$, the interpolation inequality $\|g\|_{L^2} \le \|g\|^{1/4}_{L^1}\|g\|^{3/4}_{L^3}$, and Young's inequality, we have
\begin{equation*}
\begin{aligned}
2\int \nabla a_{k-1} f_k(1-f_k) \cdot \nabla f_k \;dv &\le \|f_k\|_{L^2}^2 - \frac{2}{3}\|f_k\|_{L^3}^3\\
&\le \|f_k\|^{1/2}_{L^1}\|f_k\|^{3/2}_{L^3} - {\frac{2}{3}}\|f_k\|_{L^3}^3\\
	&\le \frac{3}{8}\|f_k\|_{L^1}.
\end{aligned} 
\end{equation*}
Using Lemma \ref{step2_lem6}, we obtain
\begin{equation*}
\|f_k\|_{L^2}^2 + 2\tau\delta_1\|\nabla f_k\|_{L^2}^2 \le \|f_{k-1}\|_{L^2}^2 + \tau\|f_{k-1}\|_{L^1}^{3/2},
\end{equation*}
which implies, recursively,
\begin{equation*}
\begin{aligned}
\sup_{0\le j \le k} \|f_j\|_{L^2}^2 + 2\delta_1\left(\sum_{j=1}^k \tau\|\nabla f_j\|_{L^2}^2\right) &\le \|f_{in}\|_{L^2}^2 + \tau\sum_{j=0}^{k-1}\|f_j\|_{L^1}^{3/2}\\
	&\le \|f_{in}\|_{L^2}^2 + {\tau}\sum_{j=0}^{k-1}\|f_{in}\|_{L^1}^{3/2}\\
	&\le \|f_{in}\|_{L^2}^2 + k\tau\|f_{in}\|^{3/2}_{L^1}.
\end{aligned}
\end{equation*}
Taking $k = N$ and recalling the definition of $f^{(N)}$ in \eqref{decomposition} finish the proof of the lemma. 
\end{proof}

\begin{lemma}[Propagation of Moments]\label{step3_lem2}
Suppose $f_{in} \in L^1_2$ and $0\le f_{in} \le 1$. Then, the following estimates hold:
\begin{equation}\label{step3_lem2_eqn1}
\|f^{(N)}\|_{L^\infty([0,T];L^1_1)} \le \|f_{in}\|_{L^1_1} + C(\|f_{in}\|_{L^1})T,
\end{equation}
and
\begin{equation}\label{step3_lem2_eqn2}
\|f^{(N)}\|_{L^\infty([0,T];L^1_2)} \le \|f_{in}\|_{L^1_2} + C(\|f_{in}\|_{L^1_1})T,
\end{equation}
where the implicit constants are independent of $\tau_n$, $\delta_1$, and $\delta_2$.
\end{lemma}

\begin{proof}
Let $\varphi_R \in C^\infty_c(\R^3)$ be as in Lemma \ref{step2_lem6}.
Then, we test \eqref{weak_nonlinear} with $\brak{v}\varphi_R(v)$ to obtain
\begin{equation*}
\begin{aligned}
\int \varphi_R \brak{v} f_k dv + \tau\delta_2\int \varphi_R \brak{v}|v|^m f_k \;dv&= \int \varphi_R \brak{v}f_{k-1} \;dv \\
    &\qquad - \tau\int A_{k-1}\nabla f_k \cdot \nabla(\brak{v}\varphi_R)  \;dv \\
	&\qquad + \tau\int \nabla a_{k-1}f_k(1-f_k) \cdot \nabla (\brak{v}\varphi_R)  \;dv \\
	&\qquad - \tau\delta_1\int \nabla f_k \cdot \nabla (\brak{v}\varphi_R)  \;dv \\
	&=: \int \varphi_R \brak{v}f_{k-1}{\;dv} - \tau\left(I_1 - I_2 + \delta_1 I_3\right)   .
\end{aligned}
\end{equation*}
We bound $I_3$ using $|\Delta (\brak{v}\varphi_R)| \le C$ to obtain
\begin{equation*}
|I_3| = \left|\int f_k\Delta(\brak{v}\varphi_R) \;dv \right|  \le \|\Delta(\brak{v}\varphi_R)\|_{L^\infty}\|f_k\|_{L^1} \lesssim \|f_k\|_{L^1}. 
\end{equation*}
For $I_2$ we use Lemma \ref{lema_upper}, $0 \le f_k \le 1$, and $|\nabla(\brak{v}\varphi_R)| \le C$:
\begin{equation*}
\begin{aligned}
|I_2| \le \|\nabla a_{k-1}\|_{L^\infty}\|f_k\|_{L^1}\|\nabla(\brak{v}\varphi_R)\|_{L^\infty} \lesssim \left(\|f_{k-1}\|_{L^{1}} + \|f_{k-1}\|_{L^\infty}\right)\|f_k\|_{L^1}.
\end{aligned}
\end{equation*}
For $I_1$, we integrate by parts twice to get
\begin{equation*}
\begin{aligned}
I_1 &= \int \nabla \cdot (A_{k-1}f_k) \cdot \nabla(\brak{v}\varphi_R) - (\nabla \cdot A_{k-1})f_k \cdot \nabla(\brak{v}\varphi_R) \;dv \\
	&= - \int \mathrm{tr}(A_{k-1}f_k\nabla^2(\brak{v}\varphi_R)) \;dv- \int (\nabla \cdot A_{k-1})f_k \cdot \nabla(\brak{v}\varphi_R) \;dv \\
	&=:- I_{1,1} - I_{1,2}.
\end{aligned}
\end{equation*}
Lemma \ref{lemA_upper} and $|\nabla^2(\brak{v}\varphi_R)| \le C$ yields
\begin{equation*}
\begin{aligned}
|I_{1,1}| &\le \|A_{k-1}\|_{L^\infty}\|f_k\|_{L^1}\|\nabla^2(\brak{v}\varphi_R)\|_{L^\infty}\\
    &\lesssim \left(\|f_{k-1}\|_{L^1} + \|f_{k-1}\|_{L^\infty}\right)\|f_k\|_{L^1}.
\end{aligned}
\end{equation*}
Lemma \ref{lema_upper} and $|\nabla(\brak{v}\varphi_R)| \le C$ yield
\begin{equation*}
\begin{aligned}
|I_{1,2}| &\le \|\nabla \cdot A_{k-1}\|_{L^\infty}\|f_k\|_{L^1}\|\nabla(\brak{v}\varphi_R)\|_{L^\infty}\\
    &\lesssim \left(\|f_{k-1}\|_{L^1} + \|f_{k-1}\|_{L^\infty}\right)\|f_k\|_{L^1}.
\end{aligned}
\end{equation*}
Combining all above estimates we obtain
\begin{equation*}
\begin{aligned}
\sup_{0 \le j \le k} \int \varphi_R \brak{v} f_j \;dv&\le \int \varphi_R \brak{v} f_{in} \;dv+ C\sum_{j=1}^{k}\tau\left(\|f_{j-1}\|_{L^1} + \|f_{j-1}\|_{L^1}^2\right)\\
	&\le \int \brak{v} f_{in} \;dv+ Ck\tau\left(\|f_{in}\|_{L^1} + \|f_{in}\|_{L^1}^2\right).
\end{aligned}
\end{equation*}
Now, 
taking $k = N$, recalling the definition of $f^{(N)}$ in \eqref{decomposition} and letting $R\rightarrow \infty$, the monotone convergence theorem implies \eqref{step3_lem2_eqn1}.

The proof of \eqref{step3_lem2_eqn2} proceeds nearly identically after testing with $\brak{v}^2\varphi_R$.
\end{proof}

The bounds in Lemmas \ref{step3_lem1} and \ref{step3_lem2} are sufficient for weak or weak star compactness. For strong compactness, we will use the version of the Aubin-Lions Lemma for piecewise constant functions \cite[Theorem 1]{DJ12}.


\begin{lemma}\label{step3_lem3}
For any $T > 0$, $f_{in} \in L^1_2$, and $0 \le  f_{in} \le 1$, for $f^{(N)}$ defined above with $0 < m < 1$, 
\begin{equation}\label{step3_lem3_eqn1}
{\| D_\tau f^{(N)}\|_{L^2([0,T];H^{-1})}}\le C(\|f_{in}\|_{L^1_2},T,\delta_1).
\end{equation}
 Moreover, the family $\{f^{(N)}\}$ is compact in $L^2([0,T];L^q)$, provided $1 \le q < 6$.
\end{lemma}

\begin{proof}
Let us define the triple $X := L^1_2 \cap H^1$, $Y := L^q\cap L^2$ and $Z := H^{-1}$ for a fixed $1 \le q < 6$. Following the proof of Lemma \ref{step2_lem3}, the embedding $X\embeds Y$ is compact for $1 \le q < 6$. Certainly, $Y \embeds Z$ continuously for this range of $q$.  Moreover, we have shown in Lemmas \ref{step3_lem1} and \ref{step3_lem2},
\begin{equation} \label{L2_comp}
\|f^{(N)}\|_{L^2([0,T];X)} \le C(\delta_1,\|f_{in}\|_{L^1},\|f_{in}\|_{L^2},T),
\end{equation}
where the constant on the right hand side is independent of $\delta_2$ and $\tau$.
To obtain \eqref{step3_lem3_eqn1} we first consider 
\begin{align*}
\left| \int_0^T \int_{\mathbb{R}{^3}} D_\tau f^{(N)} \varphi \;dvdt \right| \le& \left|  \int_0^T \int_{\mathbb{R}{^3}}  A_N\nabla f^{(N)} \cdot \nabla\varphi \;dvdt \right| \\
&+ \left| \int_0^T \int_{\mathbb{R}{^3}}   \nabla a_N f^{(N)}(1-f^{(N)}) \cdot \nabla \varphi \; dvdt\right| \\
&+ \delta_1 \left| \int_0^T \int_{\mathbb{R}{^3}}  \nabla f^{(N)} \cdot \nabla \varphi \;dvdt\right| \\
&+ \delta_2 \left| \int_0^T \int_{\mathbb{R}{^3}} f^{(N)} \varphi |v|^m\;dvdt\right|=:I_1+...+I_4.
\end{align*}
For $\varphi \in L^2([0,T];H^1)$, thanks to Lemma \ref{lemA_upper}, one gets
\begin{equation*}
\begin{aligned}
 I_1  	&\lesssim \int_0^T \|\varphi(t)\|_{H^1}\|\nabla f^{(N)}\|_{L^2}\|A_N\|_{L^\infty} \;dt\\
	&\lesssim \|\varphi\|_{L^2(0,T;H^1)}\|\nabla f^{(N)}\|_{L^2(0,T;L^2)}\left(\|f_{in}\|_{L^1} + 1\right),
\end{aligned}
\end{equation*}
and, using Lemma \ref{lema_upper}, 
\begin{equation*}
\begin{aligned}
 I_2 	&\lesssim \int_0^T \|\varphi(t)\|_{H^1}\|f^{(N)}(1-f^{(N)})\|_{L^\infty}\|\nabla a_N \|_{L^2} \;dt\\
	&\lesssim \|\varphi\|_{L^2(0,T;H^1)}\|f^{(N)}\|_{L^2(0,T;L^{6/5})}.
\end{aligned}
\end{equation*}
Finally,
\begin{equation*}
\begin{aligned}
I_3 \lesssim \|\varphi\|_{L^2(0,T;H^1)}\|\nabla f^{(N)}\|_{L^2(0,T;L^2)},
\end{aligned}
\end{equation*}
and since $2m < 2$,
\begin{equation*}
\begin{aligned}
I_4  \lesssim \|\varphi\|_{L^2(0,T;L^2)}\|f^{(N)}\|_{L^1(0,T;L^1_{2m})} \le C(T,f_{in})\|\varphi\|_{L^2(0,T;L^2)},
\end{aligned}
\end{equation*}
using Lemma \ref{step3_lem2}.  We note $\|\nabla f^{(N)}\|_{L^2([0,T];L^2)}$, $\|f^{(N)}\|_{L^\infty([0,T];L^1)}$, and $\|f^{(N)}\|_{L^\infty([0,T];L^1_2)}$ are uniformly bounded in $\delta_2$ and $N$ (but not in $\delta_1$) by Lemmas \ref{step3_lem1} and \ref{step3_lem2}. Thus (\ref{step3_lem3_eqn1}) follows. 

Theorem 1 in \cite{DJ12}, together with (\ref{step3_lem3_eqn1}) and (\ref{L2_comp}), yields the desired compactness.
\end{proof}

We are now ready to prove Proposition \ref{prop_step3}:
\begin{flushleft}
\underline{Proof of Proposition \ref{prop_step3}.}\\
\end{flushleft}

{Let} $\delta_2 = \tau$, fix some $0 < m < 1$, and $\{f^{(N)}\}_{{N}\in \mathbb{N}}$ be the corresponding sequence of piecewise constant solutions to \eqref{weak_nonlinear_II}. Thanks to  the estimates from Lemma \ref{step3_lem1}, Lemma \ref{step3_lem2}, and Lemma \ref{step3_lem3}, we may assume that $f^{(N)} $ converges to $f$, as $\tau \to 0$, in the following topologies:
\begin{itemize}
\item Weak star in $L^\infty([0,T]\times \R^3)$,
\item Weakly in $L^2([0,T]; H^1)$,
\item Weak star in $L^\infty([0,T];L^2)$,
\item Strongly in $L^p([0,T];L^q)$ for $1\le p \le 2$ and $1\le q < 6$.
\end{itemize}
Moreover, by taking a further subsequence, we will also have that $f^{( N)} \rightarrow f$ pointwise almost everywhere. Therefore, thanks to Fatou's lemma
\begin{equation*}
\|f\|_{L^\infty([0,T];L^1_2)} + \delta_1\|\nabla f\|_{L^2([0,T];L^2)} \le C(\|f_{in}\|_{L^1_2},T).
\end{equation*}
All these convergences are enough to pass to the limit $N\to +\infty$ in (\ref{weak_nonlinear_II}). We briefly highlight the convergence in the nonlinear terms. Let us first consider $\varphi \in C^\infty_c([0,T)\times \R^3)$. We have 
\begin{equation*}
\begin{aligned}
\biggr|\int_{0}^T\int_{\R^3}\bigg[\nabla a_N & f^{(N)}(1-f^{(N)}) - \nabla a[f] f(1-f)\bigg] \cdot \nabla\varphi \;dvdt \biggr|\\
    &\le \left|\int_{0}^T\int_{\R^3}\left(\nabla a_N-\nabla a[f]\right)f^{(N)}(1-f^{(N)}) \cdot \nabla\varphi\;dvdt\right|\\
 &\qquad+ \left|\int_{0}^T\int_{\R^3}\nabla a[f]\left[f(1-f) - f^{(N)}(1-f^{(N)})\right] \cdot \nabla\varphi\;dvdt \right|\\
 &=:  I_1 + I_2.
\end{aligned}
\end{equation*}
We estimate $I_1$ using H\"older's inequality and $\|f^{(N)}(1-f^{(N)})\|_{L^\infty} \le 1$:
\begin{equation*}
\begin{aligned}
I_1 &\le \int_{0}^T \|\nabla a[\sigma_Nf^{(N)}-f]\|_{L^2}\|\nabla\varphi\|_{L^2}\;dt\\
	&\lesssim \int_{0}^T \|\sigma_Nf^{(N)}-f]\|_{L^{6/5}}\|\nabla\varphi\|_{L^2}\;dt \to 0,
\end{aligned}
\end{equation*}
thanks to the strong convergence, and, similarly, using Lemma \ref{lema_upper},
\begin{equation*}
\begin{aligned}
I_2 \le2  \int_0^T \|\nabla a[f]\|_{L^\infty}\|f^{(N)}-f\|_{L^2}\|\nabla\varphi\|_{L^2 }\;dt \to 0.
\end{aligned}
\end{equation*}

Next, we handle the nonlinear term involving $A_N$, which we decompose as
\begin{equation*}
\begin{aligned}
\biggr|\int_{0}^T\int_{\R^3} & \bigg[A_N\nabla f^{(N)} - A[f(1-f)]\nabla f\bigg] \cdot \nabla\varphi \;dv dt \biggr|\\
    &\le \left|\int_{0}^T\int_{\R^3}\left(A_N - A[f(1-f)]\right)\nabla f^{(N)} \cdot \nabla\varphi \;dvdt \right|\\
 &\qquad+ \left|\int_{0}^T\int_{\R^3}A\left[f(1-f)\right](\nabla f - \nabla f^{(N)}) \cdot \nabla\varphi \;dvdt\right|\\
 &=:  J_1 + J_2.
\end{aligned}
\end{equation*}
The term $J_2$ convergence to zero thanks to the weak convergence of $f^{(N)}$ in $H^1(\mathbb{R}{^3})$ and   Lemma \ref{lemA_upper}. 
For $J_1$, we use H\"older's inequality, Lemma \ref{lemA_upper}, estimate (\ref{step3_lem1_eqn1}) and the strong convergence in $L^2([0,T];L^2)$ to obtain $J_1 \to 0$, since 
\begin{equation*}
\begin{aligned}
J_1 
	&\lesssim_T \|\nabla f^{(N)}\|_{L^2([0,T];L^2)} \|\varphi\|_{L^\infty([0,T];H^1)}\|\sigma_N(f^{(N)})-f\|_{L^\infty([0,T];L^1)}^{4/3}\|\sigma_N(f^{(N)})-f\|_{L^2([0,T];L^2)}^{2/3}.
\end{aligned}
\end{equation*}
We treat the left hand side of \eqref{weak_nonlinear_II} by integrating by parts,
\begin{equation*}
\begin{aligned}
\int_{0}^T\int_{\R^3} D_\tau f^{(N)} \varphi \;dvdt  =& - \int_{0}^{T-\tau} D_{-\tau} \varphi  f^{(N)}(t) \;dvdt\\
 & + \frac{1}{\tau}\int_{T -\tau}^T\int_{\R^3} f^{(N)}(t)\varphi(t) \;dvdt -  \frac{1}{\tau}\int_{-\tau}^{0}\int_{\R^3} f^{(N)}(t)\varphi(t + \tau)\;dvdt.
\end{aligned}
\end{equation*}
For $N$ sufficiently large, 
$$\frac{1}{\tau}\int_{T -\tau}^T\int_{\R^3} f^{(N)}(t)\varphi(t) \;dvdt = 0,$$
as $\varphi$ is compactly supported in $[0,T) \times \R^3$. Moreover, for $0\le t < \tau$, $f^{( N)}(t) = f_{in}$ so that
\begin{equation*}
\frac{1}{\tau}\int_{{-\tau}}^{ 0}\int_{\R^3} f^{(N)}(t)\varphi(t + \tau) \;dvdt = \frac{1}{\tau}\int_{{-\tau}}^{0}\int_{\R^3} f_{in}\varphi(t + \tau) \;dvdt .
\end{equation*}
Since $\varphi$ is smooth, the right hand side converges to $\int_{\R^3} \varphi(0,v)f_{in}(v)$ as $N\rightarrow \infty$. Finally, since $\varphi$ is smooth and $f^{(N)}$ are uniformly bounded in $L^2([0,T];L^2)$, we have 
\begin{equation}\label{step3_prop_eqn11}
- \int_{0}^{T-\tau} D_{-\tau} \varphi  f^{(N)}(t) \;dvdt \rightarrow - \int_0^T\int_{\R^3} f(v,t) \partial_t \varphi(v,t) \;dvdt.
\end{equation}
This concludes the proof of \eqref{approx_distribution}. 
Lemma \ref{step3_lem3} implies that, for some $g \in L^2([0,T];H^{-1})$,
\begin{equation*}
\int_{0}^T\int_{\R^3} D_\tau f^{(N)}\varphi \;dvdt \rightarrow \int_0^T\int_{\R^3} g\varphi \;dvdt,
\end{equation*}
for every $\varphi \in L^2([0,T];H^1)$. Hence,  \eqref{step3_prop_eqn11} yields $g = \partial_t f$. The distributional formulation implies
\begin{align*}
\int_0^T \brak{\varphi,\partial_t f}_{H^1\times H^{-1}}\;dt  = & - \int_0^T\int_{\R^3} \left(A[f(1-f)]\nabla f - \nabla a[f]f(1-f)\right) \cdot \nabla\varphi \;dvdt\\
& -  \delta_1\int_0^T\int_{\R^3} \nabla f \cdot \nabla \varphi\;dvdt,
\end{align*}
for each $\varphi \in C^\infty_c([0,T)\times \R^3)$. Now, fix $\Phi \in L^2([0,T];H^1)$ and let $\varphi_{\eps} \in C^\infty_c([0,T)\times \R^3)$ such that $\|\Phi - \varphi_{\eps}\|_{L^2([0,T];H^1)} \le \eps$. Then, substituting $\varphi_{\eps}$ into the above weak formulation, and passing to the limit $\varepsilon \to 0$, we obtain \eqref{approx_weak}.

Finally, we note that because $f\in L^2([0,T];H^1)$ and $\partial_t f\in L^2([0,T];H^{-1})$, $f\in C([0,T];L^2)$ and therefore \eqref{approx_weak} implies $f(t) \rightarrow f_{in}$ {strongly in $L^2$} as $t \rightarrow 0^+$. Moreover, repeating the proof of Lemma \ref{step2_lem6}, the additional $\delta_2\|f|v|^m\|_{L^1}$ term disappears thanks to the uniform bound from Lemma \ref{step3_lem2}, and we obtain conservation of mass.

\subsection{Step 4: Proof of Theorem \ref{thm_existence}}

We conclude the proof of Theorem \ref{thm_existence} by showing compactness in $\delta_1$ for solutions to \eqref{FDL_approx}. We already have uniform in $\delta_1$ bounds of the form,
\begin{equation}
    \|f_{\delta_1}\|_{L^\infty([0,T];L^1_2)} + \|f_{\delta_1}\|_{L^\infty([0,T]\times \R^3)} \le C(\|f_{in}\|_{L^1_2},T).
\end{equation}
Thus, to gain strong compactness as $\delta_1\rightarrow 0^+$, we will show (in Lemma \ref{step4_lem3}) the estimate
\begin{equation}
    \|f_{\delta_1}\|_{L^2([0,T];H^1)} + \|\partial_t f_{\delta_1}\|_{L^2([0,T];H^{-1})} \le C(\|f_{\delta_1}\|_{L^\infty([0,T];L^1_3)},T),
\end{equation}
by leveraging the degenerate dissipation present in \eqref{LFD} (see Lemma \ref{lemA_lower}), which up to this point, we have neglected. However, we do not have control over $L^1_3$ and therefore, we also show propagation of higher moments in Lemma \ref{step4_lem1}. 

To this end, we recall the dependence of our solutions on the parameter $\delta_1$. Throughout this section, we will write $f_\delta:[0,T] \times \R^3 \rightarrow \R$ to denote the solution $f_\delta$ to \eqref{FDL_approx} on $[0,T]$ constructed in Proposition \ref{prop_step3} with parameter $\delta_1 = \delta$.
Let us begin with a propagation of higher moments estimate that is uniform in $\delta$:
\begin{lemma}\label{step4_lem1}
Suppose $f_{in} \in L^1_s$ for some $s > 2$ and $T > 0$. Then, the family $\{f_\delta\}_{0 < \delta < 1}$ satisfies the uniform in $\delta$ estimate,
\begin{equation}\label{step4_lem1_eqn1}
\|f_\delta\|_{L^\infty([0,T];L^1_s)} \le C(\|f_{in}\|_{L^1_s},T).
\end{equation}
\end{lemma}

\begin{proof}
We note that the propagation of moments for $0 \le s \le 2$ follows directly from Proposition \ref{prop_step3}. We will prove the rest of the them by induction on the integer part of $s$. Indeed, fix some $2 \le n < s \le n+1$ and suppose that \eqref{step4_lem1_eqn1} holds for any $0 \le s \le n$. Then, we will test \eqref{approx_weak} with $\Psi(v,t) = \varphi_R(v)|v|^s \chi_{[0,t_0]}(t)$ where $\varphi_R \in C^\infty_c(\R^3)$ is as in Lemma \ref{step2_lem6}. We obtain
\begin{align*}
\int_{\R^3}( f_\delta(t_0) - f_{in}) \varphi_R |v|^s  \ dv =& - \int_0^{t_0}\int_{\R^3} A[f_\delta(1-f_\delta)]\nabla f_\delta  \cdot \nabla(|v|^s\varphi_R)\;dvdt \\
&+\int_0^{t_0}\int_{\R^3}   (\nabla a[f_\delta]f_\delta(1-f_\delta) - \delta\nabla f_\delta) \cdot \nabla(|v|^s\varphi_R)\;dvdt.
\end{align*}
We estimate the right hand side by decomposing into multiple parts:
\begin{equation*}
\begin{aligned}
-\int_0^{t_0}\int_{\R^3} \bigg(A[f_\delta(1-f_\delta)]\nabla f_\delta - \nabla a[f_\delta]f_\delta&(1-f_\delta) + \delta\nabla f_\delta\bigg) \cdot \nabla(|v|^s\varphi_R) \;dvdt\\
&= -\int_0^{t_0}\int_{\R^3} \nabla \cdot \left(A[f_\delta(1-f_\delta)]f_\delta\right)\cdot \nabla(|v|^s\varphi_R)\;dvdt \\
	&\quad + \int_0^{t_0}\int_{\R^3} f_\delta\left(\nabla \cdot A\right)[f_\delta(1-f_\delta)]\cdot \nabla(|v|^s\varphi_R)\;dvdt \\
	&\quad + \int_0^{t_0}\int_{\R^3}\left(\nabla a[f_\delta]f_\delta(1-f_\delta)\right) \cdot \nabla(|v|^s\varphi_R)\;dvdt \\
	&\quad - \delta\int_0^{t_0}\int_{\R^3}\nabla f_\delta \cdot \nabla(|v|^s\varphi_R)\;dvdt \\
	&= -I_1 + I_2 + I_3 - \delta I_4.
\end{aligned}
\end{equation*}
For $I_1$, after integrating by parts,  thanks to Lemma \ref{lemA_upper} and $|\nabla^2(|v|^s\varphi_R)| \le C|v|^{s-2}$ we obtain
\begin{equation*}
\begin{aligned}
|I_1|
	&\lesssim \left(\|f_\delta\|_{L^1([0,T];L^\infty)} + \|f_\delta\|_{L^1([0,T];L^1)}\right)\|f_\delta |v|^{s-2}\|_{L^\infty([0,T];L^1)}\\
	&\le C(\|f_{in}\|_{L^1_{s-2}},T),
\end{aligned}
\end{equation*}
where in the last line we used the induction hypothesis.
For $I_2$, we use Lemma \ref{lemA_upper} and $|\nabla(|v|^s\varphi_R)| \le C|v|^{s-1}$ to obtain
\begin{equation*}
|I_2| \le \|f_\delta \nabla(|v|^s\varphi_R)\|_{L^\infty([0,T];L^1)}\|\nabla \cdot A[f_\delta(1-f_\delta)]\|_{L^1([0,T];L^\infty)} \le C(\|f_{in}\|_{L^1_{s-1}},T).
\end{equation*}
Similarly, for $I_3$, we use Lemma \ref{lema_upper} and $0 \le f_\delta \le 1$ to obtain
\begin{equation*}
|I_3| \le \|f_\delta \nabla(|v|^s\varphi_R)\|_{L^\infty([0,T];L^1)}\|\nabla a[f_\delta]\|_{L^1([0,T];L^\infty)} \le C(\|f_{in}\|_{L^1_{s-1}},T).
\end{equation*}
Finally, for $I_4$, integration by parts yields
\begin{equation*}
|I_4| \le T\|f\Delta(|v|^s\varphi_R)\|_{L^\infty([0,T];L^1} \le C(\|f_{in}\|_{L^1_{s-2}},T).
\end{equation*}
Combining all above estimates, we prove \eqref{step4_lem1_eqn1} for any $s\in (n,n+1]$. The proof is complete.
\end{proof}

The following lemma, combined with Lemma \ref{lemA_lower}, gives a quantitative lower bound on the ellipticity of $A[f_\delta(1-f_\delta)]$. This will allow us to gain some control over $\nabla f_\delta$ uniformly in $\delta$.
\begin{lemma}\label{step4_lem2}
Suppose $0 \le f_{in} \le 1$, $f_{in} \in L^1_2$, $H_1(f_{in})<0$, and $T>0$. Then, $f_\delta$ has decreasing entropy, i.e. for almost every $0\le t_1 < t_2\le T$,
\begin{equation}\label{step4_lem2_desired1}
    H_1(f_\delta(t_2)) \le H_1(f_{\delta}(t_1)).
\end{equation}
Moreover, the dissipative coefficients $A[f_\delta(1-f_\delta)]$ are bounded uniformly from below:
\begin{equation}\label{step4_lem2_desired2}
A[f_\delta(1-f_\delta)] \ge \frac{C(\|f_{in}\|_{L^1_2},H_1(f_{in}),T)}{1 + |v|^3}.
\end{equation}
\end{lemma}

\begin{proof}
By Lemma \ref{lemA_lower},  \eqref{step4_lem2_desired2} is a consequence of \eqref{step4_lem2_desired1} and  
\begin{equation}\label{step4_lem2_eqn1}
    \|f_\delta(t)\|_{L^1_2} \le C(\|f_{in}\|_{L^1_2}), \quad \textrm{for all} \;t>0.
\end{equation} 
The energy bound \eqref{step4_lem2_eqn1} is shown in Lemma \ref{step4_lem1}. It remains to estimate the entropy and obtain \eqref{step4_lem2_desired1}. 
We test \eqref{approx_weak} with 
 $$\psi_\eta: = \log(f_\delta + \eta) -\log(\eta) - \log(1-f_\delta + \eta ) + \log(1 + \eta),\quad \eta>0.$$


We have
\begin{equation*}
\begin{aligned}
    \int_{t_1}^{t_2}\int_{\R^3} \psi_\eta\partial_t f_\delta \;dvdt=& -\int_{t_1}^{t_2}\int_{\R^3} \nabla \psi_\eta \cdot A[f_\delta(1-f_\delta)]\nabla f_\delta  \;dvdt\\
    &+\int_{t_1}^{t_2}\int_{\R^3} \nabla\psi_\eta \cdot \nabla a[f_\delta](f_\delta)(1-f_\delta) \;dvdt - \delta\int_{t_1}^{t_2}\int_{\R^3} \nabla \psi_\eta \cdot \nabla f_\delta \;dvdt\\
    =&: -I_1(\eta) + I_2(\eta) - I_3(\eta).
\end{aligned}
\end{equation*}
We now take $\eta \rightarrow 0^+$. For the left hand side, we use conservation of mass from Proposition \ref{prop_step3} to obtain:
\begin{equation*}
\begin{aligned}
 \lim_{\eta \rightarrow 0^+} \int_{t_1}^{t_2}\int_{\R^3} \psi_\eta\partial_t f_\delta  \;dvdt
&= \lim_{\eta \rightarrow 0^+} \int_{t_1}^{t_2}\int_{\R^3}  \log(f_\delta + \eta)\partial_tf_\delta - \log(1-f_\delta + \eta )\partial_tf_\delta  \;dvdt\\
 &= H_1(f_\delta(t_2)) - H_1(f_\delta(t_1)) .
\end{aligned}
\end{equation*}
By the monotone convergence theorem,
\begin{equation*}
\begin{aligned}
\lim_{\eta\rightarrow 0^+}I_1(\eta) &= \int_{t_1}^{t_2}\int_{\R^3}  \left[(f_\delta)^{-1} + (1- f_\delta)^{-1}\right]\nabla f_\delta \cdot A[f_\delta(1-f_\delta)]\nabla f_\delta \;dvdt\\
    &=\int_{t_1}^{t_2}\int_{\R^3}  \nabla \left[\log(f_\delta) - \log(1 - f_\delta)\right] \cdot A[f_\delta(1-f_\delta)]\nabla f_\delta \;dvdt.
\end{aligned}
\end{equation*}
Next, for $I_2$, we decompose further as
\begin{equation*}
\begin{aligned}
    I_2(\eta) =&  \int_{t_1}^{t_2}\int_{\R^3} \left[(f_\delta + \eta)^{-1} + (1- f_\delta + \eta)^{-1}\right]\nabla f_\delta \cdot \nabla a[f_\delta](f_\delta)(1-f_\delta) \;dvdt\\
    =& (1+2\eta) \int_{t_1}^{t_2} \int_{\R^3} \nabla f_\delta \cdot \nabla a[f_\delta] \;dvdt\\
    &+ \eta(1-\eta)\int_{t_1}^{t_2}\int_{\R^3} \left[(f_\delta + \eta)^{-1} + (1- f_\delta + \eta)^{-1}\right]\nabla f_\delta \cdot \nabla a[f_\delta] \;dvdt\\
    =& I_2^1(\eta) + I_2^2(\eta).
\end{aligned}
\end{equation*}
For $I_2^1$, we integrate by parts to obtain
\begin{equation*}
   \lim_{\delta \to 0} I_2^1(\eta) =  \int_{t_1}^{t_2}\int_{\R^3} f_\delta^2\;dvdt.
\end{equation*}
For $I_2^2$, we use $0 \le f_\delta \le 1$ with $f_\delta\in L^\infty([0,T];L^1)$ and $\eta \log \eta \rightarrow 0$ as $\eta \rightarrow 0^+$ to obtain
\begin{equation*}
\begin{aligned}
    |I_2^2(\eta)| &= \eta(1-\eta)\left|\int_{t_1}^{t_2}\int_{\R^3} f_\delta\left[\log(f_\delta + \eta) - \log(1 - f_\delta +\eta)\right] \;dvdt \right|\\
    &\le 2|\log(\eta)|\eta(1-\eta)\int_{t_1}^{t_2}\int_{\R^3} f_\delta\;dvdt \rightarrow 0.
\end{aligned}
\end{equation*}
Finally, we note for $I_3$ that
\begin{equation*}
    I_3(\eta) = \delta\int_{t_1}^{t_2}\int_{\R^3}\left[(f_\delta + \eta)^{-1} + (1- f_\delta + \eta)^{-1}\right] |\nabla f_\delta|^2\;dvdt \ge 0. 
\end{equation*}
Thus, combining our estimates, we have shown
\begin{equation*}
    H_1(f_\delta(t_2)) - H_1(f_\delta(t_1)) \le{ -\int_{t_1}^{t_2}\int_{\R^3}\left( \nabla\left[\log(f_\delta) - \log(1 - f_\delta)\right] \cdot A[f_\delta(1-f_\delta)]\nabla f_\delta - f_\delta^2\right)\;dvdt. }
\end{equation*}
We conclude by noticing that 
\begin{equation*}
\begin{aligned}
    -&{\int_{t_1}^{t_2}\int_{\R^3}\left( \nabla\left[\log(f_\delta) - \log(1 - f_\delta)\right] \cdot A[f_\delta(1-f_\delta)]\nabla f_\delta - f_\delta^2\right)\;dvdt }\\
    &\quad= -\frac{1}{2}\int_{t_1}^{t_2}\int_{\R^3}\int_{\R^3} f_\delta f^*_\delta(1-f_\delta)(1-f^*_\delta) \\
    &\qquad\qquad\times\left(\frac{\Pi(v-v^*)}{|v-v^*|}\left[\frac{\nabla f^*}{f^*(1-f^*)} - \frac{\nabla f}{f(1-f)}\right]\cdot \left[\frac{\nabla f^*}{f^*(1-f^*)} - \frac{\nabla f}{f(1-f)}\right]\right) \ dvdv^*dt\\
    &\quad\le 0.
\end{aligned}
\end{equation*}
\end{proof}
The next lemma contains the coercive estimate we need to pass to the limit $\delta \to 0$.

 
\begin{lemma}[$L^2$ Estimate]\label{step4_lem3}
Suppose $f_{in} \in L^1_3$ with $H_1(f_{in}) < 0$ and $T> 0$. Then, the family $\{f_\delta\}_{0 < \delta < 1}$ satisfies the estimate
\begin{equation*}
\|\partial_t f_\delta\|_{L^2([0,T];H^{-1})} + \|f_\delta\|_{L^\infty([0,T];L^1_3)} + \|\nabla f_\delta\|_{L^2([0,T];L^2)} \le C(\|f_{in}\|_{L^1_3},H_1(f_{in}),T).
\end{equation*}
\end{lemma}

\begin{proof}
We test \eqref{approx_weak}, with $\Psi(v,t) = f_\delta(v,t)\varphi_R(v) {\vv}^3 \chi_{[0,t_0]}(t)$ where $\varphi_R \in C^\infty_c(\R^3)$ is a cutoff function as in Lemma \ref{step2_lem6}. We obtain,
\begin{equation*}
\begin{aligned}
\int_{\R^3} |f_\delta(t_0)|^2 \varphi_R {\vv}^3 - |f_{in}|^2 \varphi_R {\vv}^3  \ dv = &- \int_0^{t_0}\int_{\R^3} \left(A[f_\delta(1-f_\delta)]\nabla f_\delta - \nabla a[f_\delta]f_\delta(1-f_\delta)\right) \cdot \nabla\Psi \;dvdt \\
    &- \delta\int_0^{t_0}\int_{\R^3} \nabla f_\delta \cdot \nabla\Psi \;dvdt.
\end{aligned}
\end{equation*}
We expand the right hand side as
\begin{equation*}
\begin{aligned}
- \int_0^{t_0}\int_{\R^3} &\left(A[f_\delta(1-f_\delta)]\nabla f_\delta - \nabla a[f_\delta]f_\delta(1-f_\delta)\right) \cdot \nabla\Psi \;dvdt-\delta\int_0^{t_0}\int_{\R^3} \nabla f_\delta \cdot \nabla\Psi \;dvdt\\
	= &-\int_0^{t_0}\int_{\R^3} \varphi_R{\vv}^3 A[f_\delta(1-f_\delta)]\nabla f_\delta \cdot \nabla f_\delta \;dvdt -\int_0^{t_0}\int_{\R^3} f_\delta A[f_\delta(1-f_\delta)]\nabla f_\delta \cdot \nabla(\varphi_R{\vv}^3)\;dvdt\\
	&+\int_0^{t_0}\int_{\R^3}\nabla a[f_\delta]f_\delta(1-f_\delta) \cdot f_\delta\nabla(\varphi_R{\vv}^3)\;dvdt +\int_0^{t_0}\int_{\R^3}\nabla a[f_\delta](f_\delta - f_\delta^2) \cdot \varphi_R{\vv}^3\nabla f_\delta \;dvdt\\
	&- \delta\int_0^{t_0}\int_{\R^3} \nabla f_\delta \cdot f_\delta\nabla(\varphi_R{\vv}^3) \;dvdt - \delta\int_0^{t_0}\int_{\R^3} \nabla f_\delta \cdot \varphi_R{\vv}^3\nabla f_\delta \;dvdt\\
	=: &-I_1 - I_2 + I_3 + I_4 - \delta I_5 - \delta I_6. 
\end{aligned}
\end{equation*}
We bound $I_j$ for $2 \le j \le 5$ using the propagation of moments from Lemma \ref{step4_lem1} and the upper bounds on the coefficients $A$ and $\nabla a$ from Lemma \ref{lemA_upper} and \ref{lema_upper}. We will lower bound $I_1$ using Lemma \ref{step4_lem2}. We begin to bound $I_2$ by decomposing further:
\begin{equation*}
\begin{aligned}
I_2 =& \frac{1}{2}\int_0^{t_0}\int_{\R^3} A[f_\delta(1-f_\delta)]\nabla f_\delta^2 \cdot \nabla(\varphi_R{\vv}^3) \;dvdt \\
	=& \frac{1}{2}\int_0^{t_0}\int_{\R^3} \nabla \cdot \left(A[f_\delta(1-f_\delta)]f_\delta^2\right) \cdot \nabla(\varphi_R{\vv}^3)\;dvdt
	\\&- \frac{1}{2}\int_0^{t_0}\int_{\R^3} (\nabla \cdot A[f_\delta(1-f_\delta)])f_\delta^2 \cdot \nabla(\varphi_R{\vv}^3) \;dvdt \\
	=:& I_2^1 - I_2^2.
\end{aligned}
\end{equation*}
For $I_2^1$, integration by parts, Lemma \ref{lemA_upper}, and $|\nabla^2(\varphi_R{\vv}^3)| \lesssim {\vv}$ imply
\begin{equation*}
|I_2^1| \le \|A[f_\delta(1-f_\delta)\|_{L^\infty([0,T]\times\R^3)}\|f_\delta \nabla^2(\varphi_R{\vv}^3)\|_{L^1([0,T]\times\R^3)} \lesssim C(\|f_{in}\|_{L^1_1},T).
\end{equation*}
For $I_2^2$, obtain by Lemma \ref{lemA_upper}, and ${|}\nabla(\varphi_R{\vv}^3){|} \lesssim {\vv}^2$,
\begin{equation*}
|I_2^2| \le \|\nabla \cdot A[f_\delta(1-f_\delta)\|_{L^\infty([0,T]\times\R^3)}\|f_\delta \nabla(\varphi_R {\vv}^3)\|_{L^1([0,T]\times\R^3)} \lesssim C(\|f_{in}\|_{L^1_2},T).
\end{equation*}
Piecing together, we obtain
\begin{equation*}
|I_2| \le C(\|f_{in}\|_{L^1_2},T).
\end{equation*}
For $I_3$, we directly use the estimates from Lemma \ref{step4_lem1}, \ref{lemA_upper}, and \ref{lema_upper} and ${|}\nabla(\varphi_R{\vv}^3){|} \lesssim {\vv}^2$ to obtain
\begin{equation*}
\begin{aligned}
|I_3| \le \|\nabla a[f_\delta\|_{L^\infty([0,T]\times\R^3)} \|f_\delta \nabla(\varphi_R{\vv}^3)\|_{L^1([0,T];L^1)} \le C(\|f_{in}\|_{L^1_2},T).
\end{aligned}
\end{equation*}
Next, for $I_4$, we integrate by parts and recall $-\Delta a[f] = f$ to decompose further:
\begin{equation*}
\begin{aligned}
I_4 &= \int_0^{t_0}\int_{\R^3}\nabla a[f_\delta]\cdot \varphi_R{\vv}^3\nabla\left(\frac{f_\delta^2}{2} - \frac{f_\delta^3}{3}\right) \; dvdt\\
	&= -\int_0^{t_0}\int_{\R^3}\nabla a[f_\delta]\cdot \nabla(\varphi_R{\vv}^3)\left(\frac{f_\delta^2}{2} - \frac{f_\delta^3}{3}\right) \;dvdt+ \int_0^{t_0}\int_{\R^3}\varphi_R{\vv}^3\left(\frac{f_\delta^3}{2} - \frac{f_\delta^4}{3}\right)\;dvdt\\
	&= -I_4^1 + I_4^2.
\end{aligned}
\end{equation*}
We bound $I_4^1$ using Lemma \ref{lema_upper} and $|\nabla(\varphi_R(v){\vv}^3)| \lesssim {\vv}^2$, to obtain
\begin{equation*}
|I_4^1| \le \|\nabla a[f_\delta]\|_{L^\infty([0,T]\times \R^3)}\|f_\delta\nabla(\varphi_R{\vv}^3)\|_{L^1([0,T]\times \R^3)} \le C(\|f_{in}\|_{L^1_2},T).
\end{equation*}
For $I_4^2$, we bound using $0 \le f_\delta \le 1$ so that
\begin{equation*}
|I_4^2| \le \|\varphi_R {\vv}^3 f_\delta\|_{L^1([0,T]\times \R^3)}\|f_\delta^2/2 - f_\delta^3/3\|_{L^\infty([0,T]\times \R^3)} \le C(\|f_{in}\|_{L^1_3},T).
\end{equation*}
Hence,
\begin{equation*}
|I_4| \le C(\|f_{in}\|_{L^1_3},T).
\end{equation*}
Using $|\nabla^2(\varphi_R(v){\vv}^3)| \lesssim {\vv}$ , integration by parts yields
\begin{equation*}
|I_5| \le \|f_\delta\|_{L^\infty([0,T]\times \R^3)}\|f_\delta\nabla^2(\varphi_R(v){\vv}^3)\|_{L^1([0,T]\times\R^3)} \le C(\|f_{in}\|_{L^1_1},T).
\end{equation*}
Finally, we note $I_6 \ge 0$ and by Lemma \ref{step4_lem2},
\begin{equation*}
I_1 \ge C(\|f_{in}\|_{L^1_2},H_1(f_{in}),T)\int_0^{t_0}\int_{\R^3} \frac{\varphi_R {\vv}^3}{1 + |v|^3} |\nabla f_\delta|^2 \ dvdt.
\end{equation*}
Summarizing, we obtain
\begin{align*}
\int_{\R^3} \varphi_R {\vv}^3 f_\delta(t_0)\;dv &+ C(\|f_{in}\|_{L^1_2},H_1(f_{in}),T)\int_0^{t_0}\int_{\R^3} \varphi_R |\nabla f_\delta |^2 \;dvdt\\
 &\le \int_{\R^3} \varphi_R {\vv}^3 f_{in} \;dv+ C(\|f_{in}\|_{L^1_3},T).
\end{align*}
Letting $R \rightarrow \infty$, applying the monotone convergence theorem, and taking a supremum over $t_0\in [0,T]$ yield the desired bound on $\nabla f_\delta$.

Next, we test \eqref{approx_weak} with an arbitrary test function $\Phi\in L^2([0,T];H^1)$ and, by duality, obtain a bound on $\partial_tf$. In particular, we have
\begin{equation*}
\|\partial_t f_\delta\|_{L^2([0,T];H^{-1})} = \sup_{\|\Phi\|_{L^2([0,T];H^1)} \le 1} \int_0^T\int_{\R^3} \left(A[f_\delta(1-f_\delta)]\nabla f_\delta - \nabla a[f_\delta]f_\delta(1-f_\delta) + \delta \nabla f_\delta\right) \cdot \nabla \Phi \; dvdt.
\end{equation*}
Since
\begin{equation*}
\int_0^T\int_{\R^3} A[f_\delta(1-f_\delta)]\nabla f_\delta \cdot \nabla \Phi\;dvdt  \le \|A[f_\delta(1-f_\delta)]\|_{L^\infty([0,T]\times\R^3)}\|\nabla f\|_{L^2([0,T];L^2)}\|\nabla\Phi\|_{L^2([0,T];L^2)},
\end{equation*}
and
\begin{equation*}
\int_0^T\int_{\R^3} \nabla a[f_\delta]f_\delta(1-f_\delta) \cdot \nabla \Phi\;dvdt  \le \|\nabla a[f_\delta]\|_{L^2([0,T];L^2)}\|\nabla \Phi\|_{L^2([0,T];L^2)},
\end{equation*}
we conclude
\begin{equation*}
\begin{aligned}
|\partial_t f_\delta\|_{L^2([0,T];H^{-1})} &\le C(\|f_{in}\|_{L^1}) \sup_{\|\Phi\|_{L^2([0,T];H^1)} \le 1} \|f\|_{L^2([0,T];H^1)}\|\nabla \Phi\|_{L^2([0,T];L^2)}\\
&\le C(\|f_{in}\|_{L^1_3},H_1(f_{in}),T).
\end{aligned}
\end{equation*}
This completes the proof. 
\end{proof}

In the next lemma we state a weighted $L^2$ estimate, proved via a slight modification to the $L^2$-estimate in Lemma \ref{step4_lem3}. 

\begin{lemma}\label{reg_lemma1}
Let $f$ be any weak solution to \eqref{weak_FDL} as in Theorem  \ref{thm_existence} with initial data $f_{in}$ as described above. Then, {for every $m\geq 3$,}
\begin{equation*}
    \sup_{(0, T)} \|f(\cdot )\|_{L^2_m}^2 + \int_{0}^{T}\|\nabla f(t)\|_{L^2_{m-3}}^2 \ dt \le C(f_{in}{,m,T}).
\end{equation*}
\end{lemma}

\begin{proof}
We test \eqref{weak_FDL} with $\varphi = \brak{v}^m f$, and estimate the resulting terms as in the proof of Lemma \ref{step4_lem3}. 
\end{proof}

\begin{flushleft}
\underline{Proof of Theorem \ref{thm_existence}}\\
\end{flushleft}

Fix $T > 0$, $f_{in}$ with $0 \le f_{in} \le 1$, and $f_{in} \in L^1_3$ and fix some sequence $\delta_n \rightarrow 0^+$ and let $f_n(v,t)$ be the solutions with $\delta = \delta_n$ constructed in Proposition \ref{prop_step3}. Then, the uniform-in-$\delta$ estimates from Lemma \ref{step4_lem1} and Lemma \ref{step4_lem3} together with Aubin-Lions {Lemma} imply that we may assume that $f_n \rightarrow f$ for some limit $f$ in the following topologies:
\begin{itemize}
\item Weak star in $L^\infty([0,T]\times\R^3)$,
\item Weakly in $L^2([0,T];H^1)$,
\item Weak star in $L^\infty([0,T];L^2)$,
\item Strongly in $L^p([0,T];L^q)$ for any $1\le p \le 2$ and $1\le q < 6$.
\end{itemize} 
Furthermore, we may also assume $f_n \rightarrow f$ pointwise almost everywhere on $[0,T]\times \R^3$. Therefore, by Fatou's lemma, it follows that for almost every $t\in [0,T]$,
\begin{equation}
\|f(t)\|_{L^1_3} \le C(\|f_{in}\|_{L^1_3},T).
\end{equation}
Note also that the weak star convergence in $L^\infty$ is sufficient to guarantee $0 \le f \le 1$. 

Next, since each $f_n$ solves \eqref{FDL_approx} on the time interval $[0,T]$, for any $\Phi \in C^\infty_c([0,T)\times \R^3)$, we have
\begin{align*}
\int_0^T\int_{\R^3} f_n \partial_t \Phi \;dvdt & - \int_{\R^3} f_{in}\Phi(0) \;dv =\\ &- \int_0^T\int_{\R^3} \left(A[f_n(1-f_n)]\nabla f_n - \nabla a[f_n]f_n(1-f_n) + \delta_n \nabla f_n\right) \cdot \nabla \Phi \;dvdt.
\end{align*}
We conclude the proof of existence by following the same steps as in the proof of Proposition \ref{prop_step3}. 

To show uniqueness of solution, we assume by contradiction that there exist two solutions $f$ and $g$. Their difference 
$$
w := f-g,
$$
is identically zero at $t=0$ and solves the following weak formulation: 
\begin{align*}
\int_{0}^T  \langle \varphi ,\partial_t w\rangle_{H^1,H^{-1}} \;dt = &- \int_0^T\int_{\R^3} A[f(1-f)]\nabla w \cdot \nabla \varphi \;dvdt \\
&+ \int_0^T\int_{\R^3} A[fw+(g-1)w]\nabla g \cdot \nabla \varphi \;dvdt\\
& + \int_0^T\int_{\R^3} f(1-f)  \nabla a[w]\cdot \nabla \varphi \;dvdt \\
&- \int_0^T\int_{\R^3} (fw+(g-1)w)  \nabla a[g]\cdot \nabla \varphi \;dvdt.
\end{align*}
We consider $\varphi = w \langle v \rangle ^{2m}$ for some $m > \frac{3}{2}$, and get  
\begin{align*}
  \frac{1}{2} \int_{\R^3} w^2(t) \langle v \rangle ^{2m}  \;dv= &- \int_0^T\int_{\R^3} A[f(1-f)]\nabla w \cdot \nabla (w \langle v \rangle ^{2m}) \;dvdt \\
&+ \int_0^T\int_{\R^3} A[fw+(g-1)w]\nabla g \cdot \nabla (w \langle v \rangle ^{2m}) \;dvdt\\
& + \int_0^T\int_{\R^3} f(1-f)  \nabla a[w]\cdot \nabla (w \langle v \rangle ^{2m} )\;dvdt \\
&-  \int_0^T\int_{\R^3} (fw+(g-1)w)  \nabla a[g]\cdot \nabla (w \langle v \rangle ^{2m}) \;dvdt\\
=:& \;  I_1+I_2+I_3+I_4.
\end{align*}
The term $I_1$ is estimated with Young's inequality:
\begin{align*}
I_1 =& -  \int_0^T\int_{\R^3} \frac{A[f(1-f)]}{ \langle v \rangle ^{2m}} |\nabla w \langle v \rangle ^{2m}|^2 \;dvdt + \int_0^T\int_{\R^3} \frac{A[f(1-f)]w}{ \langle v \rangle ^{2m}} \nabla w \langle v \rangle ^{2m} \cdot \nabla \langle v \rangle ^{2m} \;dvdt \\
\le&\; - (1-\delta) \int_0^T\int_{\R^3} \frac{A[f(1-f)]}{ \langle v \rangle ^{2m}} |\nabla w \langle v \rangle ^{2m}|^2 \;dvdt  + C(m,\delta, f_{in}) \int_0^T\int_{\R^3} |w|^2 \langle v \rangle ^{2m}\;dvdt.
\end{align*}
Similarly, 
\begin{align*}
I_2 \le &\;C(f_{in}) \int_0^T\int_{\R^3} A[w]\nabla g \cdot \nabla w \langle v \rangle ^{2m} \;dvdt \\
\le&\; \delta  \int_0^T\int_{\R^3} \frac{ |\nabla w \langle v \rangle ^{2m}|^2 }{\langle v \rangle ^{2m}(1+|v|)^3}\;dvdt + \frac{1}{\delta}  \int_0^T\int_{\R^3}\langle v \rangle ^{2m} (1+|v|)^3 |A[w]|^2 |\nabla g|^2\;dvdt\\
\le &\; \delta  \int_0^T\int_{\R^3}  \frac{|A[f(1-f)]|}{ \langle v \rangle ^{2m}} { |\nabla w \langle v \rangle ^{2m}|^2 }\;dvdt + \frac{1}{\delta}  \int_0^T \|w\|^2_{L^2_{2m}} \int_{\R^3} (1+|v|)^{3+2m} |\nabla g|^2\;dvdt,
\end{align*}
using the bound from below for $A[\cdot]$ and the bound from above in Lemma \ref{lemA_upper}: 
$$
\|A[h]\|_{L^\infty} \le C\|h\|_{L^1}^{2/3}\|h\|_{L^2}^{1/3} \le C \|h\|_{L^2_{2m}}, \quad \textrm{for all} \; m>\frac{3}{2}.
$$
H\"older and Hardy-Littlewood-Sobolev inequalities applied to $I_3$ lead to 
\begin{align*}
I_3 \le & \int_0^T \| \nabla a[w]\|_{L^6}\left( \int_{\R^3} f^{6/5}  |\nabla w \langle v \rangle ^{2m}|^{6/5}\;dv \right)^{\frac{5}{6}}dt \\
\le & \; \frac{1}{\delta} \int_0^T  \|w\|^2_{L^2} \;dt + \delta  \int_0^T  \left( \int_{\R^3}f(1+|v|)^{9/2+3m}\;dv\right)^{2/3} \int_{\R^3} \frac{ |\nabla w \langle v \rangle ^{2m}|^2 }{\langle v \rangle ^{2m}(1+|v|)^3}\;dvdt  \\
\le&\; \frac{1}{\delta} \int_0^T  \|w\|^2_{L^2_{2m}} \;dt + \tilde\delta  \int_0^T \int_{\R^3}  \frac{A[f(1-f)]}{ \langle v \rangle ^{2m}}{ |\nabla w \langle v \rangle ^{2m}|^2 }\;dvdt ,
\end{align*}
using Lemma \ref{reg_lemma1}} to bound the $9/2+3m$ moments of $f$ (uniformly in time). Similarly,
\begin{align*}
I_4 \le \frac{1}{\delta} \|\nabla a[g](1+|v|)^{3/2}\|^2_{L^\infty_{v,t} }\int_0^T \int_{\R^3} w^2 \langle v \rangle ^{2m}\;dvdt + {\delta}   \int_0^T\int_{\R^3} \frac{ |\nabla w \langle v \rangle ^{2m}|^2 }{ \langle v \rangle ^{2m}(1+|v|)^3}\;dvdt. 
\end{align*}
We briefly show how the term  $\|\nabla a[g](1+|v|)^{3/2}\|^2_{L^\infty_{v,t} }$ is uniformly bounded. Let $|v|$ be large enough. For $s>3$, H\"older inequality yields:
\begin{align*}
|\nabla a[f]|  \le &\int_{\R^3}\frac{f(y)}{|v-y|^2}\;dy \le  |v|^{\frac{3-2s'}{s'}}\left(\int_{B_{\frac{|v|}{2}}(|v|)} f^s\;dy\right)^{1/s}+ \frac{1}{|v|^2} \|f\|_{L^1(\R^3)} \\
 \le & \;c \frac{|v|^{\frac{3-2s'}{s'}}}{(1+|v|)^{\lambda/s}}\left(\int_{\R^3} f^s(1+|y|)^{\lambda}\;dy\right)^{1/s}+ \frac{1}{|v|^2} \|f\|_{L^1(\R^3)},
\end{align*}
 with $\frac{1}{s}+\frac{1}{s'}=1$ and $s'<3/2$. The choice of $\lambda$ so that $\frac{3-s'}{s'} - \frac{\lambda}{s}=-2$ leads to the desired estimate. 
 
 Combining the estimates for $I_1, ...,I_4$ and choosing $\delta$ small enough, one gets
 \begin{align*}
\frac{1}{2} \int_{\R^3} w^2(T) \langle v \rangle ^{2m}  \;dv  \le C  \int_0^T   \|w\|^2_{L^2_{2m}}\left( 1+  \int_{\R^3} (1+|v|)^{3+2m} |\nabla g|^2\;dv \right)\;dt.
\end{align*}
Since $\int_0^T  \int_{\R^3} (1+|v|)^{3+2m} |\nabla g|^2\;dvdt \le C$ thanks to Lemma \ref{reg_lemma1}, Gronwall's inequality implies that $w(t) = 0$ for all $t\ge0$. This concludes the proof of the theorem. 

\section{Regularity of Weak Solutions}\label{sec_reg}

In this section, we prove Theorem \ref{thm_regularity}. Throughout this section we consider initial data $f_{in}$ such that  $0\le f_{in} \le 1$, $\|f_{in}\|_{L^1 \cap L^2_m} < \infty$ for a general  $m \ge {{3}}$, and $H_1(f_{in}) < 0$.  The exact value of $m$ needed for Theorem \ref{thm_regularity} is determined in Lemma \ref{holder_lem}.

As a first step, we use Lemma \ref{reg_lemma1} to show  that weak solutions of \eqref{LFD} instantaneously regularize and belong to weighted $L^\infty (t,T, H^1)$ and weighted $L^2(t,T, H^2)$ for any $t>0$.
\begin{lemma}\label{reg_lemma2}
Let $f$ be any weak solution to \eqref{weak_FDL} as in Theorem  \ref{thm_existence}. For any $t>0$ and $m\ge 3$, we have  
\begin{equation*}
    \sup_{( t, T)} \|\nabla f(\cdot)\|_{L^2_m}^2 + \int_{t}^{T}\|\nabla^2 f(s)\|_{L^2_{m-3}}^2 \ ds \le C(f_{in})\left( 1 + \frac{1}{t}\right).
\end{equation*}
\end{lemma}

\begin{proof}
{Fix $i=1,2,3$ arbitrary.}
We first recall the notation for divided differences, 
$$\diff_h g := \frac{g(v + e_ih) - g(v)}{h},$$
for which the following discrete product formula holds,
$$\diff_h(fg) = g\diff_h f + f\diff_h g + h\diff_h f \diff_h g.$$
We test   \eqref{weak_FDL}  with $\psi(v,t) = -\chi_{[t_1,t_2]}(t)\diff_{-h}\left(\brak{v}^m\diff_h f\right)$ and obtain
\begin{align*}
    -\int_{t_1}^{t_2}\int_{\R^3} \partial_t f\diff_{-h}&\left(\brak{v}^m\diff_h f\right)\;dvdt \\
    &= \int_{t_1}^{t_2}\int_{\R^3} \nabla \diff_{-h}\left(\brak{v}^m\diff_h f\right) \cdot \left(A[f(1-f)]\nabla f - \nabla a[f]f(1-f)\right) \;dvdt.
\end{align*}
On the left hand side, we perform a discrete integration by parts:
\begin{equation*}
\begin{aligned}
	-\int_{t_1}^{t_2}\int_{\R^3} \partial_t f\diff_{-h}\left(\brak{v}^m\diff_h f\right) \;dvdt &= \frac{1}{2}\int_{t_1}^{t_2}\int_{\R^3} \brak{v}^m \partial_t\left[\diff_h f\right]^2 \;dvdt  \\
	&= \frac{1}{2}\|\brak{v}^m\diff_hf(t_2)\|_{L^2}^2 - \frac{1}{2}\|\brak{v}^m\diff_hf(t_1)\|_{L^2}^2.
\end{aligned}
\end{equation*}
We decompose the right hand side as
{\small{\begin{equation*}
\begin{aligned}
	RHS		&= -\int_{t_1}^{t_2}\int_{\R^3} \brak{v}^m\nabla\diff_h f \cdot A[f(1-f)]\nabla\diff_h f  \;dvdt - \int_{t_1}^{t_2}\int_{\R^3} \brak{v}^m\nabla\diff_h f \cdot \left(\diff_h A[f(1-f)]\nabla f \right)  \;dvdt  \\
		&\qquad - \int_{t_1}^{t_2}\int_{\R^3} \brak{v}^m\nabla\diff_h f \cdot \left(h\diff_h A[f(1-f)]\nabla \diff_h f)\right) \;dvdt  + \int_{t_1}^{t_2}\int_{\R^3} \brak{v}^m\nabla\diff_h f \cdot \left(\diff_h\nabla a[f]f(1-f)\right)  \;dvdt  \\
		&\qquad + \int_{t_1}^{t_2}\int_{\R^3} \brak{v}^m\nabla\diff_h f \cdot \left(\nabla a[f]\diff_h(f(1-f))\right)  \;dvdt  + \int_{t_1}^{t_2}\int_{\R^3} \brak{v}^m\nabla\diff_h f \cdot \left(h\diff_h\nabla a[f]\diff_h[f(1-f)]\right)  \;dvdt  \\
		&\qquad - \int_{t_1}^{t_2}\int_{\R^3} \nabla\brak{v}^m\diff_h f \cdot \left(\diff_h A[f(1-f)]\nabla f \right) \;dvdt  - \int_{t_1}^{t_2}\int_{\R^3} \nabla\brak{v}^m\diff_h f \cdot \left(A[f(1-f)]\nabla \diff_h f \right)  \;dvdt  \\
		&\qquad - \int_{t_1}^{t_2}\int_{\R^3} \nabla\brak{v}^m\diff_h f \cdot \left(h\diff_h A[f(1-f)]\nabla \diff_hf \right) \;dvdt   + \int_{t_1}^{t_2}\int_{\R^3} \nabla\brak{v}^m\diff_h f \cdot \left(\diff_h\nabla a[f]f(1-f)\right)  \;dvdt  \\
		&\qquad + \int_{t_1}^{t_2}\int_{\R^3} \nabla\brak{v}^m\diff_h f \cdot \left(\nabla a[f]\diff_h(f(1-f))\right)  \;dvdt  + \int_{t_1}^{t_2}\int_{\R^3} \nabla\brak{v}^m\diff_h f \cdot \left(h\diff_h\nabla a[f]\diff_h(f(1-f))\right)  \;dvdt  \\
		&=: \sum_{j=1}^{12}I_j.
\end{aligned}	
\end{equation*}}}
For $I_1$ we use Lemma \ref{lemA_lower} to obtain
\begin{equation*}
	I_1 \le - C(f_{in})\int_{t_1}^{t_2}\int_{\R^3} \brak{v}^{m-3} |\nabla \diff_hf|^2 \;dvdt.
\end{equation*}
Next, for any $\delta>0$, we upper bound $I_2$ using Young's inequality and $\|\nabla A[f(1-f)\|_{L^\infty} \le C(f_{in})$,
\begin{equation*}
\begin{aligned}
	|I_2| &\lesssim \delta\int_{t_1}^{t_2}\int_{\R^3} \brak{v}^{m-3}|\nabla \diff_h f|^2\;dvdt  + \delta^{-1}\int_{t_1}^{t_2}\int_{\R^3} \brak{v}^{m+3}|\nabla f|^2 \;dvdt \\
		&\lesssim \delta\int_{t_1}^{t_2}\int_{\R^3} \brak{v}^{m-3}|\nabla \diff_h f|^2\;dvdt   + \delta^{-1}\|\nabla f\|_{L^2([t_1,t_2];L^2_{m+3})}^2.
\end{aligned}
\end{equation*}
In the same way, we bound $I_3$. We bound $I_4$ using Young's inequality and $\|\nabla^2 a[f]\|_{L^p} \lesssim \|f\|_{L^p}$ for $1 < p < \infty$ by Calderon-Zygmund (see Chapter 9.4 in \cite{GT98}),
\begin{equation*}
\begin{aligned}
	|I_4| &\lesssim \delta\int_{t_1}^{t_2}\int_{\R^3} \brak{v}^{m-3}|\nabla \diff_h f|^2\;dvdt  + \delta^{-1}\int_{t_1}^{t_2}\int_{\R^3} \brak{v}^{m+3}|\diff_h \nabla a[f]|^2f^2 \;dvdt \\
		&\lesssim \delta\int_{t_1}^{t_2}\int_{\R^3} \brak{v}^{m-3}|\nabla \diff_h f|^2 \;dvdt + \delta^{-1}\int_{t_1}^{t_2}\int_{\R^3} \left(|\diff_h\nabla a[f]|^{\frac{2(m+6)}{3}} + f^{\frac{2(m+6)}{m+3}}\brak{v}^{m+6}\right) \;dvdt \\
		&\lesssim \delta\int_{t_1}^{t_2}\int_{\R^3} \brak{v}^{m-3}|\nabla \diff_h f|^2 \;dvdt+ \delta^{-1}\|f\|_{L^2([t_1,{t_2}];L^2_{m+6})}^2.
\end{aligned}
\end{equation*}
For $\delta > 0$, we bound $I_5$, using Young's inequality and Lemma \ref{lema_upper}, as
\begin{equation*}
\begin{aligned}
	|I_5| &\lesssim \delta\int_{t_1}^{t_2}\int_{\R^3} \brak{v}^{m-3}|\nabla \diff_h f|^2 \;dvdt + \delta^{-1}\int_{t_1}^{t_2}\int_{\R^3} \brak{v}^{m+3}|\diff_h f(1-f)|^2 \;dvdt\\
		&\lesssim \delta\int_{t_1}^{t_2}\int_{\R^3} \brak{v}^{m-3}|\nabla \diff_h f|^2 \;dvdt + \delta^{-1}\|\diff_h f\|_{L^2([t_1,t_2];L^2_{m+3})}^2 .
\end{aligned}
\end{equation*} 
Again, we bound $I_6$ in a similar manner to $I_4$ and $I_5$. For $I_7$, we use $|\nabla\brak{v}^m)| \lesssim \brak{v}^{m-1}$, Lemma \ref{lemA_upper}, and $\|\diff_h f\|_{L^2_m} \lesssim \|\nabla f\|_{L^2_m}$ (via a simple modification to Proposition IX.9(iii)  in \cite{Brezis}), to obtain
\begin{equation*}
\begin{aligned}
	|I_7| &\lesssim \int_{t_1}^{t_2}\int_{\R^3} \brak{v}^{m-1}|\nabla f|^2 \;dvdt.		
\end{aligned}
\end{equation*}
Next, for $I_8$, we integrate by parts and use Lemma \ref{lemA_upper} and $|\nabla^2\varphi_R(|v|)\brak{v}^m| \lesssim \brak{v}^{m-2}$ to obtain
\begin{equation*}
	\begin{aligned}
		I_8 =& -\frac{1}{2}\int_{t_1}^{t_2}\int_{\R^3} \nabla \brak{v}^m \cdot A[f(1-f)]\nabla (\diff_h f)^2 \;dvdt \\
			=& \frac{1}{2}\int_{t_1}^{t_2}\int_{\R^3} \nabla \brak{v}^m \cdot (\nabla \cdot A)[f(1-f)](\diff_h f)^2\;dvdt  \\
			&+ \frac{1}{2}\int_{t_1}^{t_2}\int_{\R^3} \nabla^2\brak{v}^m : A[f(1-f)](\diff_h f)^2\;dvdt \\
			\lesssim & \|{\nabla}f\|_{L^2([t_1,t_2];L^2_{m-1})}^2 + \|{\nabla}f\|_{L^2([t_1,t_2];L^2_{m-2})}^2.
	\end{aligned}
\end{equation*}
We bound $I_9$ in a similar manner to $I_7$ and $I_8$. For $I_{10}$, we use Young's inequality and Calderon-Zygmund, to obtain
\begin{equation*}
	\begin{aligned}
		|I_{10}| &\lesssim  \|\diff_h f\|_{L^2([t_1,t_2];L^2_{m-1})}^2 + \int_{t_1}^{t_2}\int_{\R^3}\brak{v}^{m-1}|\diff_h\nabla a[f]|^2f^2 \;dvdt \\
			&\lesssim \|\diff_h f\|_{L^2([t_1,t_2];L^2_{m-1})}^2 + \|f\|_{L^2([t_1,t_2];L^2_m)}^2.
	\end{aligned}
\end{equation*}
For $I_{11}$, we use $\nabla a[f] \in L^\infty([0,T]\times \R^3)$ from Lemma \ref{lema_upper} to obtain
\begin{equation*}
	\begin{aligned}
		|I_{11}| &\lesssim \int_{t_1}^{t_2}\int_{\R^3} \brak{v}^{m-1}|\diff_h f|^2\;dvdt \lesssim \|\diff_hf\|_{L^2([t_1,t_2];L^2_{m-1})}^2.
	\end{aligned}
\end{equation*}
Finally, $I_{12}$ is bounded similarly to $I_{10}$ and $I_{11}$. Thus, we have shown (using once more that $\|\diff_h f\|_{L^2_m} \lesssim \|\nabla f\|_{L^2_m)}$),
\begin{equation*}
\begin{aligned}
	\int_{\R^3}\varphi_R(|v|)&\brak{v}^m \left[\diff_h f(t_2)^2 - \diff_h f(t_1)^2\right] \;dv+ (C - \delta)\int_{t_1}^{t_2}\int_{\R^3} \varphi_R(|v|)\brak{v}^{m-3}|\diff_h \nabla f|^2 \;dvdt\\
	 &\lesssim (1+\delta^{-1})\left(\|\nabla f\|_{L^2([t_1,t_2];L^2_{m-2} \cap L^2_{m+3})}^2 + \|f\|_{L^2([t_1,t_2];L^2_m \cap L^2_{m+6})}^2\right), 
\end{aligned}
\end{equation*}
where the implicit constants depend only on $f_{in},\ T, $ and $m$.
Now, taking $\delta < C/2$ and taking $h \rightarrow 0^+$, {$R\to\infty$,} we see $\nabla f$ is weakly differentiable and
\begin{equation*}
\begin{aligned}
	\int_{\R^3}\brak{v}^m &|\nabla f(t_2)|^2 \;dv+ \int_{t_1}^{t_2}\int_{\R^3} \brak{v}^{m-3}|\nabla^2 f|^2 \;dvdt \lesssim  \int_{\R^3}\brak{v}^m|\nabla f(t_1)|^2 \;dv\\
		&\qquad+ \left(\|\nabla f\|_{L^2([t_1,t_2];L^2_{m-2} \cap L^2_{m+3})}^2 + \|f\|_{L^2([t_1,t_2];L^2_m \cap L^2_{m+6})}^2\right). 
\end{aligned}
\end{equation*}
Next, taking a {supremum} over $t_2$ in $[t,T]$ and an average over $t_1\in [0,t]$, and applying Lemma \ref{reg_lemma1}, we get
\begin{equation*}
	\begin{aligned}
		\sup_{(t,T)} &\|\nabla f(\cdot)\|_{L^2_m}^2 + \int_{t}^T\|\nabla^2 f(s)\|_{L^2_{m-3}}^2\;ds\\
		&\lesssim {\frac{1}{t}C(f_{in},T,m)+}
		\|\nabla f\|_{L^2([0,T];L^2 \cap L^2_{m+3})}^2 + \|f\|_{L^2([0,T];L^2 \cap L^2_{m+6})}^2\\
		&\le C(f_{in},T,m) \left(1 + \frac{1}{t}\right).\\
	\end{aligned}
\end{equation*}
This concludes the proof of the lemma.
\end{proof}

Next, we show how to control the $L^\infty (t,T,H^2) \cap L^2(t,T,H^3)$-regularity of $f$:
\begin{lemma}\label{reg_lemma3}
Let $f$ be any weak solution to \eqref{weak_FDL} as in Theorem  \ref{thm_existence} with initial data as described at the beginning of this section. For any $t>0$ and $m\ge 3$, we have  
\begin{equation*}
    \sup_{( t, T)} \|\nabla^2 f(\cdot)\|_{L^2_m}^2 + \int_{t}^{T}\|\nabla^3 f(s)\|_{L^2_{m-3}}^2 \; ds \le C(f_{in})\left( 1 + \frac{1}{t^2}\right).
\end{equation*}
\end{lemma}

\begin{proof}
	Thanks to Lemma \ref{reg_lemma2}, we can take
	$$\psi(v,t) ;= \chi_{[t_1,t_2]}\diff_{-h}\partial_{v_i} \left(\brak{v}^m\diff_{h} f_{v_i}\right),$$ 
	as test function for \eqref{weak_FDL}, and obtain
\begin{equation*}
\begin{aligned}
    \int_{t_1}^{t_2}\int_{\R^3} \partial_t f\partial_{v_i}\diff_{-h}\left(\brak{v}^m\diff_h f_{v_i}\right)  \;dvdt&= -\int_{t_1}^{t_2}\int_{\R^3} \nabla \partial_{v_i}\diff_{-h}\left(\brak{v}^m\diff_h f_{v_i}\right)\cdot \\
    &\qquad \cdot \left(A[f(1-f)]\nabla f - \nabla a[f]f(1-f)\right) \;dvdt.
\end{aligned}
\end{equation*}
On the left hand side, we perform one discrete integration by parts and one standard integration by parts and get
\begin{equation*}
\begin{aligned}
	LHS &= \frac{1}{2}\int_{t_1}^{t_2}\int_{\R^3} \brak{v}^m \partial_t\left[\diff_h f_{v_i}\right]^2  \;dvdt= \frac{1}{2}\|\brak{v}^m\diff_hf_{v_i}(t_2)\|_{L^2}^2 - \frac{1}{2}\|\brak{v}^m\diff_hf_{v_i}(t_1)\|_{L^2}^2.
\end{aligned}
\end{equation*}
We also perform discrete and standard integration by parts to decompose the right hand side as
{\small{ \begin{align*}
	RHS 
		=& -\int_{t_1}^{t_2}\int_{\R^3} \brak{v}^m\nabla\diff_h f_{v_i} \cdot A[f(1-f)]\nabla\diff_h f_{v_i} \;dvdt \\
		&- \int_{t_1}^{t_2}\int_{\R^3} \brak{v}^m\nabla\diff_h f_{v_i} \cdot \left(\diff_h A[f(1-f)]\nabla f_{v_i} + \partial_{v_i} A[f(1-f)]\diff_{h}\nabla f \right) \;dvdt\\
		&- \int_{t_1}^{t_2}\int_{\R^3} \brak{v}^m\nabla\diff_h f_{v_i} \cdot \left(\partial_{v_i}\diff_h A[f(1-f)]\nabla f\right) \;dvdt + \int_{t_1}^{t_2}\int_{\R^3} \brak{v}^m\nabla\diff_h f_{v_i} \cdot \left(\diff_h \partial_{v_i}\nabla a[f]f(1-f)\right)\;dvdt \\
		& + \int_{t_1}^{t_2}\int_{\R^3} \brak{v}^m\nabla\diff_h f_{v_i} \cdot \left(\partial_{v_i}\nabla a[f]\diff_h(f(1-f)) + \diff_h\nabla a[f]\partial_{v_i}(f(1-f))\right) \;dvdt \\
		& + \int_{t_1}^{t_2}\int_{\R^3} \brak{v}^m\nabla\diff_h f_{v_i} \cdot \left(\nabla a[f]\partial_{v_i}\diff_h[f(1-f)]\right)\;dvdt - \int_{t_1}^{t_2}\int_{\R^3} \nabla\brak{v}^m\diff_h f_{v_i} \cdot \left(A[f(1-f)]\nabla\diff_h f_{v_i} \right) \;dvdt \\
		& - \int_{t_1}^{t_2}\int_{\R^3} \nabla\brak{v}^m\diff_h f_{v_i} \cdot \left(\diff_h A[f(1-f)]\nabla f_{v_i} + \partial_{v_i} A[f(1-f)]\diff_{h}\nabla f \right) \;dvdt \\
		& - \int_{t_1}^{t_2}\int_{\R^3} \nabla\brak{v}^m\diff_h f_{v_i} \cdot \left(\partial_{v_i}\diff_h A[f(1-f)]\nabla f \right)\;dvdt + \int_{t_1}^{t_2}\int_{\R^3} \nabla\brak{v}^m\diff_h f_{v_i} \cdot \left(\diff_h \partial_{v_i}\nabla a[f]f(1-f)\right) \;dvdt \\
		& + \int_{t_1}^{t_2}\int_{\R^3} \nabla\brak{v}^m\diff_h f_{v_i} \cdot \left(\partial_{v_i}\nabla a[f]\diff_h(f(1-f)) + \diff_h\nabla a[f]\partial_{v_i}(f(1-f))\right) \;dvdt \\
		& + \int_{t_1}^{t_2}\int_{\R^3} \nabla\brak{v}^m\diff_h f_{v_i} \cdot \left(\nabla a[f]\partial_{v_i}\diff_h[f(1-f)]\right)\;dvdt +\mathcal{E} \\
		&:= \sum_{j=1}^{12}I_j + \mathcal{E},
\end{align*}}}
where $\mathcal{E}$ denotes the error terms, which originate from the discrepancy between the product rules for $\diff_h$ and $\partial_{v_i}$. These terms are bounded identically to the others and so we omit the bound on $\mathcal{E}$.	
For $I_1$, our coercive term, we use Lemma \ref{lemA_lower} to obtain
\begin{equation*}
	I_1 \le - C(f_{in})\int_{t_1}^{t_2}\int_{\R^3} \brak{v}^{m-3} |\nabla \diff_hf_{v_i}|^2.
\end{equation*}
For $I_3$, when two derivatives land on the kernel $A[f]$, we use Young's inequality, H\"older's inequality in space, the Sobolev embedding $H^1(\R^3) \embeds L^6(\R^3)$, and Calderon-Zygmund, to obtain for any $\delta > 0$, the estimate
\begin{equation*}
	\begin{aligned}
	|I_3| &\lesssim \delta\int_{t_1}^{t_2}\int_{\R^3} \brak{v}^{m-3}|\nabla\diff_h f_{v_i}|^2 \;dvdt  + \delta^{-1}\int_{t_1}^{t_2}\int_{\R^3} \brak{v}^{m+3}|\partial_{v_i}\diff_h A[f(1-f)]|^2|\nabla f|^2 \;dvdt  \\
		&\lesssim \delta\int_{t_1}^{t_2}\int_{\R^3} \brak{v}^{m-3}|\nabla\diff_h f_{v_i}|^2 \;dvdt  + \delta^{-1}\int_{t_1}^{t_2}\|f\|_{L^{3}}\|\brak{v}^{m+3}|\nabla f|^2\|_{L^3} \;dt \\
		&\lesssim \delta\int_{t_1}^{t_2}\int_{\R^3} \brak{v}^{m-3}|\nabla\diff_h f_{v_i}|^2\;dvdt   + \delta^{-1}\int_{t_1}^{t_2}\|\brak{v}^{\frac{m+3}{2}}\nabla f\|_{L^6}^2 \;dt \\
		&\lesssim \delta\int_{t_1}^{t_2}\int_{\R^3} \brak{v}^{m-3}|\nabla\diff_h f_{v_i}|^2 \;dvdt  + \delta^{-1}\int_{t_1}^{t_2}\left(\|\brak{v}^{\frac{m+1}{2}}\nabla f)\|_{L^2}^2 + \|\brak{v}^\frac{m+3}{2}\nabla^2 f\|_{L^2}^2\right) \;dt \\
		&\lesssim \delta\int_{t_1}^{t_2}\int_{\R^3} \brak{v}^{m-3}|\nabla\diff_h f_{v_i}|^2 \;dvdt+ \delta^{-1}\left(\|\nabla f\|_{L^2\left([t_1,t_2];L^2_{m+1}\right)}^2 + \|\nabla^2 f\|_{L^2\left([t_1,t_2];L^2_{m +3}\right)}^2\right).
	\end{aligned}
\end{equation*}
Similarly, for $I_4$, when two derivatives land on the kernel $\nabla a[f]$, we use Young's inequality, H\"older's inequality, Calderon-Zygmund, Lebesgue interpolation, the Sobolev embedding $H^1(\R^3) \embeds L^6(\R^3)$, and Lemma \ref{reg_lemma1} to obtain for any $\delta > 0$, the estimate
\begin{equation*}
\begin{aligned}
	|I_4| &\lesssim \delta\int_{t_1}^{t_2}\int_{\R^3} \brak{v}^{m-3}|\nabla \diff_h f_{v_i}|^2 \;dvdt + \delta^{-1}\int_{t_1}^{t_2}\int_{\R^3} \brak{v}^{m+3}|\diff_h \partial_{v_i}\nabla a[f]|^2f^2\;dvdt  \\
		&\lesssim \delta\int_{t_1}^{t_2}\int_{\R^3} \brak{v}^{m-3}|\nabla \diff_h f_{v_i}|^2\;dvdt   + \delta^{-1}\int_{t_1}^{t_2} \||\partial_{v_i}\diff_h a[\nabla f]|^2\|_{L^3}\|f^2\brak{v}^{m+3}\|_{L^{3/2}} \;dt \\
		&\lesssim \delta\int_{t_1}^{t_2}\int_{\R^3} \brak{v}^{m-3}|\nabla \diff_h f_{v_i}|^2 \;dvdt  + \delta^{-1}\int_{t_1}^{t_2} \|\partial_{v_i}\diff_h a[\nabla f]\|_{L^6}^2 \|f\brak{v}^{\frac{m+3}{2}}\|_{L^3}^2 \;dt \\
		&\lesssim \delta\int_{t_1}^{t_2}\int_{\R^3} \brak{v}^{m-3}|\nabla \diff_h f_{v_i}|^2\;dvdt  + \delta^{-1}\int_{t_1}^{t_2} \|\nabla f\|_{L^6}^2 \left(\|f\brak{v}^{\frac{m+3}{2}}\|_{L^2}\|f\brak{v}^{\frac{m+3}{2}}\|_{L^6}\right)\;dt \\
		&\lesssim \delta\int_{t_1}^{t_2}\int_{\R^3} \brak{v}^{m-3}|\nabla \diff_h f_{v_i}|^2 \;dvdt  + \delta^{-1}\int_{t_1}^{t_2} \|\nabla^2 f\|_{L^2}^2 \left(\|f\|_{L^2_{m+3}}\left\|\nabla \left(f\brak{v}^{\frac{m+3}{2}}\right)\right\|_{L^2}\right) \;dt \\
		&\lesssim \delta\int_{t_1}^{t_2}\int_{\R^3} \brak{v}^{m-3}|\nabla \diff_h f_{v_i}|^2\;dvdt  + \delta^{-1}\int_{t_1}^{t_2} \|\nabla^2 f\|_{L^2}^2 \|f\|_{L^2_{m+3}}\left(\|\nabla f\|_{L^2_{m+3}} + \|f\|_{L^2_{m+1}}\right) \;dt \\
		&\lesssim \delta\int_{t_1}^{t_2}\int_{\R^3} \brak{v}^{m-3}|\nabla \diff_h f_{v_i}|^2 \;dvdt  + \delta^{-1}\|\nabla^2 f\|_{L^2([t_1,t_2];L^2)}^2
{(}\|\nabla f\|_{L^\infty([t_1,t_2];L^2_{m+3})} {+1).}
\end{aligned}
\end{equation*}
To bound the remaining terms $I_2$ and $I_5, \cdots, I_{12}$, we modify the arguments from Lemma \ref{reg_lemma2} in a similar fashion, using the additional tool of the Sobolev embedding as necessary, to obtain
\begin{equation*}
	\begin{aligned}
		&\int_{\R^3} \brak{v}^m |\diff_h f_{v_i}(t_2)|^2 \;dv + (C-\delta)\int_{t_1}^{t_2}\int_{\R^3} \brak{v}^{m-3}|\nabla \diff_h f_{v_i}|^2\;dvdt  \lesssim 1 +  \int_{\R^3} \brak{v}^m |\diff_h f_{v_i}|^2 \;dv \\
		&\qquad\qquad +(1 +\delta^{-1})\left(\|\nabla f\|_{L^\infty([t_1,t_2];L^2_{m+6})}^4 + \|\nabla^2 f\|_{L^2([t_1,t_2];L^2_{m+3})}^4\right).
	\end{aligned}
\end{equation*}
 Thus, taking $\delta$ sufficiently small and taking {the limit} $h\rightarrow 0^+$, we conclude that $\nabla^2 f$ is weakly differentiable and we obtain
\begin{equation*}
	\begin{aligned}
		\|\nabla^2f(t_2)\|_{L^2_m}^2 + \int_{t_1}^{t_2} \|\nabla^3 f\|_{L^2_{m-3}}^2 \;ds &\lesssim 1 + \|\nabla^2 f(t_1)\|_{L^2_m}^2 +\|\nabla f\|_{L^\infty([t_1,t_2];L^2_{m+6})}^4\\
		&\qquad+ \|\nabla^2 f\|_{L^2([t_1,t_2];L^2_{m+3})}^4.
	\end{aligned}
\end{equation*}
 Taking a supremum over $t_2 \in [2t,T]$ and an average over $t_1\in [t,2t]$, and applying Lemma \ref{reg_lemma2}, we obtain
\begin{equation*}
	\begin{aligned}
		\sup_{(2t, T)} \|\nabla^2 f(\cdot)\|_{L^2_m}^2 + \int_{2t}^{T}\|\nabla^3 f\|_{L^2_{m-3}}^2\;ds  &\lesssim \frac{1}{t}\int_{t}^{T_2}\|\nabla^2 f\|_{L^2_m}^2 \;ds \\
			&\qquad+\left(\int_{t}^{T} \|\nabla^2 f\|_{L^2_{m+3}}^2 \; ds\right)^2\\
			&\qquad+\left(\sup_{(t,T)} \|\nabla f\|_{L^2_{m+6}}^2\right)^2\\
			&\lesssim \left( 1+ \frac{1}{t^2}\right).
	\end{aligned}
\end{equation*}
\end{proof}

\begin{remark}
From Lemma \ref{reg_lemma3}, one can continue to bootstrap spatial regularity, and obtain the corresponding higher regularity estimates, that provided $f_{in} \in L^2_{m+6k}$, for each $0 < t_0 < T$, $\nabla^k f \in L^\infty([t_0,T]; L^2_m)$, and moreover,
\begin{equation*}
	\sup_{t_0 < t < T} \|\nabla^k f(\cdot)\|_{L^2_m} + \left(\int_{t_0}^{T}\|\nabla^{k+1}f(s)\|^2_{L^2_{m-3}} \;ds \right)^{1/2} \lesssim_{f_{in},k,m,T} \left(1 + \frac{1}{t_0}\right)^{k/2}.
\end{equation*}
If $f_{in}$ is rapidly decaying, i.e. $f_{in} \in L^2_m$ for each $m \ge 0$, then $f$ is Schwartz class in space. That is, $f\in L^\infty([t,T];\mathcal{S}(\R^3))$ for each $t > 0$.
\end{remark}

Instead of bootstrapping spatial regularity and deducing the corresponding time regularity from the equation, we use Lemma \ref{reg_lemma3} to conclude H\"older regularity of $f$. Combined with the parabolic divergence structure of \eqref{LFD}, we deduce spatial and temporal regularity simultaneously via classical Schauder estimates. As the initial step, we have the following lemma:

\begin{lemma}\label{holder_lem}
Let $f$ be any weak solution to \eqref{weak_FDL} as in Theorem  \ref{thm_existence}, with $f_{in} \in L^1_3 \cap L^2(12)$. Then, $f\in C^{\alpha/2}((0,T];C^\alpha(\R^3))$ for some $\alpha \in (0,1)$.
\end{lemma}

\begin{proof}

By Lemma \ref{reg_lemma3}, we conclude $f\in L^\infty((0,T];W^{1,p})$ for each $2 \le p \le 6$. Therefore, by a duality argument, $\partial_t f$ belongs to  $L^\infty((0,T];W^{-1,p})$ for $2 \le p \le 6$. By a (real) interpolation of the Sobolev spaces $L^\infty((0,T);W^{1,p})$ and $W^{1,\infty}((0,T);W^{-1,p})$, we  obtain
$f\in W^{s_1,6}((0,T);W^{s_2,6})$ for $s_2$ strictly less, but as close as one wishes, than $1-2\theta$, and $s_1$ strictly less, but as close as one wishes, than $\theta$, for any $\theta \in (0,1)$ (see Theorem 3.1 in \cite{amann}). Hence, choosing $\theta < \frac{1}{4}$, Morrey's inequality implies $f\in C^{0,\alpha/2}((0,T);C^{0,\alpha} (\mathbb{R}^3))$, for some $\alpha >0$.
\end{proof}

Now, we are ready to apply a standard bootstrapping argument and conclude $f$ is smooth:
\begin{flushleft}
\underline{Proof of Theorem \ref{thm_regularity}}\\
\end{flushleft}
\vskip 1em
By Lemma \ref{holder_lem}, we conclude $f\in C^{\alpha/2}((0,T);C^\alpha(\R^3))$ for some $\alpha \in (0,1)$ and $f$ solves the divergence form parabolic equation,
\begin{equation}
\partial_t f = \nabla \cdot (A[f(1-f)]{\nabla f} - \nabla a[f](1-f)f),
\end{equation}
in the weak sense. Hence, Lemma 4.7 in \cite{GuGu15}  shows that $\nabla a[f]$ and  $A[f(1-f)] $ belong to $C^{0,\eta/2}((0,T];C^{0,\eta})$. Thus, $f$ satisfies a divergence-form parabolic equation with H\"older continuous coefficients. By Theorem 12.1 from Chapter 3 in \cite{LadSolUra1968}, we conclude $f\in C^{1,\mu/2}((0,T];C^{1,\mu})$. Bootstrapping the argument, we obtain higher regularity of the coefficients $A[f(1-f)]$ and $\nabla a[f]$,  from which $f\in C^{\infty}((0,T];C^\infty)$ follows, as desired.

\section{Long time behavior}\label{sec_ltb}

In this section we prove Theorem.~\ref{thm_longtime}. Without loss of generality, we can assume that $\eps=1$.
We first rewrite the initial value problem associated to \eqref{LFD} in the following compact form
\begin{equation}
\label{LFDII}
\begin{cases}
\pa_t f + \mathcal{T}[f] = 0 & v\in\R^3,\quad t>0,\\
f(0,v) = f_{in} & v\in\R^3,
\end{cases}
\end{equation}
where the Landau-Fermi-Dirac operator is defined by
\begin{align}
\label{LFD.2}
\mathcal{T}[f](v) &= -\nabla\cdot\int_{\R^3}\frac{\Pi(v-v^*)}{|v-v^*|}\left(
f^*(1-f^*)\nabla f - f(1-f)\nabla f^*\right)dv^*\\ \nonumber
&=-\nabla\cdot\left(
A[f(1-f)]\nabla f - f(1-f)\nabla a[f]\right),
\end{align}
and the quantities $A[\cdot ]$, $a[\cdot ]$ are defined in \eqref{defn_coef}.

We first show unconditional convergence without rate towards the steady state for \eqref{LFD}, which is the first part of Theorem ~\ref{thm_longtime}.
\begin{proposition}[Convergence to the steady state]\label{lem.cnv}
	Given any initial datum $f_{in} : \R^3\to [0,1]$, $f_{in}\in L^1_2$, such that $H_1[f_{in}] < 0$, the solution $f$ to the (\ref{LFDII}) tends to the Fermi-Dirac distribution $\cM$ with same mass, momentum and energy as $f_{in}$ when $t\to\infty$.
\end{proposition}
\begin{proof}
We recall that $f$ satisfies a uniform in time bound in $L^1\cap L^\infty$, and therefore $\sup_{t\geq 0} \|f (t)\|_{L^2} < \infty$. In what follows we will often make use of this relation without mentioning it.	
	
	Integrating the entropy balance equation in time yields
	\begin{align*}
	H_1[f(t)] + \int_0^t D[f(\tau)]d\tau \leq H_1[f_{in}],\qquad t>0,
	\end{align*}
	with 
	\begin{align*}
	D[f] = \int_{\R^3}\frac{A[f(1-f)]}{f(1-f)}\nabla f\cdot \nabla f dv - 8\pi\int_{\R^3}f^2 dv \geq 0.
	\end{align*}
	Since $D[f(\cdot)]\in L^1(0,\infty)$, there exists a sequence $t_n\to\infty$ such that $D[f(t_n)]\to 0$ as $n\to\infty$. Define $f_n = f(t_n)$. Given the lower bound for $A$ we deduce
	\begin{align*}
	\int_{\R^3}|\nabla f_n|^2 \brak{v}^{-3}dv \lesssim \int_{\R^3}\frac{A[f_n(1-f_n)]}{f_n(1-f_n)}\nabla f_n\cdot \nabla f_n dv
	\lesssim\int_{\R^3}f_n^2 dv + D[f_n]\lesssim 1.
	\end{align*}
	Therefore $\brak{v}^{-3/2}\nabla f_n$ is bounded in $L^2$. However $f_n\nabla\brak{v}^{-3/2}$ is bounded in $L^2$, so the product $f_n\brak{v}^{-3/2}$ is bounded in $H^1$. Furthermore $f_n\brak{v}^2$ is bounded in $L^1$. We deduce via Sobolev embedding that $f_n$ is relatively compact in $L^2$, and more in general (via the $L^\infty$ bounds and the bound on the second moment of $f_n$) in $L^p$ for every $p\in [1,\infty)$. Let us denote with $f_\infty$ its limit.
	We have that  $\brak{v}^{-3/2}\nabla f_n\rightharpoonup\brak{v}^{3/2}\nabla f_\infty$ weakly in $L^2$. This is enough to deduce via a generalized Fatou argument \cite[Lemma A.4]{BJPZ} that 
	\begin{align*}
	D_\delta[f_\infty] \leq \liminf_{n\to\infty}D_\delta[f_n] 
	\leq \liminf_{n\to\infty}D[f_n] = 0,
	\end{align*}
	with
	\begin{align*}
	D_\delta[f] := \int_{B_{1/\delta}}\frac{A[f(1-f)]}{f(1-f)+\delta}\nabla f\cdot \nabla f dv - 8\pi\int_{\R^3}f^2 dv,
	\end{align*}
	and $\delta>0$ is arbitrary. Via monotone convergence we deduce
	\begin{align*}
	0\leq D[f_\infty] = \lim_{\delta\to 0}D_\delta[f_\infty]\leq 0.
	\end{align*}
	It follows that $D[f_\infty] = 0$. Since we know that $\int_{\R^3}f_n(1-f_n)dv\geq c>0$, it follows \cite{BL04} that
	$f_\infty = \cM$. This means that $f_n = f(t_n)\to\cM$ strongly in $L^p$ for $p\in [1,\infty)$. In particular the relative Fermi-Dirac entropy $H_1[f(t_n)\vert\cM] = H_1[f(t_n)]-H_1[\cM]\to 0$ as $n\to\infty$. On the other hand, we know that $t\mapsto H_1[f(t)\vert\cM]$ is non-increasing, so it must hold
	$\lim_{t\to\infty}H_1[f(t)\vert\cM]=0$. This easily implies the strong convergence $f(t)\to\cM$ as $t\to\infty$ in $L^1$.
	This finishes the proof of the Lemma.	
\end{proof}
Our next goal is to prove exponential convergence of the solution $f(t)$ to \eqref{LFD} towards the steady state $\cM$ in case the initial datum $f_{in}$ is close enough to $\cM$ in the norm $L^2(m)$. This is in the second part of Theorem ~\ref{thm_longtime}.
We linearize our equation around the steady state $\cM$. We will work in weighted Lebesgue spaces with weight $m$ defined by 
\begin{align}
\label{m}
m := \cM(1-\cM), 
\end{align}
where $\cM$ is the Fermi-Dirac distribution defined in \eqref{Fermi-Dirac}. Writing
\begin{align}
\label{def.h}
h:= \frac{f - \cM}{m} ,\qquad
\textrm{and} \quad -\frac{1}{m}\mathcal{T}[f] =: L h + \Gamma_2[h,h] + \Gamma_3[h,h,h]
\end{align}
it defines the linearized operator $L$ and the quadratic and cubic perturbations $\Gamma_2$, $\Gamma_3$, respectively.

Via straightforward computations \cite{ABDL_21_II} one finds
\begin{align}
\label{1}
(Lh)(v) &= \frac{1}{m(v)}\nabla\cdot\int_{\R^3}
\frac{m(v^*)m(v)}{|v-v^*|}\Pi(v-v^*)(\nabla h(v) - \nabla h(v^*)) dv^* ,\\
\label{1.G2}
\Gamma_2[h,h](v) &= \frac{1}{m(v)}\nabla\cdot\left(
A[(1-2\cM)mh]\nabla(mh) - A[m^2 h^2]\nabla\cM\right.\\
\nonumber
&\left.\qquad - (1-2\cM)mh\nabla a[mh] + m^2 h^2\nabla a[\cM]
\right),\\
\label{1.G3}
\Gamma_3[h,h,h](v) &= \frac{1}{m(v)}\nabla\cdot\left(
-A[m^2 h^2)]\nabla (mh) + m^2 h^2\nabla a[m h]\right).
\end{align}
Define the spaces
$$
L^2(m) := L^2(\R^3,m(v)dv),\quad H^1(m) := H^1(\R^3,m(v)dv),
$$
and recall that $\brak{v} = (1+|v|^2)^{1/2}$. 

Our goal is to prove a spectral gap estimate for the linearized operator $L$.
We will apply \cite[Lemma 10]{daus2016hypocoercivity}. In order to do so, we adapt the latter result's framework and therefore define for $k\geq 0$ the following Hilbert spaces
\begin{align*}
\mathcal{H}_0^k &= L^2(m \langle v \rangle^{k-1} ),\\
\mathcal{H}^k &= \left\{ h\in \mathcal{H}_0^k ~ : ~ \|h\|_{\mathcal{H}^k}^2 \equiv \|h\|_{ \mathcal{H}_0^k }^2 + 
\int_{\R^3}\nabla h\cdot A[m]\nabla h\,\langle v \rangle^{k} m dv<\infty \right\}.
\end{align*}
Clearly $\mathcal{H}^k\hookrightarrow \mathcal{H}_0^k$ with continuous embedding.

We split then the linearized operator $L$ into two contributions, in the following fashion:
\begin{align}
\label{L.split}
L =& \cK_k - \Lambda_k,\\
\label{Lambda}
(\Lambda_k h)(v) :=& -
\frac{1}{m(v)}\nabla\cdot\left[\left(\int_{\R^3}
\frac{m(v^*)m(v)}{|v-v^*|}\Pi(v-v^*) dv^* 
\right)\nabla h(v)\right] \\
\nonumber
& -8\pi m(v)h(v) + \xi\int_{\R^3}m h \langle v \rangle^k dv,\\
\label{Kappa}
(\cK_k h)(v) :=& -
\frac{1}{m(v)}\nabla\cdot\int_{\R^3}
\frac{m(v^*)m(v)}{|v-v^*|}\Pi(v-v^*)\nabla h(v^*) dv^* \\
\nonumber
& -8\pi m(v)h(v) + \xi\int_{\R^3}m h \langle v \rangle^k dv,
\end{align}
where $\xi>0$ is an arbitrary constant, to be specified later. 
We also recall the definition of the Maxwellian $M$:
\begin{align*}
M(v) = e^{-b|v-u|^2},\qquad v\in\R^3,
\end{align*}
and point out that $M\sim m$ (via direct computations).

We prove now the following coercivity estimate for $\Lambda_k$.
\begin{lemma}\label{lem.Lambda}
$\Lambda_k : \mathcal{H}^k \to (\mathcal{H}^k)'$ is bounded and 
$(\Lambda_k h, h)_{L^2(m\langle v \rangle^{k})} \gtrsim \|h\|_{\mathcal{H}^k}^2$ for every $h\in\mathcal{H}^k$, provided that $\xi>0$ is large enough.
\end{lemma}
\begin{proof}
From \eqref{Lambda} and the definition \eqref{defn_coef} of $A$ it follows, via an integration by parts,
\begin{align}\nonumber
(\Lambda_k h_1, h_2)_{L^2(m\langle v \rangle^{k})} 
=& \int_{\R^3}\nabla (\langle v \rangle^{k}h_2(v))\cdot
\left(\int_{\R^3}
\frac{m(v^*)}{|v-v^*|}\Pi(v-v^*) dv^* 
\right)\nabla h_1(v)\, m(v) dv\\
\nonumber
& -8\pi\int_{\R^3} h_1 h_2 \langle v \rangle^{k}m^2 dv 
+ \xi \left( \int_{\R^3}h_1 \langle v \rangle^k m dv\right)
\left( \int_{\R^3}h_2 \langle v \rangle^{k}m dv\right)\\
\label{Lambda.quadform}
=& 8\pi\int_{\R^3}\nabla h_2(v)\cdot A[m]\nabla h_1(v)\, \langle v \rangle^{k}m(v) dv\\
\nonumber
& +8k\pi\int_{\R^3}\langle v \rangle^{k-2} h_2(v) v\cdot A[m]\nabla h_1(v)\, m(v) dv\\
\nonumber
& -8\pi\int_{\R^3} h_1 h_2 m^2 \langle v \rangle^{k} dv 
+ \xi \left( \int_{\R^3}h_1 m\langle v \rangle^{k} dv\right)
\left( \int_{\R^3}h_2 m\langle v \rangle^{k} dv\right),
\end{align}
for $ h_1, h_2 \in \mathcal{H}^k$. Being $A[m](v)$ symmetric and positive definite for $v\in\R^3$,  Cauchy-Schwartz yields
\begin{align*}
\Big|( \Lambda_k h_1,& h_2)_{L^2(\langle v \rangle^{k}m)}\Big|\\
\lesssim &
\int_{\R^3}
\left( \nabla h_2(v)\cdot A[m]\nabla h_2(v) \right)^{1/2}
\left( \nabla h_1(v)\cdot A[m]\nabla h_1(v) \right)^{1/2}\,
\langle v \rangle^{k} m(v) dv\\
& + \int_{\R^3} \left( \nabla h_1(v)\cdot A[m]\nabla h_1(v) \right)^{1/2}
\left( h_2(v)^2\langle v \rangle^{-4}
v\cdot A[m]v \right)^{1/2}
\langle v \rangle^{k} m(v) dv\\
& + \|h_1\|_{\mathcal{H}_0}\|h_2\|_{\mathcal{H}_0}\\
\lesssim & 
\left( \int_{\R^3} \nabla h_2\cdot A[m]\nabla h_2\,\langle v \rangle^k m dv \right)^{1/2}
\left( \int_{\R^3} \nabla h_1\cdot A[m]\nabla h_1\,\langle v \rangle^k m dv \right)^{1/2}\\
& + \left( \int_{\R^3} \nabla h_1(v)\cdot A[m]\nabla h_1(v)\langle v \rangle^{k} m(v) dv \right)^{1/2}
\left( \int_{\R^3}h_2(v)^2 \langle v \rangle^{k-4} v\cdot A[m]v\, m dv \right)^{1/2}\\
& + \|h_1\|_{\mathcal{H}_0^k}\|h_2\|_{\mathcal{H}_0^k}.
\end{align*}
Therefore
\begin{align*}
\left|(\Lambda_k h_1, h_2)_{L^2(m\langle v \rangle^k)}\right|\lesssim & 
\|h_1\|_{\mathcal{H}^k}\|h_2\|_{\mathcal{H}^k}.
\end{align*}
Via a duality argument it follows that $\Lambda$ is bounded as an operator $\mathcal{H}^k\to (\mathcal{H}^k)'$.

Choosing $h_1 = h_2 = h$ in \eqref{Lambda.quadform} yields
\begin{align}
\nonumber (\Lambda_k h, h)_{L^2(m\langle v \rangle^k)} =&
\; 8\pi\int_{\R^3}\nabla h\cdot A[m]\nabla h\,
\langle v \rangle^{k} m dv
-8\pi\; \int_{\R^3} h^2 \langle v \rangle^{k}m^2 dv 
+ \xi \left( \int_{\R^3}h \langle v \rangle^{k}m dv\right)^2 \\
&+8k\pi\int_{\R^3}\langle v \rangle^{k-2} h(v) v\cdot A[m]\nabla h(v)\, m(v) dv.\label{Lambda.lb.0}
\end{align}
The last integral can be estimated via Cauchy-Schwartz:
\begin{align}
\nonumber (\Lambda_k h, h)_{L^2(m\langle v \rangle^k)} \geq &
\; 4\pi\int_{\R^3}\nabla h\cdot A[m]\nabla h\,
\langle v \rangle^{k} m dv
-8\pi\int_{\R^3} h^2 \langle v \rangle^{k}m^2 dv 
+ \xi \left( \int_{\R^3}h \langle v \rangle^{k}m dv\right)^2 \\ \label{Lambda.lb}
&-C\int_{\R^3} h^2 \langle v \rangle^{k-5}m dv.
\end{align}
Let us focus on the first integral on the right-hand side of \eqref{Lambda.lb}.
Lemma \ref{lemA_lower} and the fact that $H_1[\cM]<0$ lead to
\begin{align*}
\int_{\R^3}\nabla h\cdot A[m]\nabla h\,\langle v \rangle^k m dv\gtrsim 
\int_{\R^3}|\nabla h|^2 m(v) {\brak{v}^{k-3}} dv .
\end{align*}
For every $R>0$, since $m(v)\langle v \rangle^{k-3}$ is uniformly positive on $B_R$ (with an $R-$dependent lower bound), it follows via (the standard) Sobolev's embedding and Poincar\'e's Lemma
\begin{align*}
\int_{\R^3}\nabla h\cdot A[m]\nabla h\,\langle v \rangle^k m dv &\geq
c_R\int_{B_R}|\nabla h|^2 dv \geq c_R \left\|h - \fint_{B_R}h dv \right\|_{L^6(B_R)}^2\\
&\geq c_R \|h\|_{L^6(B_R)}^2 - c_R'\left( \int_{B_R}h dv \right)^2\\
&\geq c_R \|h\|_{L^6(B_R)}^2 - c_R''\int_{B_R}h^2 dv .
\end{align*}
From \eqref{Lambda.lb} and the above inequality we deduce
\begin{align}\nonumber
(\Lambda_k h, h)_{L^2(m\langle v \rangle^k)} \geq &\int_{\R^3}\nabla h\cdot A[m]\nabla h\,\langle v \rangle^k m dv
+ c_R \|h\|_{L^6(B_R)}^2 - c_R''\int_{B_R}h^2 dv \\
\nonumber
& -8\pi\int_{B_R} h^2\langle v \rangle^k m^2 dv -8\pi\int_{\R^3\backslash B_R} h^2 \langle v \rangle^k m^2 dv 
+ \xi \left( \int_{\R^3}h\langle v \rangle^k m dv\right)^2\\
\nonumber
& -C\int_{B_R} h^2 \langle v \rangle^{k-5}m dv
-C\int_{\R^3\backslash B_R} h^2 \langle v \rangle^{k-5}m dv\\
\nonumber
\geq 
&\int_{\R^3}\nabla h\cdot A[m]\nabla h\,\langle v \rangle^k m dv
+ c_R \|h\|_{L^6(B_R)}^2 - c_R'''\|h\|_{L^2(B_R)}^2 \\
& \label{nova.1} -8\pi\int_{\R^3\backslash B_R} h^2\langle v \rangle^k m^2 dv 
+ \xi \tilde c_R \|h\|_{L^1(B_R)}^2
-C\int_{\R^3\backslash B_R} h^2 \langle v \rangle^{k-5}m dv.
\end{align}
Let us now consider
\begin{align*}
\|h\|_{\mathcal{H}_0^k}^2 = \int_{\R^3} h^2 \langle v\rangle^{k-1} m(v) dv\lesssim
\int_{\R^3} \langle v\rangle^{k-3}h^2 \langle v\rangle^{2} M(v) dv.
\end{align*}
Young's inequality with the convex conjugated functions 
$s\mapsto\frac{s}{\delta}\log\frac{s}{\eta} - \frac{s}{\delta}$, 
$s\mapsto\eta\delta^{-1}e^{\delta s}$ (with $\eta>0$ arbitrary and $\delta > 0$ fixed small enough such that $\int_{\R^3}e^{\delta |v|^2}M(v)dv < \infty$) leads to
\begin{align*}
\|h\|_{\mathcal{H}_0^k}^2 \lesssim
\int_{\R^3} [\delta^{-1}\langle v\rangle^{k-3}h^2\log(\eta^{-1}\langle v\rangle^{k-3}h^2) - \delta^{-1}\langle v\rangle^{k-3}h^2] M(v)dv + 
\eta\delta^{-1}\int_{\R^3} e^{\delta \langle v\rangle^{2}} M(v) dv.
\end{align*}
By defining $u = \langle v\rangle^{(k-3)/2}h$ and rescaling
$\eta\mapsto \|u\|_{L^2(G)}^2\eta$, the above inequality can be rewritten as
\begin{align*}
\|h\|_{\mathcal{H}_0^k}^2 \lesssim \int_{\R^3}u^2\log\frac{u^2}{\|u\|_{L^2(M)}^2} M(v) dv
-c\|u\|_{L^2(M)}^2(1 + \log\eta)
+ \eta\|u\|_{L^2(M)}^2 .
\end{align*}
By employing the log-Sobolev's inequality with Gaussian weight \cite{gross1975logarithmic} one obtains
\begin{align*}
\|h\|_{\mathcal{H}_0^k}^2 \lesssim \|\nabla u\|_{L^2(M)}^2
+\|u\|_{L^2(M)}^2(\eta -c - c\log\eta).
\end{align*}
Replacing $u$ with  $\langle v\rangle^{(k-3)/2}h$ and choosing $\eta>0$ the minimum point of $\eta -c - c\log\eta$, one finds
\begin{align}\label{pizza}
\|h\|_{\mathcal{H}_0^k}^2 \lesssim \int_{\R^3} |\nabla h|^2 \langle v \rangle^{k-3} M(v) dv +
\int_{\R^3}h^2 \langle v \rangle^{k-3}M(v) dv .
\end{align}
Lemma \ref{lemA_lower}, relation $m\sim M$ and \eqref{pizza} yield
\begin{align}\label{pizza.2}
\|h\|_{\mathcal{H}_0}^2 \lesssim 
\int_{\R^3}\nabla h\cdot A[m]\nabla h\,\langle v \rangle^{k} m dv
 + \int_{\R^3}h^2 \langle v \rangle^{k-3}m dv .
\end{align}
At this point, \eqref{nova.1} and \eqref{pizza.2} yield
\begin{align*}
(\Lambda_k h, h)_{L^2(m\langle v \rangle^{k})} \gtrsim 
& \|h\|_{\mathcal{H}_0^k}^2 +
\int_{\R^3}\nabla h\cdot A[m]\nabla h\,\langle v \rangle^{k} m dv
+ c_R \|h\|_{L^6(B_R)}^2 - c_R'''\|h\|_{L^2(B_R)}^2 \\
\nonumber
& -8\pi\int_{\R^3\backslash B_R} h^2\langle v \rangle^{k} m^2 dv 
+ \xi \tilde c_R \|h\|_{L^1(B_R)}^2
-C\int_{\R^3}h^2 \langle v \rangle^{k-3}m dv\\
& -C\int_{\R^3\backslash B_R} h^2 \langle v \rangle^{k-5}m dv,
\end{align*}
which implies, given that $m\lesssim \langle v \rangle^{-3}$,
\begin{align}\label{nova.2}
(\Lambda_k h, h)_{L^2(m\langle v \rangle^{k})} \gtrsim 
& \|h\|_{\mathcal{H}_0^k}^2 +
\int_{\R^3}\nabla h\cdot A[m]\nabla h\,\langle v \rangle^{k} m dv 
- C\int_{\R^3\backslash B_R} h^2 m \langle v \rangle^{k-3} dv\\
& + c_R \|h\|_{L^6(B_R)}^2 - c_R'''\|h\|_{L^2(B_R)}^2 
+ \xi \tilde c_R \|h\|_{L^1(B_R)}^2 .\nonumber
\end{align}
Choosing $R>0$, we absorb the third integral on the right-hand side of \eqref{nova.2} via $\|h\|_{\mathcal{H}_0}^2$, yielding
\begin{align}\label{nova.3}
(\Lambda_k h, h)_{L^2(m\langle v \rangle^{k})} \geq 
& \|h\|_{\mathcal{H}_0^k}^2 +
\int_{\R^3}\nabla h\cdot A[m]\nabla h\,\langle v \rangle^{k} m dv\\ 
& + \|h\|_{L^6(B_R)}^2 - K\|h\|_{L^2(B_R)}^2 
+ \xi \|h\|_{L^1(B_R)}^2 .\nonumber
\end{align}
By interpolating $L^2$ between $L^1$ and $L^6$ and applying Young's inequality one finds
\begin{align*}
K\|h\|_{L^2(B_R)}^2 \leq K \|h\|_{L^1(B_R)}^{4/5}\|h\|_{L^6(B_R)}^{6/5}\leq 
\frac{2}{5}\xi \|h\|_{L^1(B_R)}^{2} + \frac{3}{5} K^{5/3}\xi^{-2/3}\|h\|_{L^6(B_R)}^2.
\end{align*}
Therefore, for $\xi>0$ large enough, it holds $\|h\|_{L^6(B_R)}^2 - K\|h\|_{L^2(B_R)}^2 + \xi \|h\|_{L^1(B_R)}^2 \geq 0$.
We conclude
\begin{align*}
(\Lambda_k h, h)_{L^2(m\langle v \rangle^{k})} \geq 
& \|h\|_{\mathcal{H}_0^k}^2 +
\int_{\R^3}\nabla h\cdot A[m]\nabla h\,\langle v \rangle^{k} m dv = \|h\|_{\mathcal{H}^k}^2.
\end{align*}
This finishes the proof of the Lemma.
\end{proof}
Concerning $\cK_k$, we are going to prove the following result:
\begin{lemma}\label{lem.new} For $k\geq 0$ it holds
\begin{align}
\label{Kappa.2}
(\cK_k h)(v) &= \frac{\nabla m(v)}{m(v)}\cdot \left(
\tilde{\cK}\ast (h\nabla m) - \nabla\tilde{\cK}\ast (h m)\right)
+\nabla\tilde{\cK}\ast(h\nabla m) + \xi\int_{\R^3}mh \langle v \rangle^{k} dv,
\end{align}
with
\begin{align*}
\tilde{\cK}(v) = \frac{\Pi(v)}{|v|}.
\end{align*}
%
%
Furthermore $\cK_k : \mathcal{H}_0^k\to \mathcal{H}_0^k$ is a compact operator and the following bound holds for $k\geq 0$
\begin{align}\label{bound.K1}
|(\cK_k h, h)_{L^2(m\langle v \rangle^k)}|\lesssim \|h\|_{L^2(m\langle v \rangle^{k-2})}^2.
\end{align}
\end{lemma}
\begin{proof}
An integration by parts yields
	\begin{align*}
	(\cK_k h)(v) =& -
	\frac{\nabla m(v)}{m(v)}\cdot\int_{\R^3}
	\frac{m(v^*)}{|v-v^*|}\Pi(v-v^*)\nabla h(v^*) dv^*\\
	& - \nabla\cdot\int_{\R^3}
	\frac{m(v^*)}{|v-v^*|}\Pi(v-v^*)\nabla h(v^*) dv^*
	-8\pi m(v)h(v) + \xi\int_{\R^3}m h\langle v \rangle^{k} dv\\
	=\;& \frac{\nabla m(v)}{m(v)}\cdot\int_{\R^3}
	\frac{\nabla m(v^*)}{|v-v^*|}\Pi(v-v^*) h(v^*) dv^*\\
	& + \frac{\nabla m(v)}{m(v)}\cdot\int_{\R^3}
	h(v^*) m(v^*) \nabla_{v^*}\left[\frac{\Pi(v-v^*)}{|v-v^*|} \right] dv^*\\
	&+ \nabla\cdot\int_{\R^3}
	\frac{h(v^*)}{|v-v^*|}\Pi(v-v^*)\nabla m(v^*) dv^*\\
	& - \nabla\cdot\int_{\R^3}
	\frac{\Pi(v-v^*)}{|v-v^*|}\nabla [m(v^*) h(v^*)] dv^* 
	-8\pi m(v)h(v) + \xi\int_{\R^3}m h\langle v \rangle^{k} dv.
	\end{align*}
	Since 
	\begin{align*}
	- \nabla\cdot\int_{\R^3}
	\frac{\Pi(v-v^*)}{|v-v^*|}\nabla f(v^*) dv^*
	= 8\pi f(v)\qquad\forall f\in C^\infty_c(\R^3),
	\end{align*}
	we deduce that \eqref{Kappa.2} holds.
	
	Let now $(h_n)_{n\in\N}$ be a bounded sequence in $\mathcal{H}_0^k = L^2(m\langle v \rangle^{k-1})$. For $1<p\leq 2$, $s\in\R$, we have
	\begin{align*}
	\|h_n\nabla m\|_{L^p(\R^3)}^p
	\lesssim\int_{\R^3} |h_n|^p |\nabla m^{1-1/p}|^p \langle v \rangle^{-s} ~ 
	m\langle v \rangle^{s} dv
	\lesssim \int_{\R^3} |h_n|^p m\langle v \rangle^{s} dv.
	\end{align*}
	H\"older's inequality yields 
	\begin{align}\label{hnp2}
	\|h_n\nabla m\|_{L^p(\R^3)}\lesssim 
	\|h_n\|_{L^2(m \langle v \rangle^s )},\qquad 1<p\leq 2,\quad s\in\R.
	\end{align}
	In a similar way, one shows
	\begin{align}\label{hnp2.bis}
	\|h_n m\|_{L^p(\R^3)}\lesssim 
	\|h_n\|_{L^2(m \langle v \rangle^s )},\qquad 1<p\leq 2,\quad s\in\R.
	\end{align}
	This means that $h_n\nabla m$, $h_n m$ are bounded in $L^p(\R^3)$ for $1<p\leq 2$. Let us now consider, for $R>0$ arbitrary,
	\begin{align*}
	\nabla\tilde{\cK}\ast(h_n\nabla m) = 
	({\bf 1}_{ B_R }\nabla\tilde{\cK})\ast(h_n\nabla m) 
	+ ({\bf 1}_{ \R^3\backslash B_R }\nabla\tilde{\cK})\ast(h_n\nabla m).
	\end{align*}
	Given that $\nabla\tilde{\cK}\in L^{1}(B_R)$, from \cite[Corollary 4.28]{Brezis} it follows that $({\bf 1}_{ B_R }\nabla\tilde{\cK})\ast(h_n\nabla m)$
	is relatively compact in $L^2(\Omega)$ for every measurable set $\Omega$ with finite measure. A Cantor diagonal argument yields the existence of a subsequence of $h_n$ (not relabeled) such that 
	$({\bf 1}_{ B_R }\nabla\tilde{\cK})\ast(h_n\nabla m)$
	is strongly convergent in $L^2(B_r)$ for every $r\in\N$. Given that $\int_{\R^3}m \langle v \rangle^{k-1}dv < \infty$, it is easily seen that
	\begin{align}\label{doccia.1}
	({\bf 1}_{ B_R }\nabla\tilde{\cK})\ast(h_n\nabla m)\to ({\bf 1}_{ B_R }\nabla\tilde{\cK})\ast(h\nabla m)\quad\mbox{  strongly in $L^2(m \langle v \rangle^{k-1}) = \mathcal{H}_0^k$.}
	\end{align}
	On the other hand, Young's inequality for convolutions yields
	\begin{align*}
	\|({\bf 1}_{ \R^3\backslash B_R }\nabla\tilde{\cK})\ast(h_n\nabla m)\|_{L^2(\R^3)}\leq 
	\|\nabla\tilde{\cK}\|_{L^q(\R^3\backslash B_R)}
	\|h_n\nabla m\|_{L^p(\R^3)},\quad \frac{3}{2} = \frac{1}{p} + \frac{1}{q},\quad 1<p<\frac{6}{5}.
	\end{align*}
	Since $q > 3/2$ then $\|\nabla\tilde{\cK}\|_{L^q(\R^3\backslash B_R)}\to 0$ as $R\to \infty$, while $\|h_n\nabla m\|_{L^p(\R^3)}\lesssim 1$ for $1<p<6/5$.
	From this fact and \eqref{doccia.1} we obtain
	\begin{align}\label{doccia.2}
	\nabla\tilde{\cK}\ast(h_n\nabla m)\to \nabla\tilde{\cK}\ast(h\nabla m)\quad\mbox{  strongly in $\mathcal{H}_0^k$.}
	\end{align}
	In a similar way one shows that 
	\begin{align}\label{doccia.3}
	\frac{\nabla m}{m}\cdot\nabla\tilde{\cK}\ast(h_n m)\to \frac{\nabla m}{m}\cdot\nabla\tilde{\cK}\ast(h m)\quad\mbox{  strongly in $\mathcal{H}_0^k$.}
	\end{align}
	Let us now deal with $\tilde{\cK}\ast (h_n\nabla m)$. 
	One can prove, via a similar argument as the one employed to show \eqref{doccia.1}, that (up to subsequences)
	\begin{align}\label{urca}
	({\bf 1}_{ B_R }\tilde{\cK})\ast (h_n\nabla m)\to 
	({\bf 1}_{ B_R }\tilde{\cK})\ast (h\nabla m)
	\quad\mbox{  strongly in $\mathcal{H}_0^k$.}
	\end{align}
	On the other hand, for $\zeta\in (0,1)$,
	\begin{align*}
	|( ({\bf 1}_{\R^3\backslash B_R }\tilde{\cK})\ast (h_n\nabla m))(v)|\leq R^{-\zeta}
	|( |\cdot|^{\zeta-1}\ast |h_n\nabla m| )(v)|,\qquad v\in\R^3,
	\end{align*}
	so Hardy-Littlewood-Sobolev's inequality yields
	\begin{align*}
	\|( ({\bf 1}_{\R^3\backslash B_R }\tilde{\cK})\ast (h_n\nabla m))\|_{L^q(\R^3)}\lesssim
	R^{-\zeta}\|h_n\nabla m\|_{L^p(\R^3)},\qquad 
	\frac{1}{p} + \frac{1-\zeta}{3} = 1 + \frac{1}{q},\quad 1<p\leq 2.
	\end{align*}
	This means that
	\begin{align*}
	\lim_{R\to \infty}\sup_{n\in\N}\|({\bf 1}_{\R^3\backslash B_R }\tilde{\cK})\ast (h_n\nabla m)\|_{L^q(\R^3)} = 0.
	\end{align*}
	Putting the above relation and \eqref{urca} together yields the strong convergence of $\tilde{\cK}\ast (h_n\nabla m)$ in $\mathcal{H}_0^k$.
	Finally, $\int_{\R^3}m h_n dv$ is obviously relatively compact in $\mathcal{H}_0$. Thus $\cK_k : \mathcal{H}_0^k\to \mathcal{H}_0^k$ is a compact operator for every $k\geq 0$. Bound \eqref{bound.K1} is a straightforward byproduct of the previous computations and of estimates \eqref{hnp2}, \eqref{hnp2.bis}.
	This finishes the proof of the Lemma.
\end{proof}

We now want to prove the following theorem:


\begin{theorem}[Spectral gap for $L$]\label{lem.L}
There exists a constant $C_{L}>0$ such that
\begin{align}
\label{spgap.L}
-(L h, h)_{L^2(m)} &\geq C_{L} \left(
\int_{\R^3}A[m]\nabla h\cdot\nabla h \, m dv +
\|h \|_{L^2(m\langle v \rangle^{-1})}^2 \right),\quad\forall h\in D(L)\cap N(L)^\perp .
\end{align}
\end{theorem}
\begin{proof} 
From \cite{BL04} we know that for all $h\in\mathcal{H}^0$
\begin{align*}
(Lh,h)_{L^2(m)}\leq 0,
\end{align*}
and  equality holds if and only if $h\in N(L)$. This fact, Lemma \ref{lem.Lambda}, \ref{lem.new}, and \cite[Lemma 10]{daus2016hypocoercivity}
yield \eqref{spgap.L}.
\end{proof}
\begin{remark}
The constant $C_L$ appearing in the statement of Theorem ~\ref{lem.L} is not explicit. It is a consequence of 
\cite[Lemma 10]{daus2016hypocoercivity}, whose proof is non-constructive.
A similar estimate already appeared in \cite{guo2002landau,mouhot2006explicit} for the classical Landau equation.

\end{remark}
%
%
Next, we show some bounds for $A$ and $\nabla a$. Define preliminarily for $p, q \geq 1$ and $g : \R^3\to\R$ arbitrary measurable function
\begin{align*}
\mathcal{E}_{p,q}^\perp[g] =&
\int_{\R^3} |g(w)|dw + 
\left(\int_{\R^3} |w|^{p}|g(w)|^p dw \right)^{\frac{1}{p}}
+ \left(\int_{\R^3} |w|^{q}|g(w)|^q dw \right)^{\frac{1}{q}},\\
\mathcal{E}_{p,q}^\parallel[g] =&
\int_{\R^3}|w|^2 |g(w)|dw + 
\left(\int_{\R^3} |w|^{3p}|g(w)|^p dw \right)^{\frac{1}{p}}
+ \left(\int_{\R^3} |w|^{3q}|g(w)|^q dw \right)^{\frac{1}{q}},\\
\mathcal{E}_{p,q}[g] =& \mathcal{E}_{p,q}^\perp[g] + \mathcal{E}_{p,q}^\parallel[g],\\
%
\widetilde{\mathcal{E}}^\perp_{p,q}[g] =& 
\left( \int_{\R^3}|g(w)|^p dw \right)^{1/p} + 
\left( \int_{\R^3}|g(w)|^q dw \right)^{1/q} \\ \nonumber
&+
\left( \int_{\R^3}|w|^{2p}|g(w)|^{2p} dw \right)^{1/2p} + 
\left( \int_{\R^3}|w|^{2q}|g(w)|^{2q} dw \right)^{1/2q} ,\\
\widetilde{\mathcal{E}}^\parallel_{p,q}[g] =& 
\int_{\R^3} |g(w)|dw + 
\left(\int_{\R^3} |w|^{p}|g(w)|^p dw \right)^{\frac{1}{p}}
+ \left(\int_{\R^3} |w|^{q}|g(w)|^q dw \right)^{\frac{1}{q}}\\
\nonumber
&+ \left( \int_{\R^3}|w|^{4p}|g(w)|^{2p} dw \right)^{1/2p} + 
\left( \int_{\R^3}|w|^{4q}|g(w)|^{2q} dw \right)^{1/2q},\\
\widetilde{\mathcal{E}}_{p,q}[g] =& \widetilde{\mathcal{E}}_{p,q}^\perp[g] + \widetilde{\mathcal{E}}_{p,q}^\parallel[g].
\end{align*}
The following result holds.
\begin{lemma}[Bounds for $A$] \label{lem.A}
For every $p, q\in [1,\infty)$, $p<\frac{3}{2}<q$,
and every $z\in\R^3$,
\begin{align}\label{A.bound}
z\cdot A[g](v)z &\lesssim_{p,q} 
\frac{\mathcal{E}_{p,q}^\perp[g]}{|v|} |z^\perp|^2 + \frac{\mathcal{E}_{p,q}^\parallel[g]}{|v|^{3}}|z^\parallel|^2 
\lesssim_{p,q}\mathcal{E}_{p,q}[g]\, z\cdot A[m](v)z ,\\
\label{a.bound}
|\Pi(v)\nabla a[g]| &\lesssim_{p,q} \widetilde{\mathcal{E}}^\perp_{p,q}[g]|v|^{-1},\qquad
\left|\frac{v}{|v|}\cdot\nabla a[g]\right| \lesssim_{p,q} \widetilde{\mathcal{E}}^\parallel_{p,q}[g]|v|^{-2},
\end{align}
with $z^\parallel = |v|^{-2}(v\cdot z)v$, $z^\perp = z - z^{\parallel} = \Pi(v)z$, for every $g \in L^1(\R^d)$ such that $\mathcal{E}_{p,q}^\perp[g]<\infty$, $\mathcal{E}_{p,q}^\parallel[g]<\infty$,
$\widetilde{\mathcal{E}}_{p,q}^\perp[g]<\infty$, $\widetilde{\mathcal{E}}_{p,q}^\parallel[g]<\infty$.
\end{lemma}
\begin{proof}
The upper bound in \eqref{A.bound} is already known since $m$ can be estimated from below via the Maxwell-Boltzmann distribution. Therefore we only prove the lower bound.
	
We first observe that it is enough to prove the statement for $z = z^\parallel$ and $z = z^\perp$, since via Cauchy-Schwartz and Young's inequality it holds (remember that $A[g]$ is symmetric and positive definite)
\begin{align*}
z^\perp\cdot A[g]z^\parallel \leq\left( z^\perp\cdot A[g]z^\perp \right)^{1/2}\left( z^\parallel\cdot A[g]z^\parallel \right)^{1/2}\leq
\frac{1}{2}z^\perp\cdot A[g]z^\perp + 
\frac{1}{2}z^\parallel\cdot A[g]z^\parallel .
\end{align*}
Let now deal with the case $z=z^\parallel$. We start by considering $z=v$. It holds
\begin{align*}
v\cdot A[g](v)v &= \int_{\R^3}\frac{g(w)}{|v-w|}v\cdot\Pi(v-w)v\, dw\\
&= \int_{\R^3}\frac{g(w)}{|v-w|}w\cdot\Pi(v-w)w\, dw\\
&\leq \int_{\R^3}\frac{|w|^2 |g(w)|}{|v-w|}dw .
\end{align*}
%
%
Let us now consider
\begin{align*}
|v|\int_{\R^3}\frac{|w|^2 |g(w)|}{|v-w|}dw &\leq
\int_{\R^3} (|v-w| + |w|) \frac{|w|^2 |g(w)|}{|v-w|}dw \\
&\leq
\int_{\R^3}|w|^2 |g(w)|dw 
+ \int_{\R^3}\frac{|w|^3 |g(w)|}{|v-w|}dw.
\end{align*}
Since
\begin{align*}
\int_{\R^3}\frac{|w|^3 |g(w)|}{|v-w|}dw =
\int_{B_1(v)}\frac{|w|^3 |g(w)|}{|v-w|}dw +
\int_{\R^3\backslash B_1(v)}\frac{|w|^3 |g(w)|}{|v-w|}dw ,
\end{align*}
 H\uml older inequality yields
\begin{align*}
\int_{\R^3}\frac{|w|^3 |g(w)|}{|v-w|}dw\lesssim_{\ell_1,\ell_2}
\| |\cdot|^3 g \|_{3/2+\ell_1} + \| |\cdot|^3 g \|_{3/2-\ell_2},
\quad\forall \ell_1>0,~~ \forall\ell_2\in \left(0,\frac{1}{2}\right].
\end{align*}
It follows
\begin{align*}
&\frac{v}{|v|}\cdot A[g](v)\frac{v}{|v|} \lesssim_{p,q} 
\mathcal{E}_{p,q}^\parallel[g] \, |v|^{-3},\qquad
1\leq p < \frac{3}{2} < q .
\end{align*}
Let us now consider, for $z = z^\perp$, $|z|=1$,
\begin{align*}
|v|\, & z\cdot A[g](v)z = \int_{\R^3}\frac{|v| g(w)}{|v-w|}z\cdot\Pi(v-w)z\, dw  \\
&\leq \int_{\R^3}\frac{|v| |g(w)|}{|v-w|}\, dw
\leq \int_{\R^3}|g(w)|\, dw + 
\int_{\R^3}\frac{|w| |g(w)|}{|v-w|}\, dw\\
&\lesssim_{\ell_1,\ell_2} \|g\|_1 + 
\| |\cdot| g \|_{3/2+\ell_1} + \| |\cdot| g \|_{3/2-\ell_2},\quad\forall \ell_1>0,~~ \forall\ell_2\in \left(0,\frac{1}{2}\right].
\end{align*}
It follows
\begin{align*}
& z\cdot A[g](v)z \lesssim_{p,q}\mathcal{E}_{p,q}^\perp[g]\, |v|^{-1},\qquad
 1\leq p < \frac{3}{2} < q .
\end{align*}
Hence  \eqref{A.bound} holds. 

Let us now prove \eqref{a.bound}.
It holds (via Young's inequality for convolutions)
\begin{align*}
|v| |\nabla a[g](v)| &\leq \int_{\R^3}\frac{|v-w|+|w|}{|v-w|^2} |g(w)| dw\\
&= \int_{\R^3}\frac{|g(w)|}{|v-w|}dw + 
\int_{\R^3}\frac{|w||g(w)|}{|v-w|^2}dw\\
&\leq \widetilde{\mathcal{E}}_{p,q}^\perp[g],
\end{align*}
while 
\begin{align*}
|v|^2 \left|\frac{v}{|v|}\cdot\nabla a[g]\right| =&
|v| \left| \int_{\R^3}g(w)\frac{(v-w)\cdot v}{|v-w|^3}dw \right|\\
\leq & |v|\int_{\R^3}\frac{|g(w)|}{|v-w|}dw + 
|v|\int_{\R^3}\frac{|g(w)||w|}{|v-w|^2}dw\\
\leq & \int_{\R^3}|g(w)|dw + 2\int_{\R^3}\frac{|g(w)||w|}{|v-w|}dw
+ \int_{\R^3}\frac{|g(w)||w|^2}{|v-w|^2}dw\\
\leq & \widetilde{\mathcal{E}}_{p,q}^\parallel[g].
\end{align*}
This finishes the proof of the Lemma.	
\end{proof}
The next lemma deals with the nonlinear contributions $\Gamma_2$ and $\Gamma_3$. 
\begin{lemma}[Bounds for nonlinear terms]\label{lem.G.bound}
For every $p, q\geq 1$, $p < \frac{3}{2} < q$, $k\geq 0$
it holds
\begin{align}
( &\Gamma_2(h,h), h)_{L^2(m \langle v \rangle^k )} \lesssim_{p,q}
\Big[ \rho 
(\mathcal{E}_{p,q}[mh] 
+ \widetilde{\mathcal{E}}_{p,q}[mh] ) \label{G2.bound} \\ \nonumber
&+ \rho\|m^{1/2}h\|_{2}^{2/3} + \rho\|m^{1/2}h\|_{2}^{4/3}
+ \rho^{-1} \Big]\left(
\int_{\R^3}A[m]\nabla h\cdot\nabla h\,\langle v \rangle^k m dv + 
\int_{\R^3} h^2\langle v \rangle^{k-1} m dv \right) ,
\end{align}
\begin{align}
\label{G3.bound}
(\Gamma_3[h,& h,h], h)_{L^2(m\langle v \rangle^k)} \\
\nonumber
\lesssim_{p,q} &\; \rho (\mathcal{E}_{p,q}[mh] + \widetilde{\mathcal{E}}_{p,q}[mh])\left(
\int_{\R^3}A[m]\nabla h\cdot\nabla h\,\langle v \rangle^k m dv + 
\int_{\R^3} h^2\langle v \rangle^{k-1} m dv \right)\\
\nonumber
& + \rho^{-1}\int_{\R^3}A[m]\nabla h\cdot\nabla h\langle v \rangle^k m dv ,
\end{align}
for every $\rho>0$.
\end{lemma}
\begin{proof}
Let us first consider the contribution of the quadratic terms.
\begin{align*}
(\Gamma_2(h,h),h)_{L^2(m\langle v \rangle^k)} =& 
-\int_{\R^3}\langle v \rangle^k\nabla h\cdot
\left( A[(1-2\cM)mh]\nabla(mh) - A[m^2 h^2]\nabla\cM\right.\\
&\left.\qquad - (1-2\cM)mh\nabla a[mh] + m^2 h^2\nabla a[\cM]
\right)dv\\
&-\int_{\R^3}h \nabla\langle v \rangle^k\cdot
\left( A[(1-2\cM)mh]\nabla(mh) - A[m^2 h^2]\nabla\cM\right.\\
&\left.\qquad - (1-2\cM)mh\nabla a[mh] + m^2 h^2\nabla a[\cM]
\right)dv,
\end{align*}
that can be rewritten as 
\begin{align*}
\langle \Gamma_2(h,h), & h\rangle_{L^2(m)} = \sum_{j=1}^5 I_j
+ \sum_{j=1}^5 I'_j,\\
I_1 :=& -\int_{\R^3}A[(1-2\cM)mh]\nabla h\cdot\nabla h\,\langle v \rangle^k mdv, \\
I_2 :=& -\int_{\R^3}A[(1-2\cM)mh]\nabla h\cdot\nabla m\,\langle v \rangle^k h dv,\\
I_3 :=& +\int_{\R^3}\nabla h\cdot A[m^2 h^2]\nabla\cM\,\langle v \rangle^k dv,\\
I_4 :=& +\int_{\R^3}\nabla h\cdot (1-2\cM)mh\nabla a[mh]\,\langle v \rangle^k dv,\\
I_5 :=& -\int_{\R^3}\nabla h\cdot m^2 h^2\nabla a[\cM]\, \langle v \rangle^k dv ,\\
I'_1 :=& -\int_{\R^3}A[(1-2\cM)mh]\nabla h\cdot\nabla\langle v \rangle^k\, h mdv, \\
I'_2 :=& -\int_{\R^3}A[(1-2\cM)mh]\nabla \langle v \rangle^k\cdot\nabla m\,
h^2 dv,\\
I'_3 :=& +\int_{\R^3}\nabla \langle v \rangle^k\cdot A[m^2 h^2]\nabla\cM\,h dv,\\
I'_4 :=& +\int_{\R^3}\nabla \langle v \rangle^k\cdot (1-2\cM)mh\nabla a[mh]\,h dv,\\
I'_5 :=& -\int_{\R^3}\nabla \langle v \rangle^k\cdot m^2 h^3\nabla a[\cM]\, dv ,
\end{align*}
For every $1\leq p < \frac{3}{2} < q$, thanks to \eqref{A.bound}, we get 
\begin{align*}
I_1 \lesssim_{p,q}\mathcal{E}_{p,q}[mh]\int_{\R^3}A[m]\nabla h\cdot\nabla h\, m \langle v \rangle^k dv ,
\end{align*} 
while Cauchy-Schwartz and Young's inequality lead to
\begin{align*}
I_2 \lesssim\int_{\R^3}A[|mh|]\nabla h\cdot\nabla h\, \langle v \rangle^k m dv + 
\int_{\R^3}A[|mh|]\nabla\log{m}\cdot\nabla\log{m}\, h^2\langle v \rangle^k m dv.
\end{align*}
However, it is easy to see (via direct computation) that
$$
\nabla\log m(v) = b\frac{1 - a e^{-b|u-v|^2/2}}{1 + a e^{-b|u-v|^2/2}}(u-v),
$$
so, using \eqref{A.bound}, we obtain
\begin{align*}
A[|mh|]\nabla\log{m}\cdot\nabla\log{m}\lesssim_{p,q}\mathcal{E}_{p,q}[mh]\langle v \rangle^{-1},
\end{align*}
which implies 
\begin{align*}
I_2 \lesssim_{p,q}\int_{\R^3}A[|mh|]\nabla h\cdot\nabla h\,\langle v \rangle^k m dv + 
\mathcal{E}_{p,q}[mh]\int_{\R^3} h^2 \langle v \rangle^{k-1}m dv.
\end{align*}
Applying \eqref{A.bound} once again leads to
\begin{align*}
I_2 \lesssim_{p,q}\mathcal{E}_{p,q}[mh]\left(
\int_{\R^3}A[m]\nabla h\cdot\nabla h\,\langle v \rangle^k m dv + 
\int_{\R^3} h^2 \langle v \rangle^{k-1}m dv \right).
\end{align*}
Let us now consider, for arbitrary $\rho>0$,
\begin{align*}
I_3 &\leq \rho\int_{\R^3} A[m^2 h^2]\nabla h\cdot\nabla h\,\langle v \rangle^k m dv + \rho^{-1}
\int_{\R^3} A[m^2 h^2]\nabla\log\left( \frac{\cM}{1-\cM} \right)\cdot \nabla\log\left( \frac{\cM}{1-\cM} \right)\langle v \rangle^k m dv\\
&=\rho\int_{\R^3} A[m^2 h^2]\nabla h\cdot\nabla h\,\langle v \rangle^k m dv + 
\rho^{-1}b^2\int_{\R^3} A[m^2 h^2](u-v)\cdot (u-v)\langle v \rangle^k m dv .
\end{align*}
It is quite easy to see that 
\begin{align*}
\int_{\R^3} A[m^2 h^2](u-v)\cdot (u-v)\langle v \rangle^k m dv &\leq
\int_{\R^3} m^2(w)h^2(w)\int_{\R^3}\frac{m(v)|u-v|^2}{|v-w|}\langle v \rangle^k dv \, dw\\
&\lesssim \int_{\R^3}m^2 h^2 dv \lesssim \int_{\R^3}m h^2\langle v \rangle^{k-1} dv ,
\end{align*}
while, on the other hand,
\begin{align*}
\int_{\R^3} A[m^2 h^2]\nabla h\cdot\nabla h\,\langle v \rangle^k m dv\lesssim_{p,q}\mathcal{E}_{p,q}[m^2 h^2]\int_{\R^3}A[m]\nabla h\cdot\nabla h \,\langle v \rangle^k m dv .
\end{align*}
Since $|mh| = |f-\cM|\leq 1$, it follows
\begin{align*}
I_3\lesssim_{p,q} \rho^{-1}\int_{\R^3}m h^2\langle v \rangle^{k-1} dv + \rho \mathcal{E}_{p,q}[mh]\int_{\R^3}A[m]\nabla h\cdot\nabla h \,\langle v \rangle^k m dv .
\end{align*}
Let us now deal with $I_4$. Young's inequality yields
\begin{align*}
I_4 =& \int_{\R^3}\nabla h\cdot (1-2\cM)mh\nabla a[mh]\,\langle v \rangle^k dv\\
=& \int_{\R^3}\Pi(v)\nabla h\cdot (1-2\cM)mh \Pi(v)\nabla a[mh]\,\langle v \rangle^k dv \\
&+ 
\int_{\R^3}\frac{v\otimes v}{|v|^2}\nabla h\cdot (1-2\cM)mh\frac{v\otimes v}{|v|^2}\nabla a[mh]\,\langle v \rangle^k dv\\
\lesssim & \frac{1}{\rho}\int_{\R^3}|\Pi(v)\nabla h|^2 \brak{v}^{k-1}m dv + 
\rho\int_{\R^3}h^2 |\Pi(v)\nabla a[mh]|^2 \brak{v}^{k+1} m dv\\
& + \frac{1}{\rho}\int_{\R^3}\left|\frac{v\otimes v}{|v|^2}\nabla h\right|^2 \brak{v}^{k-3}m dv + 
\rho\int_{\R^3}h^2 \left|\frac{v\otimes v}{|v|^2}\nabla a[mh]\right|^2 \brak{v}^{k+3} m dv.
\end{align*}
From \eqref{A.bound}, \eqref{a.bound} it follows
\begin{align*}
I_4 \lesssim & \frac{1}{\rho}\int_{\R^3}A[m]\nabla h\cdot\nabla h \langle v \rangle^k m dv + \rho \widetilde{\mathcal{E}}_{p,q}[mh]\int_{\R^3} m h^2\langle v \rangle^{k-1} dv.
\end{align*}
Finally, let us consider, for a generic $0<\eta<1/3$,
\begin{align*}
I_5 =& -\int_{\R^3}\nabla h\cdot m^2 h^2\nabla a[\cM]\,\langle v \rangle^k dv \\
\lesssim & \int_{\R^3} m^{1/2+\eta}|\nabla h| \, m^{3/2-\eta}|h|^2 \,\langle v \rangle^k dv\\
\lesssim &
 \int_{\R^3} m^{1/2+\eta}|\nabla h| \, m^{7/6-\eta}|h|^{5/3} \,\langle v \rangle^k dv,
\end{align*}
where the last inequality holds because $|mh|^{1/3} = |f-\cM|^{1/3}\leq 1$. It follows via Cauchy-Schwartz inequality
\begin{align*}
I_5 \lesssim & \|\langle v \rangle^k m^{1/2+\eta}|\nabla h|\|_2\|m^{7/6-\eta}|h|^{5/3}\|_2 \lesssim_\eta
\|\brak{v}^{(k-3)/2}m^{1/2}|\nabla h|\|_2
\|m^{7/10-3\eta/5}|h|\|_{10/3}^{5/3}.
\end{align*}
Gagliardo-Nirenberg inequality leads to
\begin{align*}
I_5 \lesssim_\eta &\|\brak{v}^{(k-3)/2}m^{1/2}|\nabla h|\|_2
\|m^{7/10-3\eta/5}h\|_{2}^{2/3}\|\nabla(m^{7/10-3\eta/5}h)\|_2\\
\lesssim_\eta & \|\brak{v}^{(k-3)/2}m^{1/2}|\nabla h|\|_2
\|m^{7/10-3\eta/5}h\|_{2}^{2/3}\left(
\|m^{7/10-3\eta/5}\nabla h\|_2 + 
\|h\nabla(m^{7/10-3\eta/5})\|_2\right)\\
\lesssim_\eta & \|\brak{v}^{(k-3)/2}m^{1/2}|\nabla h|\|_2^2\|m^{7/10-3\eta/5}h\|_{2}^{2/3} \\
& +  \|\brak{v}^{(k-3)/2}m^{1/2}|\nabla h|\|_2
\|m^{7/10-3\eta/5}h\|_{2}^{2/3}\|h\, m^{7/10-3\eta/5}\nabla\log m \|_2 .
\end{align*}
Choosing $\eta = 1/6$ yields
\begin{align*}
I_5 \lesssim & \|\brak{v}^{(k-3)/2}m^{1/2}|\nabla h|\|_2^2
\|m^{1/2}h\|_{2}^{2/3} \\
& +  \|\brak{v}^{(k-3)/2}m^{1/2}|\nabla h|\|_2
\|\langle v \rangle^{(k-1)/2}m^{1/2}h\|_{2}
\|m^{1/2}h\|_{2}^{2/3}\\
\lesssim & \rho ( \|m^{1/2}h\|_{2}^{2/3} + \|m^{1/2}h\|_{2}^{4/3})
\|\brak{v}^{(k-3)/2}m^{1/2}|\nabla h|\|_2^2 + 
\rho^{-1} \|\langle v \rangle^{(k-1)/2}m^{1/2}h\|_{2}^2 .
\end{align*}
From \eqref{A.bound} we conclude
\begin{align*}
I_5 \lesssim &\rho ( \|m^{1/2}h\|_{2}^{2/3} + \|m^{1/2}h\|_{2}^{4/3}) \int_{\R^3}A[m]\nabla h\cdot\nabla h\,\langle v \rangle^k m dv
+ \rho^{-1}\int_{\R^3}h^2 \langle v \rangle^{k-1}m dv .
\end{align*}
Since $|\nabla \langle v \rangle^k |\lesssim \langle v \rangle^{k-1}$,
the terms $I'_1,\ldots, I'_5$ can be estimated in a similar way as the terms
$I_1,\ldots, I_5$. Therefore we deduce that \eqref{G2.bound} holds.

Next, we deal with the contributions from the cubic terms:
\begin{align*}
\langle \Gamma_3[h,h,h], h\rangle_{L^2(m)} =&
\int_{\R^3} \nabla h\cdot\left(
A[m^2 h^2]\nabla (mh) - m^2 h^2\nabla a[m h]\right)\langle v \rangle^k dv \\
& +\int_{\R^3} \nabla\langle v \rangle^k \cdot\left(
A[m^2 h^2]\nabla (mh) - m^2 h^2\nabla a[m h]\right) h dv\\
=& I_6 + I_7 + I_8 + I'_6 + I'_7 + I'_8,\\
I_6 :=& \int_{\R^3}\nabla h\cdot A[m^2 h^2]\nabla h\,\langle v \rangle^k m dv,\\
I_7 :=& \int_{\R^3}\nabla h\cdot A[m^2 h^2]\nabla m\,\langle v \rangle^k h dv,\\
I_8 :=& -\int_{\R^3}\nabla h \cdot m^2 h^2\nabla a[m h] \langle v \rangle^k dv,\\
I'_6 :=& \int_{\R^3}\nabla \langle v \rangle^k\cdot A[m^2 h^2]\nabla h\,h m dv,\\
I'_7 :=& \int_{\R^3}\nabla \langle v \rangle^k\cdot A[m^2 h^2]\nabla m\,h^2 dv,\\
I'_8 :=& -\int_{\R^3}\nabla \langle v \rangle^k \cdot m^2 h^3\nabla a[m h] dv.
\end{align*}
From \eqref{A.bound} and relation $|mh|\leq 1$ it follows
\begin{align*}
I_6 \lesssim_{p,q} &\mathcal{E}_{p,q}[m h]\int_{\R^3}A[m]\nabla h\cdot \nabla h \,\langle v \rangle^k m dv.
\end{align*}
The term $I_7$ can be estimated like $I_2$ to obtain
\begin{align*}
I_7 \lesssim_{p,q}\mathcal{E}_{p,q}[mh]\left(
\int_{\R^3}A[m]\nabla h\cdot\nabla h\,\langle v \rangle^k m dv + 
\int_{\R^3} h^2 \langle v \rangle^{k-1}m dv \right).
\end{align*}
The term $I_8$ can be estimated like $I_4$ to obtain
\begin{align*}
I_8 \lesssim & \frac{1}{\rho}\int_{\R^3}A[m]\nabla h\cdot\nabla h \langle v \rangle^k m dv + \rho \widetilde{\mathcal{E}}_{p,q}[mh]\int_{\R^3} m^3 h^4\langle v \rangle^{k-1} dv,
\end{align*}
but, given that $m^2 h^2\leq 1$, it follows
\begin{align*}
I_8 \lesssim & \frac{1}{\rho}\int_{\R^3}A[m]\nabla h\cdot\nabla h \langle v \rangle^k m dv + \rho \widetilde{\mathcal{E}}_{p,q}[mh]\int_{\R^3} m h^2\langle v \rangle^{k-1} dv.
\end{align*}
Finally, since  $|\nabla \langle v \rangle^k |\lesssim \langle v \rangle^{k-1}$,
the terms $I'_6,\ldots, I'_8$ can be estimated in a similar way as the terms
$I_6,\ldots, I_8$.
Therefore we deduce that \eqref{G3.bound} holds. This finishes the proof of the Lemma.
\end{proof}
%
%
We are now ready to prove the conditional algebraic convergence result, thereby concluding the proof of Thr.~\ref{thm_longtime}.
\begin{lemma}[Algebraic rate of convergence for initial data close to equilibrium]\label{lem.rate}
There exists a constant $\ell>0$ such that, if $\int_{\R^3}(f_{in} - \cM)^2 m^{-1}dv < \ell$, and if 
$\int_{\R^3}(f_{in} - \cM)^2 \langle v \rangle^N m^{-1}dv < \infty$
for some $N\geq 1$, then 
\begin{align*}
\int_{\R^3}(f(t) - \cM)^2 m^{-1}dv \lesssim (1+t)^{-N},\qquad t>0.
\end{align*}
\end{lemma}
\begin{proof}
	From \eqref{LFD}, \eqref{def.h} it follows that
the perturbation $h = (f-\cM)/m$ satisfies the equation
\begin{align}\label{eq.h}
\pa_t h = L h + \Gamma_2[h,h] + \Gamma_3[h,h,h].
\end{align}
Testing the above equation against $h$ in the sense of $L^2(m)$ yields
\begin{align*}
\frac{d}{dt}\frac{1}{2}\int_{\R^3}h^2 m dv = 
\langle L h, h\rangle_{L^2(m)} + \langle\Gamma_2[h,h],h\rangle_{L^2(m)} + \langle\Gamma_3[h,h,h],h\rangle_{L^2(m)}.
\end{align*}
From \eqref{spgap.L}, \eqref{G2.bound}, \eqref{G3.bound} it follows
that a suitable constant $C(p,q)>0$ exists such that
\begin{align*}
\frac{d}{dt} & \frac{1}{2}\int_{\R^3}h^2 m dv \leq 
\Big[ \rho C(p,q)
(\mathcal{E}_{p,q}[mh] 
+ \widetilde{\mathcal{E}}_{p,q}[mh] ) \\ \nonumber
&+ \rho\|m^{1/2}h\|_{2}^{2/3} + \rho\|m^{1/2}h\|_{2}^{4/3}
+ \rho^{-1} - C_L \Big]\left(
\int_{\R^3}A[m]\nabla h\cdot\nabla h\, m dv + 
\int_{\R^3} h^2 m\langle v \rangle^{-1} dv \right).
\end{align*}
We are now going to prove that 
\begin{equation}
\label{E.bound}
\exists\alpha_1>0:\quad
\mathcal{E}_{p,q}[mh]
+ \widetilde{\mathcal{E}}_{p,q}[mh]\lesssim_{p,q} 
\rho^{-2}+\rho^{\alpha_1}\|m^{1/2}h\|_2.
\end{equation}
Indeed, the left-hand side of \eqref{E.bound} is a sum of terms having the form
\begin{align*}
J_{k,s} = \left( \int_{\R^3}|v|^{k}|mh|^{s} dv\right)^{1/s},\quad k\geq 0,~~ s\geq 1.
\end{align*}
If $s\geq 2$ then from the property $|mh|\leq 1$ and the fact that $|v|^k\sqrt{m(v)}$ is bounded in $\R^3$ for every $k\geq 0$ it follows immediately that
\begin{align*}
J_{k,s}\leq \left( \int_{\R^3}|v|^{k}|mh|^{2} dv\right)^{1/s}\lesssim_k
\left(\int_{\R^3}m h^{2} dv\right)^{1/s} ,
\end{align*}
so via Young's inequality
\begin{align*}
J_{k,s}\lesssim_{s,k} \rho^{-2} + \rho^{s-2}\|m^{1/2}h\|_2.
\end{align*}
If $1\leq s < 2$, it suffices to notice that 
\begin{align*}
J_{k,s} = \| m^{1/2}h \|_{L^s(\R^3, |v|^k m^{s/2}(v)dv)}.
\end{align*}
Since $|v|^k m^{s/2}\in L^1\cap L^\infty(\R^3)$, Jensen's inequality yields
\begin{align*}
J_{k,s}\lesssim_{k,s}\| m^{1/2}h \|_{L^2(\R^3, |v|^k m^{s/2}(v)dv)}
\lesssim_{k,s}\| m^{1/2}h \|_{2}.
\end{align*}
Therefore \eqref{E.bound} holds. We therefore conclude that, for some suitable constant $C'(p,q)>0$ and $\alpha>1$,
\begin{align*}
\frac{d}{dt} & \int_{\R^3}h^2 m dv \leq
\Big[ C'(p,q)( \rho^{\alpha} \|m^{1/2}h\|_{2}^{2}
+ \rho^{-1}) - C_L \Big]\left(
\int_{\R^3}A[m]\nabla h\cdot\nabla h\, m dv + 
\int_{\R^3} h^2 m\langle v \rangle^{-1} dv \right),
\end{align*}
for every $\rho>1$. We point out that
\begin{align*}
C'(p,q)( \rho^{\alpha} \|m^{1/2}h\|_{2}^{2}
+ \rho^{-1}) - C_L = C'(p,q)\rho^\alpha\left(
\|m^{1/2}h\|_{2}^{2} - \tilde\ell(\rho)\right),\quad
\tilde\ell(\rho) = \left(\frac{C_L}{C'(p,q)}-\rho^{-1}\right)\rho^{-\alpha}.
\end{align*}
The maximum of $\tilde\ell(\rho)$ is achieved for $\rho = \frac{1+\alpha}{\alpha C_L}C'(p,q)$. Choosing $\rho$ in this way yields
\begin{align*}
\frac{d}{dt} & \int_{\R^3}h^2 m dv \leq C''(p,q)
\Big[ \|m^{1/2}h\|_{2}^{2}
 - \ell \Big]\left(
\int_{\R^3}A[m]\nabla h\cdot\nabla h\, m dv + 
\int_{\R^3} h^2 m\langle v \rangle^{-1} dv \right),
\end{align*}
for $C''(p,q) = \left(\frac{1+\alpha}{\alpha C_L}\right)^\alpha C'(p,q)^{1+\alpha}$ and
$$\ell := \frac{C_L}{(1+\alpha)C'(p,q)}\left(\frac{1+\alpha}{\alpha C_L}C'(p,q)\right)^{-\alpha} = 
\left( \frac{C_L}{C'(p,q)} \right)^{1+\alpha}\frac{\alpha^\alpha}{(1+\alpha)^{1+\alpha}}.
$$
Since $\|m^{1/2}h(\cdot, 0)\|_{2}^{2}
- \ell < 0$ by assumption on the initial data, we deduce that $\|m^{1/2}h(\cdot, t)\|_{2}\leq \|m^{1/2}h(\cdot, 0)\|_{2}$ for all $t>0$. 
It follows that, for some $\lambda > 0$,
\begin{align}\label{almost}
 \frac{d}{dt} & \int_{\R^3}h^2 m dv \leq -\lambda\left(
 \int_{\R^3}A[m]\nabla h\cdot\nabla h\, m dv + 
 \int_{\R^3} h^2 m\langle v \rangle^{-1} dv \right).
\end{align}
Integrating \eqref{almost} in time yields
\begin{align}\label{almost.int}
\sup_{t>0}\int_{\R^3}h^2 m dv +
\lambda\int_0^\infty \int_{\R^3} h^2 m \langle v \rangle^{-1} dv dt
\leq \int_{\R^3}h(\cdot,0)^2 m dv .
\end{align}
We will now show that $\sup_{t>0}\int_{\R^3}h^2 m \langle v \rangle^N dv < \infty$. We proceed iteratively, proving that 
\begin{align}\label{end.goal}
\sup_{t>0}\int_{\R^3}h^2 m \langle v \rangle^j dv
+ \int_0^\infty \int_{\R^3} h^2 m \langle v \rangle^{j-1} dv dt
 < \infty ,
\end{align}
for $j=0,\ldots,\lfloor N\rfloor$.
We argue by induction on $j$. Estimate \eqref{almost.int} and the assumption on the initial datum imply that \eqref{end.goal} holds for $j=0$. Let us now assume that \eqref{end.goal} holds for $j=0,\ldots,k-1$, $1\leq k\leq \lfloor N\rfloor$ generic.
By testing \eqref{eq.h} against $h$ in the sense of $L^2(m\langle v\rangle^k)$, exploiting Lemma \ref{lem.Lambda} and bound \eqref{bound.K1} and proceeding like in the proof of \eqref{almost} one finds
\begin{align}\label{almost.2}
\frac{d}{dt} \int_{\R^3}h^2 m \langle v \rangle^k dv \leq & -\lambda_k\left(
\int_{\R^3}A[m]\nabla h\cdot\nabla h\,\langle v \rangle^k m dv + 
\int_{\R^3} h^2 m\langle v \rangle^{k-1} dv \right)\\
\nonumber
&+\mu_k \int_{\R^3}h^2 m \langle v \rangle^{k-2}dv,
\end{align}
for some $\lambda_k$, $\mu_k>0$. By integrating \eqref{almost.2} in time we get
\begin{align}\label{almost.3}
\sup_{t>0}\int_{\R^3}h^2 m \langle v \rangle^k dv 
&+ \lambda_k\int_0^\infty \int_{\R^3} h^2 m\langle v \rangle^{k-1} dv dt\\
\nonumber
&\leq \mu_k\int_0^\infty \int_{\R^3}h^2 m \langle v \rangle^{k-2}dv dt 
+ \int_{\R^3}h(\cdot,0)^2 m \langle v \rangle^k dv.
\end{align}
From the assumption that $\int_{\R^3}h(\cdot,0)^2 m \langle v \rangle^k dv<\infty$ for $k\leq N$ as well as the inductive hypothesis it follows
that the right-hand side of \eqref{almost.3} is finite,
meaning that \eqref{end.goal} holds for $j=k$. Via the induction principle we deduce that \eqref{end.goal} holds for $j=0,\ldots,\lfloor N\rfloor$. 
Choosing $k=N$ in \eqref{almost.3} and exploiting \eqref{end.goal} for $j=\lfloor N \rfloor$ yields \eqref{end.goal} for $k=N$.
In particular
$$\sup_{t>0}\int_{\R^3}h^2 m \langle v \rangle^N dv < \infty . $$
Therefore via H\"older's inequality
\begin{align*}
\int_{\R^3}h^2 m dv \leq \left( \int_{\R^3}h^2 m\langle v \rangle^{-1} dv \right)^{ \frac{N}{N+1} }\left( \int_{\R^3}h^2 m\langle v \rangle^N dv \right)^{ \frac{1}{N+1} }
\lesssim \left( \int_{\R^3}h^2 m\langle v \rangle^{-1} dv \right)^{\frac{N}{N+1}},
\end{align*}
so from \eqref{almost} it follows
\begin{align*}
\frac{d}{dt} & \int_{\R^3}h^2 m dv \leq -\lambda\left(
\int_{\R^3}A[m]\nabla h\cdot\nabla h\, m dv + 
\left(\int_{\R^3} h^2 m dv\right)^{ \frac{N+1}{N} } \right).
\end{align*}
This (via Gronwall's inequality) finishes the proof of the lemma, and of Theorem \ref{thm_longtime}.
\end{proof}

\bibliography{landaurefs}
\bibliographystyle{plain}

\end{document}